\pdfoutput=1
\documentclass[a4paper,10pt]{amsart}
\usepackage[utf8]{inputenc}
\usepackage[english]{babel}
\usepackage[autostyle=true]{csquotes}
\usepackage[T1]{fontenc}
\usepackage{lmodern}
\usepackage{microtype}

\usepackage{amsmath,amssymb,amsthm}
\usepackage{thmtools,thm-restate}
\usepackage{slashed}
\usepackage{mathrsfs}

\usepackage[pdfusetitle,hyperfootnotes=false]{hyperref}
\hypersetup{
  colorlinks=true,
  linkcolor={blue!60!black},
  citecolor={green!40!black},
  urlcolor={blue},
}

\allowdisplaybreaks[1] 
\usepackage{tikz-cd}

\usepackage[citestyle=alphabetic,bibstyle=my-alphabetic,backend=biber,url=false,doi=true,isbn=false,maxbibnames=99,firstinits=true]{biblatex}
\renewcommand*{\bibfont}{\small}
\AtEveryBibitem{\clearfield{month}}
\AtEveryBibitem{\clearfield{day}}
\AtEveryBibitem{\clearfield{series}}
\AtEveryBibitem{\clearfield{language}}
\AtEveryBibitem{
  \ifentrytype{book}
    {\clearfield{pages}
     \clearfield{volume}}
    {}
}
\bibliography{literature.bib}

\usepackage{paralist}

\declaretheorem[name=Theorem, numberwithin=section]{thm}
\declaretheorem[name=Theorem, numbered=no]{thm*}
\declaretheorem[name=Lemma,numberlike=thm]{lem}
\declaretheorem[name=Corollary,numberlike=thm]{cor}
\declaretheorem[name=Proposition,numberlike=thm]{prop}
\declaretheorem[name=Definition,numberlike=thm, style=definition]{defi}
\declaretheorem[name=Example, numberlike=thm, style=remark]{ex}
\declaretheorem[name=Remark, numberlike=thm, style=remark]{rem}

\usepackage[noabbrev]{cleveref} %Important to load after numberwithin!
\crefname{figure}{Figure}{Figure}
\crefname{table}{Table}{Table}
\crefname{thm}{Theorem}{Theorems}
\crefname{lem}{Lemma}{Lemmas}
\crefname{defi}{Definition}{Definitions}
\crefname{cor}{Corollary}{Corollaries}
\crefname{prop}{Proposition}{Propositions}
\crefname{ex}{Example}{Examples}
\crefname{rem}{Remark}{Remarks}
\crefname{section}{Section}{Sections}
\crefname{chapter}{Chapter}{Chapters}
\crefname{appendix}{Appendix}{Appendices}
\crefdefaultlabelformat{\textup{#2#1#3}}
\creflabelformat{enumi}{\textup{(#2#1#3)}}

\numberwithin{equation}{section}

\usepackage[final]{fixme}
\fxsetup{layout={margin}}

\newcommand{\parensup}[1]{\textup{(}#1\textup{)}}

\newcommand{\tens}{\otimes}
\newcommand{\tensGr}{\widehat{\otimes}}

\newcommand{\Lp}{\mathrm{L}}

\newcommand{\field}[1]{\mathbb{#1}}
\newcommand{\C}{\field{C}}
\newcommand{\Z}{\field{Z}}
\newcommand{\N}{\field{N}}
\newcommand{\Q}{\field{Q}}
\newcommand{\R}{\field{R}}
\newcommand{\Rext}{[-\infty,\infty]}
\newcommand{\Rgeq}{\R_{\geq 0}}
\newcommand{\Rleq}{\R_{\leq 0}}
\newcommand{\Rgr}{\R_{> 0}}

\newcommand{\Cstar}{\mathrm{C}^*}
\newcommand{\cont}{\mathcal{C}}
\newcommand{\contZ}{\cont_0}
\newcommand{\proj}{\operatorname{pr}}
\newcommand{\Mat}{\mathrm{M}}

\newcommand{\boundary}{\partial}
\newcommand{\bd}{\boundary}
\newcommand{\bdMV}{\bd_{\mathrm{MV}}}
\newcommand{\id}{\operatorname{id}}
\newcommand{\iu}{\mathrm{i}}
\newcommand{\eu}{\mathrm{e}}
\newcommand{\ev}{\operatorname{ev}}

\newcommand{\ran}{\operatorname{ran}}
\newcommand{\supp}{\operatorname{supp}}

\newcommand{\D}{\mathrm{d}}

\newcommand{\charFun}{\mathbf{1}}

\newcommand{\iso}{\cong}

\newcommand{\unitary}{\mathscr{U}}
\newcommand{\bndOps}{\mathbb{B}}
\newcommand{\cptOps}{\mathbb{K}}

\newcommand{\hilbert}{\mathcal{H}}
\newcommand{\hilbertGr}{\mathcal{H}}

\newcommand{\SO}{\mathrm{SO}}
\newcommand{\Spin}{\mathrm{Spin}}

\newcommand{\Hom}{\operatorname{Hom}}

\newcommand{\act}{\circlearrowright}

\newcommand{\scalCurv}{\operatorname{scal}}

\newcommand{\Ind}{\operatorname{Ind}}
\newcommand{\IndRel}{\operatorname{Ind}_{\mathrm{rel}}}
\newcommand{\IndL}{\Ind_{\mathrm{L}}}
\newcommand{\IndLP}[1]{\Ind_{\mathrm{L,#1}}}

\newcommand{\IndLTildeP}[1]{\widetilde{\Ind}_{\mathrm{L, #1}}}
\newcommand{\secInv}{\rho}

\newcommand{\domain}{\operatorname{dom}}

\renewcommand{\leq}{\leqslant}
\renewcommand{\geq}{\geqslant}

\newcommand{\roeAlg}{\mathrm{C}^*}
\newcommand{\roeAlgLoc}{\mathrm{C}^*_L}

\newcommand{\roeAlgLocZ}{\mathrm{C}^*_{L,0}}
\newcommand{\roeAlgLocP}[1]{\mathrm{C}^*_{L,#1}}
\newcommand{\structureAlg}{\mathrm{D}^*}
\newcommand{\structureAlgLoc}{\mathrm{D}^*_L}

\newcommand{\structureGp}{\mathrm{S}}

\newcommand{\cliffAlg}{\mathrm{Cl}}
\newcommand{\cliffAlgDual}{\mathrm{Cl}^\ast}
\newcommand{\cliffAlgC}{\mathbb{C}\mathrm{l}}
\newcommand{\RReal}{\mathbf{C}}

\newcommand{\secProd}{\boxtimes}
\newcommand{\KTh}{\mathrm{K}}
\newcommand{\KOTh}{\KTh}
\newcommand{\KPh}{\ast}
\newcommand{\KThProd}{\times}
\newcommand{\KThGrProd}{\KThProd}
\newcommand{\KOThGrProd}{\KThGrProd}
\newcommand{\KOThProd}{\KThProd}
\newcommand{\KThGr}{\widehat{\KTh}}
\newcommand{\KOThGr}{\widehat{\KOTh}}
\newcommand{\ETh}{\mathrm{E}}

\newcommand{\KOSpec}{\underline{\KOTh}}

\newcommand{\smashProd}{\wedge}

\newcommand{\tensMap}{\mathrm{t}}

\newcommand{\principalBdl}{\mathrm{P}}
\newcommand{\spinorBdl}{\slashed{S}}

\newcommand{\contZGr}{\mathcal{S}}
\newcommand{\suspCstar}{\Sigma}
\newcommand{\pathsCstarMV}{\Omega}

\newcommand{\comult}{\triangle}

\newcommand{\tangentBdl}{\mathrm{T}}

\newcommand{\clnDiracOp}{\slashed{\mathfrak{D}}}
\newcommand{\diracOp}{\slashed{D}}
\newcommand{\clnSpinorBdl}{\slashed{\mathfrak{S}}}
\newcommand{\clnXmodule}{\mathfrak{H}}
\newcommand{\Xmodule}{\mathcal{H}}

\newcommand{\Ad}{\operatorname{Ad}}

\newcommand{\End}{\operatorname{End}}

\newcommand{\asympMorphism}{\dashrightarrow}

\newcommand{\nbh}{\mathcal{N}}
\newcommand{\idealRel}{\unlhd}

\newcommand{\Efree}{\mathrm{E}}
\newcommand{\Bfree}{\mathrm{B}}
\DeclareMathOperator*{\colim}{colim}
\newcommand{\CstarRed}{\Cstar_{\mathrm{r}}}
\DeclareMathOperator*{\propag}{prop}

\newcommand{\conjg}[1]{\bar{#1}}
\newcommand{\longconjg}[1]{\overline{#1}}

\newcommand{\myEmail}{\href{mailto:math@rzeidler.eu}{\nolinkurl{math@rzeidler.eu}}}

\title[PSC and product formulas for secondary index invariants]{Positive scalar curvature and product formulas for secondary index invariants}
\author{Rudolf Zeidler}
\subjclass[2010]{46L80 (Primary) 53C20 (Secondary)}
\thanks{The author is supported by the German Research Foundation (DFG) through the Research Training Group 1493 \enquote{Mathematical structures in modern quantum physics.}}
\address{Mathematisches Institut, Georg--August--Universität Göttingen, Bunsenstraße 3-5, 37073 Göttingen, Germany}
\email{\myEmail}
\date{}

\begin{document}

\begin{abstract}
We introduce partial secondary invariants associated to complete Riemannian metrics which have uniformly positive scalar curvature outside a prescribed subset on a spin manifold.
These can be used to distinguish such Riemannian metrics up to concordance relative to the prescribed subset.
We exhibit a general external product formula for partial secondary invariants, from which we deduce product formulas for the higher $\rho$-invariant of a metric with uniformly positive scalar curvature as well as for the higher relative index of two metrics with uniformly positive scalar curvature.

Our methods yield a new conceptual proof of the secondary partitioned manifold index theorem and a refined version of the delocalized APS-index theorem of Piazza--Schick for the spinor Dirac operator in all dimensions.
We establish a partitioned manifold index theorem for the higher relative index.
We also show that secondary invariants are stable with respect to direct products with aspherical manifolds that have fundamental groups of finite asymptotic dimension.
Moreover, we construct examples of complete metrics with uniformly positive scalar curvature on non-compact spin manifolds which can be distinguished up to concordance relative to subsets which are coarsely negligible in a certain sense.

A technical novelty in this paper is that we use Yu's localization algebras in combination with the description of $\KTh$-theory for graded $\Cstar$-algebras due to Trout.
This formalism allows direct definitions of all the invariants we consider in terms of the functional calculus of the Dirac operator and enables us to give concise proofs of the product formulas.
\end{abstract}

\maketitle

\section{Introduction}
Let $X$ be an $n$-dimensional complete spin manifold equipped with a Riemannian metric $g$ of \emph{uniformly positive scalar curvature} (henceforth abbreviated by \enquote{\emph{upsc}}).
Moreover, let $\Gamma$ be a countable discrete group that acts on $X$ freely and properly by spin structure preserving isometries.
The (equivariant) \emph{coarse assembly map} of Higson and Roe $\mu \colon \KTh^\Gamma_\KPh(X) \to \KTh_\KPh(\roeAlg(X)^\Gamma)$, from (equivariant locally finite) $\KTh$-homology of $X$ into the $\KTh$-theory of the equivariant Roe algebra, fits into a long exact sequence
\begin{equation*}
    \dotsm \to \KTh_{n+1}(\roeAlg(X)^\Gamma) \overset{\bd}{\to} \structureGp_n^\Gamma(X) \to \KTh_n^\Gamma(X) \overset{\mu}{\to} \KTh_n(\roeAlg(X)^\Gamma) \to \dotsm,
\end{equation*}
where $\structureGp_n^\Gamma(X)$ is the \emph{analytic structure group} (which can also be realized as the $\KTh$-theory of a certain $\Cstar$-algebra).
The spinor Dirac operator $\diracOp$ on $X$ can be used to construct the $\KTh$-homological fundamental class $[\diracOp] \in \KTh_n^\Gamma(X)$, and $\Ind^\Gamma(\diracOp) = \mu\left( [\diracOp] \right) \in \KTh_n(\roeAlg(X)^\Gamma)$ is the (equivariant) \emph{coarse index} of $\diracOp$.
It is a well-known consequence of the Schrödinger--Lichnerowicz formula $\diracOp^2 = \nabla^\ast \nabla + \frac{\scalCurv_g}{4}$ that upsc implies vanishing of the index.
This statement can be refined by constructing a secondary invariant $\secInv^\Gamma(g) \in \structureGp_n^\Gamma(X)$, called the \emph{higher $\rho$-invariant} of the metric $g$, which is a lift of $[\diracOp] \in \KTh_n^\Gamma(X)$ to the structure group.
The $\rho$-invariant can be viewed as a $\KTh$-theoretic embodiment of the geometric reason for the vanishing of the index given by upsc.
The $\rho$-invariant can be used to distinguish psc metrics up to bordism, see~\cite[Corollary 1.16]{PS14Rho}.
Moreover, if $g_0$ and $g_1$ are two $\Gamma$-invariant metrics of upsc which are in the same uniform equivalence class,\footnote{This condition is automatically satisfied if the action is cocompact.} then there is a \emph{relative index} $\IndRel^\Gamma(g_0, g_1) \in \KTh_{n+1}(\roeAlg(X)^\Gamma)$ such that $\bd (\IndRel^\Gamma(g_0, g_1)) = \rho^\Gamma(g_0) - \rho^\Gamma(g_1)$.
The relative index is zero if the two positive scalar curvature metrics are concordant.

These secondary invariants have been the focus of intensive study in the recent past, see for instance~\cite{higson-roe:KhomologyAssemblyAndRigidityTheoremsForRelativeEtaInvariants,siegel:PhDthesis,XY14Relative,PS14Rho,weinberger-yu:finitePartOfOperatorKtheoryForGroupsFinitelyEmbeddable,XY14Positive,xie-yu:HigherRhoInvariantsAndTheModuliSpaceOfPSC}.
However, in standard applications, one takes $X := \widetilde{M}$ to be the universal covering of a closed spin manifold $M$ and $\Gamma = \pi_1(M)$.
One thereby obtains higher secondary invariants associated to psc metrics on $M$ by lifting them to $X$.
For instance, Weinberger--Yu~\cite{weinberger-yu:finitePartOfOperatorKtheoryForGroupsFinitelyEmbeddable} and Xie--Yu~\cite{xie-yu:HigherRhoInvariantsAndTheModuliSpaceOfPSC} apply this approach to distinguish psc metrics which are constructed using torsion elements of different orders in $\pi_1(M)$.
The $\Gamma$-action on $X = \widetilde{M}$ is cocompact and thus there is a canonical isomorphism $\KTh_\KPh(\roeAlg(X)^\Gamma) = \KTh_\KPh(\CstarRed \Gamma)$.
In contrast, we wish to emphasize that secondary index theory for positive scalar curvature also applies in non-(co)compact situations and indeed leads to interesting applications there.
Thus we always work in the setup of coarse index theory and define all the invariants for the general (possibly non-cocompact) case.

One central motivation for this paper is to give conceptual proofs in all dimensions of the secondary partitioned manifold index theorem and the \enquote{delocalized APS-index theorem} of Piazza--Schick~\cite{PS14Rho}.
Originally these results have been established only in the even-dimensional case but Xie--Yu~\cite{XY14Positive} reproved the delocalized APS-index theorem in all dimensions.
The main ingredient will be a variant of the product formula,
\begin{equation}
   \rho^{\Gamma_1}(g_1) \secProd [\diracOp_{X_2}] = \rho^{\Gamma_1 \times \Gamma_2}(g_1 \oplus g_2), \label{eq:introProductFormula}
\end{equation}
 where $X_i$ is a spin manifold with Riemannian metric $g_i$ endowed with a free and proper isometric $\Gamma_i$-action, $i= 1,2$, such that both $g_1$ on $X_1$ as well as $g_1 \oplus g_2$ on $X_1 \times X_2$ have upsc.
 Here \enquote{$\secProd$} denotes a suitable external product,
\begin{equation*}
   \structureGp_n^{\Gamma_1}(X_1) \tens \KTh_m^{\Gamma_2}(X_2) \overset{\secProd}{\to} \structureGp^{\Gamma_1 \times \Gamma_2}_{n+m}(X_1 \times X_2).
\end{equation*}
Another ingredient is the compatibility of the above product with Mayer--Vietoris boundary maps (see~\cref{sec:boundary}).

In addition, we establish a product formula for the relative index,
\begin{equation}
  \IndRel^{\Gamma_1 \times \Gamma_2}(g_{1,0} \oplus g_2, g_{1,1} \oplus g_2) = \IndRel^{\Gamma_1}(g_{1,0}, g_{1,1}) \secProd \Ind^{\Gamma_2}(\diracOp_2), \label{eq:introRelProductFormula}
\end{equation}
where $g_{1,i}$ for $i=0,1$ are two $\Gamma$-invariant metrics of uniform psc on $X_1$ in the same uniform equivalence class such that $g_{1,i} \oplus g_2$ have upsc on $X_1 \times X_2$.

A product formula as in~\labelcref{eq:introProductFormula} was established in the thesis of Siegel~\cite{siegel:PhDthesis}, where a construction of the structure group in terms of a customized notion of Kasparov cycles is used.
However, in Siegel's construction the compatibility between the exterior product and the Mayer--Vietoris boundary map appears to be not straightforward.
Recently, this approach has also been studied by Zenobi~\cite{zenobi:mappingTheSurgeryExactSequenceForTopologicalManifolds} with a focus on the signature operator and secondary invariants associated to homotopy equivalences.
Moreover, the product formula can be implemented using the geometric picture of the structure group due to Deeley--Goffeng~\cite{DG13GeometricI}.
Another discussion of \labelcref{eq:introProductFormula} can be found implicitly in the work of Xie--Yu~\cite[838--839]{XY14Positive} using Yu's localization algebras.
In the present paper, we introduce a variant of the latter approach which entails technical simplifications in the proof of the product formula.

The localization algebra $\roeAlgLoc(X)$ has been introduced by Yu to provide an alternative model for $\KTh$-homology and study the coarse assembly map~\cite{yu:localizationAlgebrasAndCoarseBaumConnes}.
The $\Gamma$-equivariant localization algebra $\roeAlgLoc(X)^\Gamma$ is the $\Cstar$-algebra generated by uniformly continuous $1$-parameter families $(L_t)_{t \in [1,\infty)}$ of operators $L_t$, each of which lies in the $\Gamma$-equivariant Roe algebra $\roeAlg(X)^\Gamma$, such that the propagation of $L_t$ tends to zero as $t \to \infty$.
There is a surjective evaluation homomorphism $\ev_1 \colon \roeAlgLoc(X)^\Gamma \to \roeAlg(X)^\Gamma$.
The main feature is that there exists an isomorphism $\IndL^\Gamma \colon \KOTh_\KPh^\Gamma(X) \overset{\iso}{\to} \KOTh_\KPh(\roeAlgLoc(X)^\Gamma)$, called the \enquote{local index map}, such that the $\Gamma$-equivariant coarse assembly map factors as $\mu = (\ev_1)_\KPh \circ \IndL^\Gamma$.
We will refer to the image of $[\diracOp]$ under the local index map as the \emph{local index class $\IndL^\Gamma(\diracOp)$} of the Dirac operator.
 Moreover, there is an exact sequence
\begin{equation*}
    0 \to \roeAlgLocZ(X)^\Gamma \to \roeAlgLoc(X)^\Gamma \to \roeAlg(X)^\Gamma \to 0,
\end{equation*}
where $\roeAlgLocZ(X)^\Gamma$ is the kernel of $\ev_1$, and the $\KTh$-theory group $\KTh_\KPh(\roeAlgLocZ(X)^\Gamma)$ is a suitable model for the structure group $\structureGp^\Gamma_\KPh(X)$.
The long exact sequence in $\KTh$-theory associated to this short exact sequence yields the desired (equivariant) Higson--Roe sequence after identifying $\KTh$-homology with the $\KTh$-theory of $\roeAlgLoc(X)^\Gamma$ via the local index map.

The main technical novelty of this paper is that we combine Yu's localization algebras with the description of $\KTh$-theory for graded $\Cstar$-algebras due to Trout~\cite{trout:gradedKTheoryEllipticOperators}.
That is, we view the $\KTh_0$-group of a $\Z_2$-graded $\Cstar$-algebra $A$ as the set of homotopy classes of graded $\ast$-homomorphisms $\contZGr \to A \tensGr \cptOps$, where $\contZGr$ is equal to $\contZ(\R)$ but graded into even and odd functions.
This formalism allows almost tautological definitions of the local index class as well as of the secondary invariant in terms of the functional calculus of the Dirac operator (see~\cref{subsec:localIndexClasses,subsec:partialIndex,subsec:rhoInvariant,subsec:twoPSCrelIndex}).
In order to treat all dimensions at once, we consider a $\cliffAlg_n$-linear variant of the localization algebras (see~\cref{sec:localizationAlgebras}) and consistently work with the $\cliffAlg_n$-linear (or $n$-multigraded) Dirac operator $\clnDiracOp$. 
Thus, our approach generalizes to the setting of real $\KTh$-theory.
Using this setup we give a concise and self-contained proof of the product formulas (see~\cref{subsec:productFormulas}).
Here we greatly benefit from the fact that this description of $\KTh$-theory is well adapted to products in the context of index theory.

To compare our construction to previous approaches, in \cref{sec:comparison} we describe our secondary invariants explicitly as ordinary complex $\KTh$-theory classes in terms of projections and unitaries. 
As a consequence, we observe that our construction of the local index classes and $\rho$-invariants is essentially equivalent to Xie--Yu's~\cite{XY14Positive}.

Note that the relative index $\IndRel^\Gamma(g_0, g_1)$ of two upsc metrics on $X$ can be defined as an index associated to a metric on $\R \times X$ which interpolates between $g_0$ and $g_1$ (but does not necessarily have upsc along the way), see~\cref{subsec:twoPSCrelIndex}.
Therefore, to deal with both~\labelcref{eq:introProductFormula,eq:introRelProductFormula} in a unified way, we are led to consider metrics which have upsc outside a prescribed $\Gamma$-invariant subset $Z \subset X$.
In this situation Roe has shown that the coarse index can be \enquote{localized} to the subset $Z$,~\cite{roe:indexTheoryCoarseGeometryTopologyOfManifolds,Roe15:Partial}.
In \cref{subsec:partialIndex}, we observe that this implies that the local index class lifts to a particular element $\IndLP{Z}^\Gamma(\clnDiracOp^{g}) \in \KTh_n(\roeAlgLocP{Z}(X)^\Gamma)$, where $\roeAlgLocP{Z}(X)^\Gamma$ is the closure of all elements $L \in \roeAlgLoc(X)^\Gamma$ such that the operator $\ev_1(L)$ is supported in an $R$-neighborhood of the subset $Z$ for some $R \geq 0$.
We view the class $\IndLP{Z}^\Gamma(\clnDiracOp^{g})$ as a \emph{partial secondary invariant} associated to the fact that the underlying metric has partially upsc.

Although our initial motivation was to deal with the relative index, we wish to point out that these partial secondary invariants are of intrinsic interest in the study of certain positive scalar curvature phenomena on non-compact manifolds.
Recall that the \enquote{localized coarse index} of Roe~\cite{Roe15:Partial} gives obstructions against partial uniform positive scalar curvature.
Our partial secondary invariants can be viewed as a secondary analogue of this, that is, they can be used to obtain obstructions against \enquote{partial concordance} of upsc metrics.
Indeed, the class $\IndLP{Z}^\Gamma(\clnDiracOp^{g})$ depends on $g$ only up to \enquote{concordance relative to $Z$} (see~\cref{defi:concordanceRelZ,thm:concordanceInvariance}).

\subsection*{Main results}
A significant part of this paper is devoted to developing a theory of partial secondary invariants.
As a central principle, we have the following external product formula.
\begin{restatable*}{thm}{thmProductFormula} \label{thm:productFormula}
 Let $X = X_1 \times X_2$ be a product of two complete spin manifolds and suppose that $X_i$ is endowed with a free and proper $\Gamma_i$-action \parensup{$i=1,2$}.
 Suppose that the metric $g_1$ on $X_1$ has upsc outside a closed $\Gamma_1$-invariant subset $Z_1 \subset X_1$, and $g = g_1 \oplus g_2$ has upsc outside $Z = Z_1 \times X_2 \subset X_1 \times X_2$.
 Then,
 \begin{equation}
  \IndLP{Z}^{\Gamma_1 \times \Gamma_2}(\clnDiracOp^{g}_X) =  \IndLP{Z_1}^{\Gamma_1}(\clnDiracOp_{X_1}^{g_1}) \secProd \IndL^{\Gamma_2}(\clnDiracOp_{X_2}). \label{eq:generalProductFormula}
 \end{equation}
\end{restatable*}
Here we use an external product,
\begin{equation*}
  \KTh_{n_1}\left(\roeAlgLocP{Z_1}(X_1)^{\Gamma_1} \right) \tens \KTh_{n_2}\left(\roeAlgLoc(X_2)^{\Gamma_2}\right) \overset{\secProd}{\to} \KTh_{n_1 + n_2}\left(\roeAlgLocP{Z_1 \times X_2}(X_1 \times X_2)^{\Gamma_1 \times \Gamma_2} \right),
\end{equation*}
that we construct in \cref{subsec:externalProductLocAlgebras}.

From this theorem we deduce \labelcref{eq:introProductFormula,eq:introRelProductFormula} as well as the classical product formula for the fundamental classes (that is, in our setup, local index classes), see~\cref{cor:productFormula,cor:productFormulaRelIndex}.
For example, setting $Z = \emptyset$ and taking $\secInv^\Gamma(g)$ to be $\IndLP{\emptyset}^\Gamma(\clnDiracOp^g)$, we obtain precisely \labelcref{eq:introProductFormula}.

It is crucial for applications that the external product is compatible with (Mayer--Vietoris) boundary maps.
In our construction this is the case because the external product we use is induced by the external product in $\KTh$-theory, where compatibility with boundary maps can be checked abstractly.
Together with the product formula we obtain a partitioned manifold index theorem for partial secondary invariants  as follows:

 Let $W$ be a spin manifold endowed with a free and proper $\Gamma$-action and a $\Gamma$-invariant Riemannian metric $h$.
 Suppose that $X \subset W$ is a closed submanifold of codimension one such that $W \setminus X$ has two connected components, each of which is $\Gamma$-invariant individually, and the metric $h$ has a product structure on a tubular neighborhood of $X$.
 In this situation, we say that $(W,h)$ is \emph{partitioned by $(X,g)$}, where $g$ denotes the restriction of $h$ to $X$.
 Furthermore, if $Z \subseteq W$ is a closed $\Gamma$-invariant subset that satisfies a certain technical condition (see~\cref{defi:admissible}), then we have the following result:
 
\begin{restatable*}{thm}{thmPartitionedManifold} \label{thm:partitionedManifold}
 Let $(W,h)$ be a complete Riemannian spin manifold endowed with a free and proper $\Gamma$-action and suppose that it is partitioned by $(X,g)$.
 Let $Z \subseteq W$ be a closed $\Gamma$-invariant subset that is admissible with respect to $W_+$ and suppose that $h$ has upsc outside $Z$.
 Then the Mayer--Vietoris boundary map
 \begin{equation*}
  \bdMV \colon \KTh_{n+1}(\roeAlgLocP{Z}(W)^\Gamma) \to \KTh_n(\roeAlgLocP{Z \cap X}(X)^\Gamma)
 \end{equation*}
  associated to the cover $W = W_+ \cup W_-$ satisfies
  \begin{equation*}
   \bdMV\left( \IndLP{Z}^\Gamma(\clnDiracOp_W^h) \right) = \IndLP{Z \cap X}^\Gamma(\clnDiracOp_X^g).
  \end{equation*}
\end{restatable*}

Specializing to the case where $Z = \emptyset$ and $X$ is $\Gamma$-cocompact, we obtain a new proof of the secondary partitioned manifold index theorem of Piazza--Schick~\cite[Theorem 1.22]{PS14Rho} for all dimensions.
We also deduce a partitioned manifold index theorem for the relative index of two upsc metrics (see~\cref{thm:relativeIndexPartitionedManifold}).

Moreover, it turns out that the delocalized APS-index theorem of Piazza--Schick is itself a consequence of the partitioned manifold index theorem for partial secondary invariants.
Indeed, consider a spin manifold $W$ with boundary $\bd W = X$ endowed with a complete Riemannian metric $h$ and a proper and free isometric action of $\Gamma$.
  Suppose that $h$ is a product metric~$g \oplus \D{t}^2$ on a collar neighborhood of $X$.
  We also assume that the inclusion $X \hookrightarrow W$ is a coarse equivalence.\footnote{This is automatically satisfied for instance if $W$ is a $\Gamma$-covering of a null-bordism of some closed manifold.}
  This implies that the index $\Ind^\Gamma(\clnDiracOp_X)$ vanishes (this is a variant of coarse bordism invariance, see~\cite{W12Bordism}).
  More precisely, we construct a \enquote{$\rho$-invariant of the null-bordism $W$} which we denote by $\rho^\Gamma(W) \in \KTh_{n}(\roeAlgLocZ(X)^\Gamma)$ and which lifts the local index class, see~\cref{defi:rhoOfBordism}.
  In addition, if the metric $g$ on $X$ has upsc, then we can also construct a certain relative index $\Ind_W^\Gamma(\clnDiracOp_W) \in \KTh_\KPh(\roeAlg(X)^\Gamma) \iso \KTh_\KPh(\roeAlg(W)^\Gamma)$, compare~\cite{PS14Rho}.
  If we attach a cylindrical end $X \times [0,\infty)$ to $W$, we may apply the partitioned manifold index theorem for partial secondary invariants to the resulting manifold.
  Together with some additional formal arguments this yields:
  
\begin{restatable*}{thm}{thmRefinedDeloalizedAPS} \label{thm:refinedDeloalizedAPS}
  In the above setup, suppose that the metric $g_X$ has upsc on $X = \bd W$.
  Then
  \begin{equation*}
    \bd(\Ind_W^\Gamma(\clnDiracOp_W)) = \rho^\Gamma(g_X) - \rho^\Gamma(W) \in \KTh_n(\roeAlgLocZ(X)^\Gamma),
  \end{equation*}
  where $\bd \colon \KTh_{\KPh+1}(\roeAlg(X)^\Gamma) \to \KTh_\KPh(\roeAlgLocZ(X)^\Gamma)$ is the boundary map in the Higson--Roe exact sequence of $X$.
\end{restatable*}
This theorem is a refined version of \cite[Theorem 1.14]{PS14Rho} (respectively \cite[Theorem 4.1]{XY14Positive}) because we have an equality in $\KTh_\KPh(\roeAlgLocZ(X)^\Gamma)$ instead of merely in $\KTh_\KPh(\roeAlgLocZ(W)^\Gamma)$.
In fact, the correction term $\rho^\Gamma(W)$ vanishes after pushing it forward to $\KTh_\KPh(\roeAlgLocZ(W)^\Gamma)$ and thus it recovers the original result of Piazza--Schick.
  
We now turn to geometric consequences of our methods.
The main application of (partial) secondary invariants for positive scalar curvature is to distinguish upsc metrics up to bordism or concordance.
Product formulas allow to extend these applications simply by taking direct products with certain manifolds.
Before stating the first corollary, we recall a geometric concept due to Gromov.

A complete Riemannian manifold $Y$ is called \emph{hypereuclidean} if it admits a proper Lipschitz map $Y \to \R^q$ of degree $1$ into some Euclidean space $\R^q$ (if this is the case, then $q = \dim Y$).
Furthermore, we say that $Y$ is \emph{stably hypereuclidean} if $Y \times \R^k$ is hypereuclidean for some $k \geq 0$.

\begin{restatable*}{cor}{thmDistinguishStabilizeByHypereucl} \label{thm:distinguishStabilizeByHypereucl}
  Let $X$ be a proper metric space endowed with a proper and free $\Gamma$-action, $Z \subset X$ a closed $\Gamma$-invariant subset.
  Suppose that $Y$ is a $q$-dimensional complete spin manifold that is stably hypereuclidean and endowed with a proper and free $\Lambda$-action.
  Then the map
  \begin{equation*}
  \KTh_\KPh(\roeAlgLocP{Z}(X)^\Gamma) \to \KTh_{\KPh+q}(\roeAlgLocP{Z \times Y}(X \times Y)^{\Gamma \times \Lambda}), \quad  x \mapsto x \secProd \IndL^\Lambda(\clnDiracOp_{Y}),
  \end{equation*}
  is split-injective.
  
  In particular, if we suppose that $X$ is a spin manifold endowed with two complete $\Gamma$-invariant metrics $g_0$ and $g_1$ which are in the same uniform equivalence class and have upsc outside $Z$ with $\IndLP{Z}^\Gamma(\clnDiracOp^{g_0}_X) \neq \IndLP{Z}^\Gamma(\clnDiracOp^{g_1}_X)$, then also $\IndLP{Z \times Y}^{\Gamma \times \Lambda}(\clnDiracOp^{g_0 \oplus g_Y}_{X \times Y}) \neq \IndLP{Z \times Y}^{\Gamma \times \Lambda}(\clnDiracOp^{g_1 \oplus g_Y}_{X \times Y})$ provided that $g_i \oplus g_Y$ for $i=0,1$ have upsc outside $Z \times Y$.
\end{restatable*}
In particular, with $Z = \emptyset$, this yields a stability result for $\rho$-invariants concerning products with (stably) hypereuclidean manifolds.

If $M$ is a closed $n$-dimensional spin manifold together with a reference map $u \colon M \to \Bfree \Gamma$ (for instance, take $\Gamma = \pi_1(M)$ and $u$ to be the map classifying the universal covering), then one can define a higher $\rho$-invariant $\rho^u(g) \in \structureGp_n^\Gamma$ for each psc metric $g$ on $M$, where $\structureGp_\KPh^\Gamma$ is the universal structure group associated to $\Gamma$, see~\cite{PS14Rho} and \cref{defi:univStructGp}.
By a result of Dranishnikov~\cite[Theorem 3.5]{dranishnikov:onHypereuclideanManifolds}, the universal covering of an aspherical manifold is stably hypereuclidean if the fundamental group has finite asymptotic dimension.
Combining this with our corollary above, we obtain the following stability result for higher $\rho$-invariants on closed manifolds:
\begin{restatable*}{cor}{corDistinguishAfterProductWithAspherical} \label{cor:distinguishAfterProductWithAspherical}
  Let $M_i$ be closed spin, $u_i \colon M_i \to \Bfree \Gamma$ and $g_i$ a psc metric on $M_i$, $i=0,1$, such that $\rho^{u_0}(g_0) \neq \rho^{u_1}(g_1)$.
  Let $N$ be a closed aspherical spin manifold such that $\Lambda = \pi_1(N)$ has finite asymptotic dimension.
  Let $g_N$ be a Riemannian metric such that $g_i \oplus g_N$ has psc on $M_i \times N$ for $i=0,1$.
    Then $\rho^{u_0 \times \id_{N}}(g_0 \oplus g_N) \neq \rho^{u_1 \times \id_N}(g_1 \oplus g_N)$.
\end{restatable*}
 In particular, this implies that $(M_0 \times N, u_0 \times \id_N, g_0 \oplus g_N)$ and $(M_1 \times N, u_1 \times \id_N, g_1 \oplus g_N)$ are not bordant as psc manifolds with reference map (see \cref{sec:higherInvariants} for a more detailed discussion).

 However, \cref{cor:distinguishAfterProductWithAspherical} would be false if we allowed $N$ to admit psc (this is excluded in a strong sense since $N$ is assumed to be aspherical and finite asymptotic dimension implies the analytic Novikov conjecture for $\Lambda$).
 Indeed, it is a simple geometric fact that any two psc metrics $h_i$ on $M \times N$ that can be written as a product $h_i = g_{M,i} \oplus g_{N,i}$ are concordant if $M$ and $N$ both admit psc individually.
 Thus, if $M$ and $N$ both admit psc and we have $\rho^{v}(h_0) \neq \rho^{v}(h_1)$ on $M \times N$ for some $v \colon M \times N \to \Bfree \Gamma$, then $h_0$ or $h_1$ is not concordant to a metric of product form (see~\cref{prop:notAProductMetric}).
 
 Our methods also allow to produce examples of complete upsc metrics on non-compact manifolds which are distinguished by certain partial secondary invariants.
 As input for the constructions we will use psc metrics on closed spin manifolds which are distinguished by the higher $\rho$-invariant (for instance from \cite{weinberger-yu:finitePartOfOperatorKtheoryForGroupsFinitelyEmbeddable,xie-yu:HigherRhoInvariantsAndTheModuliSpaceOfPSC}).
 For instance, we have a corollary of our partitioned manifold index theorem which can be viewed as a secondary analogue of \cite[Proposition 3.2]{Roe15:Partial}.
 \begin{restatable*}{cor}{thmCorOfPartitionedMfd}\label{thm:corOfPartitionedMfd}
  Let $M$ be a closed spin manifold together with a map $u \colon M \to \Bfree \Gamma$ and two psc metrics $g_0$ and $g_1$ such that $\rho^u(g_0) \neq \rho^u(g_1)$.
  Let $W$ be a spin manifold with two complete upsc metrics $h_0$, $h_1$ in the same uniform equivalence class such that $(W,h_i)$ is partitioned by $(M,g_i)$, $i=0,1$.
  Suppose furthermore that $u$ extends to a map $W \to \Bfree \Gamma$.
  Then $h_0$ and $h_1$ are not concordant relative to $W_-$ \parensup{or $W_+$}, where $W_\pm$ are the connected components of $W \setminus M$.
\end{restatable*}
 We can draw similar consequences from \cref{thm:distinguishStabilizeByHypereucl}.
 We say that a subset $Z \subseteq X$ is \emph{coarsely negligible in $X$} if the inclusion map coarsely factors through a flasque space (see~\cref{defi:coarseNegligible}).
 Examples of coarsely negligible subsets include compact subsets of complete Riemannian manifolds and half spaces in Euclidean spaces.
 \begin{restatable*}{thm}{thmProductCoarselyNegligible} \label{thm:productCoarselyNegligible}
    Let $M$ be a closed spin manifold together with a map $u \colon M \to \Bfree \Gamma$ and two psc metrics $g_0$ and $g_1$.
    Moreover, let $Y$ be a spin manifold with a complete Riemannian metric $g_Y$ and $Z \subseteq Y$ some subset.
    Suppose that,
    \begin{enumerate}[(i)]
    \item $\rho^u(g_0) \neq \rho^u(g_1)$,
     \item $g_i \oplus g_Y$ have upsc on $M \times Y$ for $i=0,1$,
     \item $(Y, g_Y)$ is stably hypereuclidean,
     \item $Z$ is coarsely negligible in $Y$.
    \end{enumerate}    
    Then the metrics $g_0 \oplus g_X$ and $g_1 \oplus g_X$ are not concordant on $M \times Y$ relative to~$M \times Z$.
\end{restatable*}
A simple example to which this theorem can be always applied is $Y = \R^q$ and $Z = [0,\infty) \times \R^{q-1}$.

\subsection*{Acknowledgements}
The author wishes to express his gratitude to his advisor Thomas Schick for sharing his knowledge and many fruitful discussions.
He also would like to thank Ralf Meyer and the anonymous referees for useful comments on the manuscript.

   %Main novelity: Combination of graded K-Theory with Cln-linear version! This has been suggested by Piazza-Schick and Xie--Yu

\section{Prerequisites} \label{sec:gradedKtheory}
\subsection{Graded \texorpdfstring{$\Cstar$}{C*}-algebras and \texorpdfstring{$\KTh$}{K}-Theory }
We use the approach to $\KTh$-theory for graded $\Cstar$-algebras due to Trout~\cite{trout:gradedKTheoryEllipticOperators} but we will mostly follow the exposition in the lecture notes~\cite{higson-guentner:GroupCstarAlgebrasAndKTheory}.

We work with \emph{Real} $\Cstar$-algebras.
The reader may \enquote{complexify} this section simply by ignoring the Real structure.
A \emph{Real} $\Cstar$-algebra is a complex $\Cstar$-algebra $A$ together with an involutive conjugate-linear $\ast$-automorphism $A \to A$, $a \mapsto \conjg{a}$.
We require $*$-homomorphisms $\varphi \colon A \to B$ between Real $\Cstar$-algebras to preserve the Real structure, that is, $\varphi(\conjg{a}) = \longconjg{\varphi(a)}$ for all $a \in A$.

A \emph{grading} on a (Real) $\Cstar$-algebra $A$ is a Real $*$-automorphism $\alpha \colon A \to A$ such that $\alpha^2 = \id$.
A $\Cstar$-algebra together with a grading is called a $\emph{graded $\Cstar$-algebra}$.
Alternatively, a grading may be viewed as a direct sum decomposition $A = A^{(0)} \oplus A^{(1)}$ into selfadjoint subspaces such that $A^{(i)} A^{(j)} \subseteq A^{(i+j)}$, where $A^{(i)}$ is the $(-1)^i$-eigenspace of $\alpha$, $i \in \Z_2$.
All $*$-homomorphisms $\varphi \colon A \to B$ between graded $\Cstar$-algebras will be assumed to preserve the grading in the sense that $\varphi \circ \alpha = \alpha \circ \varphi$.
Note that any $\Cstar$-algebra can be \emph{trivially graded} by setting $\alpha = \id$.

The Real $\Cstar$-algebra of continuous functions on the real line which vanish at infinity admits a grading defined by the reflection map $f \mapsto (x \mapsto f(-x))$.
We will denote this graded Real $\Cstar$-algebra by $\contZGr$.

Let $\hilbert = \hilbert^{(0)} \oplus \hilbert^{(1)}$ be a fixed graded Real Hilbert space, where $\hilbert^{(0)} = \hilbert^{(1)}$  is countably infinite-dimensional. 
Let $\cptOps$ denote the Real $\Cstar$-algebra of compact operators on $\hilbert$, graded by the decomposition into diagonal and off-diagonal matrices.
Such a grading is known as a \emph{standard even grading}.

Given two graded $\Cstar$-algebras $A$ and $B$, we denote their maximal graded tensor product by $A \tensGr B$.
We will always use \emph{maximal} tensor products unless specified otherwise.

A central feature of this paper is the use of Clifford algebras.
See~\cite{atiyah-bott-shapiro:cliffordModules,lawson-michelsohn:spinGeometry} for general references.
We use the following notation.
The \emph{Clifford algebra} $\cliffAlg_{n,m}$ is the Real $\Cstar$-algebra generated by real, odd generators $\{ e_1, \dotsc, e_n, \varepsilon_1, \dotsc, \varepsilon_m \}$ subject to the relations $e_i e_j + e_j e_i = -2 \delta_{ij}, \varepsilon_k \varepsilon_l + \varepsilon_l \varepsilon_k = +2 \delta_{kl}, e_i \varepsilon_k + \varepsilon_k e_i = 0, e_i^* = -e_i, \varepsilon_k^* = \varepsilon_k$.
As shorthands we denote $\cliffAlg_{n,0}$ by $\cliffAlg_n$ and $\cliffAlg_{0,n}$ by $\cliffAlgDual_n$.
There is a canonical isomorphism $\cliffAlg_{n,m} \tensGr \cliffAlg_{n^\prime, m^\prime} = \cliffAlg_{n+n^\prime, m+m^\prime}$.
In cases where we do not use the Real structure, we will denote the Clifford algebras by $\cliffAlgC_n$.
Moreover, $\cliffAlg_{n,n}$ is isomorphic to the matrix algebra $\Mat_{2^n}(\RReal)$.
Here \enquote{$\RReal$} denotes the Real algebra $\C$ endowed with the standard complex conjugation.
For $n > 0$, the algebra $\cliffAlg_{n,n}$ is endowed with a standard even grading using an identification $\cliffAlg_{n,n} \iso \Mat_{2^{n}}(\RReal) = \Mat_2(\Mat_{2^{n-1}}(\RReal))$.
In particular, we have an isomorphism $\cliffAlg_{n,n} \tensGr \cptOps \iso \cptOps$ for all $n \geq 0$.

Let $A$, $B$ be graded Real $\Cstar$-algebras, then we denote by $[A,B]$ the set of homotopy classes of $*$-homomorphisms $A \to B$ (with respect to homotopies preserving the given Real structure and grading).
In other words, $[A,B] = \pi_0(\Hom(A,B))$, where $\Hom(A,B)$ denotes the space of $\ast$-homomorphisms $A \to B$ endowed with the point-norm topology.
The homotopy class of a $*$-homomorphism $\varphi \colon A \to B$ will be denoted by $[\varphi]$. 

\begin{defi}
 Let $A$ be a graded Real $\Cstar$-algebra. 
 For $n \geq 0$, we define the group 
  \begin{equation*}
  \KOThGr_n(A) := \pi_n(\KOSpec(A)),
  \end{equation*} 
  where $\KOSpec(A) := \Hom(\contZGr, A \tensGr \cptOps)$ with the zero map as base-point.
\end{defi}

 If $A$ is a trivially graded algebra, we use the symbol $\KOTh_*(A)$ synonymously with $\KOThGr_*(A)$.
 This is justified because for trivially graded algebras this agrees with ordinary $\KOTh$-theory of $\Cstar$-algebras defined in terms of projections.
 Explicit isomorphisms in the complex case are explained in \cref{subsec:triviallyGraded}.
 
 One can verify that the $n$-fold loop space of $\KOSpec(A)$ is canonically homeomorphic to $\KOSpec(\suspCstar^n A)$.
 Here $\suspCstar^n A$ denotes the $n$-fold \emph{suspension} of $A$, that is, $\suspCstar^n A = \contZ(\R^n) \tensGr A$, where $\contZ(\R^n)$ is endowed with the trivial grading.
 In particular, we have $\KOThGr_n(A) = \pi_0(\KOSpec(\suspCstar^n A)) = \KOThGr_0(\suspCstar^n A) = \left[\contZGr, \suspCstar^n A \tensGr \cptOps \right]$.

 The direct sum induces a map $\KOSpec(A) \times \KOSpec(A) \to \KOSpec(A)$, taking a pair $(\phi, \psi) \in \KOSpec(A) \times \KOSpec(A)$ to the composition $\contZGr \overset{\phi \oplus \psi}{\to} (A \tensGr \cptOps ) \oplus (A \tensGr \cptOps) = A \tensGr (\cptOps \oplus \cptOps) \subset A \tensGr \cptOps$, where we use an embedding $\cptOps \oplus \cptOps \subset \cptOps$ coming from the diagonal embedding $\cptOps \oplus \cptOps \subset \Mat_2(\cptOps)$ and an even unitary isomorphism $\hilbertGr \oplus \hilbertGr \iso \hilbertGr$.
 The choice of such unitary does not matter up to homotopy.
 It can be shown that this defines a commutative H-group structure on $\KOSpec(A)$, thereby turning $\KThGr_n(A) = \pi_n(\KOSpec(A))$ into an abelian group for all $n \geq 0$.
 By a general principle in homotopy theory, this agrees with the homotopy group structure on $\pi_n$ for $n \geq 1$.
 It is possible to turn $\KOSpec(A)$ into a spectrum so that its homotopy groups are precisely the $\KTh$-theory groups we have just defined, see~\cite{dellAmbrogio-emerson-kandelaki-meyer:AFunctorialEquivariantKTheorySpectrumAndAnEquivariantLefschetzFormula}.
  
 \begin{rem} \label{rem:defineGradedKTheoryClassByBareHom}
  Any graded $*$-homomorphism $\varphi \colon \contZGr \to A$ defines an element $[\varphi] := [\varphi \tensGr e_{11}] \in \KOThGr_0(A)$, where $e_{11}$ is an even rank $1$ projection in $\cptOps$.
 \end{rem}

\subsection{External Product}
There is a comultiplication $\comult \colon \contZGr \to \contZGr \tensGr \contZGr, f \mapsto f(\mathrm{x} \tensGr 1 + 1 \tensGr \mathrm{x})$, given by the functional calculus of the unbounded multiplier $\mathrm{x} \tensGr 1 + 1 \tensGr \mathrm{x}$.
The comultiplication $\comult$ is coassociative and counital (with counit $\eta \colon \contZGr \to \C, \eta(f) = f(0)$).

On the generators $\{ \eu^{-\mathrm{x}^2}, \mathrm{x} \eu^{-\mathrm{x}^2}\}$ of $\contZGr$ the comultiplication satisfies,
\begin{equation}
 \comult(\eu^{-\mathrm{x}^2}) = \eu^{-\mathrm{x}^2} \tensGr \eu^{-\mathrm{x}^2}, \qquad \comult(\mathrm{x} \eu^{-\mathrm{x}^2}) = \mathrm{x} \eu^{-\mathrm{x}^2} \tensGr \eu^{-\mathrm{x}^2} + \eu^{-\mathrm{x}^2} \tensGr \mathrm{x} \eu^{-\mathrm{x}^2}. \label{eq:comultOnGen}
\end{equation}

An alternative construction of $\comult$ not depending on the theory of unbounded multipliers is indicated in~\cite[Section 1.3]{higson-guentner:GroupCstarAlgebrasAndKTheory}.

 Let $\varepsilon > 0$ and denote by $\contZGr(-\varepsilon, \varepsilon)$ the graded ideal in $\contZGr$ consisting of those functions which vanish outside the interval $(-\varepsilon, \varepsilon)$.
 It can be checked that the comultiplication $\comult$ preserves $\contZGr(-\varepsilon, \varepsilon)$ in the sense that $\comult \left(\contZGr(-\varepsilon, \varepsilon)\right) \subseteq \contZGr(-\varepsilon, \varepsilon) \tensGr \contZGr(-\varepsilon, \varepsilon)$.

Before moving on to define the external product, we note the following elementary lemmas since they are essential to our discussion of secondary invariants in \cref{subsec:partialIndex}.

\begin{lem} \label{lem:elementaryHomotopyEqu}
 The inclusion maps $\contZGr(-r,r) \hookrightarrow \contZGr$ and $\contZGr(-r,r) \tensGr \contZGr(-r,r) \hookrightarrow \contZGr \tensGr \contZGr$ are homotopy equivalences of graded $\Cstar$-algebras for all $r > 0$.
\end{lem}

\begin{lem} \label{cor:compressionCommutesWithComult}
 Let $\psi \colon \contZGr \to \contZGr(-r,r)$ be a graded $*$-homomorphism that is a homotopy inverse to the inclusion $\contZGr(-r,r) \hookrightarrow \contZGr$.
 Then $\comult \circ \psi$ and $\psi \tensGr \psi \circ \comult$ are homotopic as graded $*$-homomorphisms $\contZGr \to \contZGr(-r,r) \tensGr \contZGr(-r,r)$.
\end{lem}

\begin{defi}[{\cite[Section 1.7]{higson-guentner:GroupCstarAlgebrasAndKTheory}}]
 The external product $\KOThGr_n(A) \tens \KOThGr_m(B) \overset{\KOThProd}{\to} \KOThGr_{n+m}(A \tensGr B)$ is induced by the map $\KOSpec(A) \smashProd \KOSpec(B) \to \KOSpec(A \tensGr B)$, taking a pair $(\phi, \psi) \in \KOSpec(A) \times \KOSpec(B)$ to the composition $\contZGr \overset{\triangle}{\to} \contZGr \tensGr \contZGr \overset{\phi \tensGr \psi}{\to} A \tensGr \cptOps \tensGr B \tensGr \cptOps \cong A \tensGr B \tensGr ( \cptOps \tensGr \cptOps ) \cong A \tensGr B \tensGr \cptOps$.
\end{defi}

Here we implicitly use a fixed isomorphism $\cptOps \tensGr \cptOps \iso \cptOps$ coming from an even unitary isomorphism $\hilbertGr \tensGr \hilbertGr \iso \hilbertGr$.
As in the case of the direct sum, the choice of such an identification does not matter up to homotopy.
\begin{rem}
 If $x = [\phi] \in \KOThGr_0(A)$ and $y = [\psi] \in \KOThGr_0(B)$ are represented by homomorphisms $\phi \colon \contZGr \to A$ and $\psi \colon \contZGr \to B$ as in \cref{rem:defineGradedKTheoryClassByBareHom}, then $x \KOThGrProd y$ is represented by $\phi \tensGr \psi \circ \comult \colon \contZGr \to A \tensGr B$ (the rank $1$ projections take care of themselves because $e_{11} \tensGr e_{11} \in \cptOps \tensGr \cptOps \iso \cptOps$ is again an even rank $1$ projection).
\end{rem}

\subsection{Bott periodicity} \label{subsec:bott}
In this subsection, we briefly sketch a variant of the \enquote{Dirac--dual-Dirac} approach to Bott periodicity using Clifford algebras.
For more elaborations and proofs we refer to \cite[Section 1.10]{higson-guentner:GroupCstarAlgebrasAndKTheory},~\cite[Lemma 4.3]{dumitrascu:asymptoticMorphismKHomology}.

The  \emph{dual Dirac element} or \emph{Bott element} is the class $b_n \in \KOThGr_0\left(\contZ(\R^n) \tensGr \cliffAlgDual_n \right)$ defined by the graded $\ast$-homomorphism
\begin{equation*}
   \beta_n \colon \contZGr \to \contZ(\R^n, \cliffAlgDual_n), \quad \beta(f) = \left( v \mapsto f(v) \right),
\end{equation*}
where \enquote{$f(v)$} denotes the application of $f$ on $v \in \R^n \subseteq \cliffAlgDual_n$ via the functional calculus in $\cliffAlgDual_n$.

We will occasionally use \emph{asymptotic morphisms} and ideas from $\mathrm{E}$-theory.
For general references see~\cite{guentner-higson-trout:EquivariantETheoryForCstarAlgebras,higson-guentner:GroupCstarAlgebrasAndKTheory}.
This first surfaces in the following, where we use that an asymptotic morphism $\alpha \colon \contZGr \tensGr A \asympMorphism B$ induces a map on $\KTh$-theory $\alpha_\KPh \colon \KThGr_\KPh(A) \to \KThGr_\KPh(B)$, see~\cite[Remark 1.11]{higson-guentner:GroupCstarAlgebrasAndKTheory}.
There is an asymptotic morphism $\alpha \colon \contZGr \tensGr \contZ(\R^n) \asympMorphism \cptOps \tensGr \cliffAlg_n$, called the \emph{Dirac element}.
Indeed, $\alpha$ is defined using the Dirac operator $\clnDiracOp_{\R^n} = \sum_{i = 1}^n e_i \cdot \frac{\partial}{\partial x_i}$ on $\Lp^2(\R^n, \cliffAlg_n)$, $\alpha_t(f \tensGr g) = f(\frac{1}{t} \clnDiracOp_{\R^n}) g$.
It is the inverse of the dual Dirac element in the sense that the induced homomorphism $\alpha_* \colon \KOThGr_0(\contZ(\R^n) \tensGr \cliffAlgDual_n) \to \KOThGr_0(\cptOps) = \KOThGr_0(\RReal)$ maps $b_n$ to the unit element $1 \in \KOThGr_0(\RReal)$.
Employing a variant of Atiyah's rotation trick, one can conclude that the \emph{Bott map},
\begin{gather*}
   \KOThGr_0(A) \to \KOThGr_0( A \tensGr \contZ(\R^n) \tensGr \cliffAlg^*_n ), \quad x \mapsto x \KOThGrProd b_n,
\end{gather*}
is an isomorphism for any graded $\Cstar$-algebra $A$ and all $n \in \N$.
In particular, there is a natural isomorphism
\begin{equation}
 \KOThGr_0(A \tensGr \cliffAlg_n) \iso \KOThGr_0\left( A \tensGr \contZ(\R^n) \tensGr \cliffAlg_{n} \tensGr \cliffAlgDual_n \right) \iso \KOThGr_0\left(\suspCstar^n A \right) = \KOThGr_n\left(A\right), \label{eq:higherKTheoryByCliffAlg}
\end{equation}
where the second isomorphism follows from the fact that $\cliffAlg_n \tensGr \cliffAlgDual_n = \cliffAlg_{n,n} \iso \Mat_{2^{n}}(\RReal)$ with a standard even grading.
Moreover, $\cliffAlg_8 \iso \Mat_{16}(\RReal)$ with a standard even grading, which together with \labelcref{eq:higherKTheoryByCliffAlg} implies $8$-fold periodicity of real $\KTh$-theory.
Similarly, in the complex case, we get $2$-fold periodicity due to $\cliffAlgC_2 \iso \Mat_2(\C)$.
\subsection{Long exact sequences} \label{subsec:longExactSeq}
Let $0 \to I \to A \to A/I \to 0$ be a short exact sequence of graded $\Cstar$-algebras.
Then the induced map $\KOSpec \left(A\right) \to \KOSpec(A/I)$ is a Serre fibration with fiber $\KOSpec(I)$.
Thus, there is a long exact sequence of homotopy groups which yields the natural long exact sequence in $\KOTh$-theory,
\begin{equation*}
	\dotsm \to \KOThGr_{n+1}(A/I) \overset{\bd}{\to} \KOThGr_n(I) \to \KOThGr_n(A) \to \KOThGr_n(A/I) \to \dotsm \to \KOThGr_0(A/I).
\end{equation*}
Using Bott periodicity one can also define $\KOTh$-theory for negative degrees and extend the exact sequence to the right.
In fact, it then becomes the $24$ term (respectively $6$ term in the complex case) cyclic exact sequence.

If $B$ is another graded $\Cstar$-algebra, then the sequence $0 \to I \tensGr B \to A \tensGr B \to A/I \tensGr B \to 0$ is exact\footnote{Recall that we use the maximal tensor product.} and it can be checked that the boundary map is compatible with the external product, that is, $\bd(x) \KThProd y = \bd(x \KThProd y)$ for $x \in \KThGr_\KPh(A/I), y \in \KThGr_\KPh(B)$.

   In addition, there is a Mayer--Vietoris sequence for the $\KOTh$-theory of graded $\Cstar$-algebras.
   Here we consider a graded $\Cstar$-algebra $A$ and two closed two-sided graded ideals $I_1, I_2 \idealRel A$ such that $I_1 + I_2 = A$.
   Then the long exact \emph{Mayer--Vietoris sequence} reads as follows:
   \begin{equation}
      \dotsm \to \KOThGr_{n+1}(A) \overset{\bdMV}{\longrightarrow} \KOThGr_n(I_1 \cap I_2) \to \KOThGr_n(I_1) \oplus \KOThGr_n(I_2) \to \KOThGr_n(A) \to \dotsm \label{eq:abstractMV}
   \end{equation}
   To construct this sequence, consider the auxiliary $\Cstar$-algebra $\pathsCstarMV(A; I_1,I_2)$ which consists of paths $f \colon \Rext \to A$ such that $f(-\infty) \in I_{1}$ and $f(+\infty) \in I_2$.
      The inclusion $I_1 \cap I_2 \hookrightarrow \pathsCstarMV(A; I_1,I_2)$ sending $b \in I_1 \cap I_2$ to the constant path at $b$ is a $\KOTh$-isomorphism, see for instance~\cite[Lemma 3.1]{siegel:mayerVietorisAnalyticStructureGroup}.
    Then there is a short exact sequence, $0 \to \suspCstar A \to \pathsCstarMV(A; I_1, I_2) \overset{\ev_{\mp \infty}}{\longrightarrow} I_1 \oplus I_2 \to 0$,
   which induces a long exact sequence
   \begin{equation*}
      \dotsm \to \KOThGr_{n}(\suspCstar A) \to \KOThGr_n(\pathsCstarMV(A; I_1, I_2)) \to \KOThGr_n(I_1 \oplus I_2) \to \KOThGr_{n-1}(\suspCstar A) \to \dotsm,
   \end{equation*}
   which becomes precisely of the form \labelcref{eq:abstractMV}.
   Again, it can be verified that the Mayer--Vietoris sequence is compatible with the external products. In particular, there is a commutative diagram
   \begin{equation*}
   \begin{tikzcd}[row sep=small]
    \KOThGr_{n+1}(A) \tens \KOThGr_m(B) \rar{\bdMV \tens \id} \dar{\KThProd}& \KOThGr_n(I_1 \cap I_2) \tens \KOThGr_m(B) \dar{\KThProd}\\
		    \KOThGr_{n+m+1}(A \tensGr B) \rar{\bdMV} & \KOThGr_{n+m}((I_1 \tensGr B) \cap (I_2 \tensGr B)).
    \end{tikzcd}
   \end{equation*}
   
   \begin{rem}\label{rem:alternativeMVBoundary}
     There is an alternative description of the Mayer--Vietoris boundary map in terms of the boundary map of some short exact sequence.
     Indeed, the inclusion induces an isomorphism $\iota \colon {I_1}/{I_1 \cap I_2} \iso A/{I_2}$.
     Consider the boundary map $\bd_{1} \colon \KTh_{\KPh}({I_1}/{I_1 \cap I_2}) \to \KTh_{\KPh-1}(I_1 \cap I_2)$ associated to $0 \to I_1 \cap I_2 \to I_1 \to {I_1}/{I_1 \cap I_2} \to 0$ and the canonical projection $\pi_2 \colon A \to A/I_2$.
     Then we have
     \begin{equation*}
       \bdMV(x) = \bd_1 (\iota^{-1} \circ \pi_2)_\KPh(x) \in \KTh_{\KPh-1}(I_1 \cap I_2)
     \end{equation*}
     for all $x \in \KTh_\KPh(A)$.
   \end{rem}

\section{Yu's localization algebras} \label{sec:localizationAlgebras}
Yu~\cite{yu:localizationAlgebrasAndCoarseBaumConnes} has introduced the localization algebra $\roeAlgLoc(X)$ to provide an alternative model for the $\KTh$-homology of a proper metric space $X$.
   In fact, there is an isomorphism $\KOTh_n(X) \overset{\iso}{\to} \KOTh_n(\roeAlgLoc(X))$, called the \enquote{local index map}.
   This can be proved for simplicial complexes using a Mayer--Vietoris argument~\cite{yu:localizationAlgebrasAndCoarseBaumConnes}, or alternatively, the general case can be reduced to Paschke duality~\cite{roe-qiao:onTheLocalizationAlgebraOfYu}.
   
   In this section, we review the definition of the localization algebras and introduce a $\cliffAlg_n$-linear version thereof.

\subsection{\texorpdfstring{$\cliffAlg_n$}{Cln}-linear localization algebras} 
   Let $X$ be a proper metric space endowed with an isometric, free and proper action of a countable discrete group $\Gamma$.
   
   A \emph{$\Gamma$-equivariant $X$-module}, or simply \emph{$(X, \Gamma)$-module}, is a Hilbert space $H$ together with a $\ast$-representation $\rho \colon \contZ(X) \to \bndOps(H)$ and a unitary representation $U \colon \Gamma \to \unitary(H)$ such that $\rho(f) U_\gamma = U_\gamma \rho(\gamma^* f) $ for all $\gamma \in \Gamma$ and $f \in \contZ(X)$.
   If $\Gamma$ is the trivial group we refer to it just as an \emph{$X$-module}.
   
    A \emph{$\cliffAlg_n$-linear $\Gamma$-equivariant $X$-module}, or simply \emph{$(X, \Gamma, \cliffAlg_n)$-module}, is defined analogously but we replace the Hilbert space $H$ by a graded Hilbert $\cliffAlg_n$-module $\clnXmodule$ and require that the representations are by even bounded $\cliffAlg_n$-linear operators.\footnote{In the case of $\cliffAlg_n$, an operator is bounded and $\cliffAlg_n$-linear iff it is an adjointable Hilbert-module map.}
    
    An $(X,\Gamma)$-module $H$ is called \emph{ample} if the representation $\rho$ is non-degenerate and $\rho(f)$ is not a compact operator for any $f \in \contZ(X)$, $f \neq 0$.
    In the $\cliffAlg_n$-linear setting, we say an $(X,\Gamma, \cliffAlg_n)$-module is \emph{ample} if it is isomorphic to $H \tensGr \cliffAlg_n$, where $H$ is an ample $(X, \Gamma)$-module.

Let $T \in \bndOps(\clnXmodule_X, \clnXmodule_Y)$, where $X$ and $Y$ are proper metric spaces with corresponding modules $\clnXmodule_X$ and $\clnXmodule_Y$.
The \emph{support} of $T$ is the subset $\supp(T) \subseteq Y \times X$ such that $(y, x) \not\in \supp(T)$ if and only if there exists $f \in \contZ(X)$ and $g \in \contZ(Y)$ such that $g T f = 0$ but $f(x) \neq 0 \neq g(y)$.
In the expression \enquote{$g T f$} we have suppressed the representations of the function algebras to simplify the notation.

The \emph{propagation} $\propag(T)$ of an operator $T \in \bndOps(\clnXmodule_X)$ is the minimal $R \in [0, +\infty]$ such that $\supp(T)$ lies in the $R$-neighborhood of the diagonal in $X \times X$.

An operator $T \in \bndOps(\clnXmodule_X)$ is called \emph{locally compact} if $f T \in \cptOps(\clnXmodule_X) \ni T f$ for all $f \in \contZ(X)$, and \emph{pseudo-local} if $f T - T f \in \cptOps(\clnXmodule_X)$ for all $f \in \contZ(X)$.

Let $A \subset X$ be a subspace. We say that  $T \in \bndOps(\clnXmodule_X)$ is \emph{supported near $A$} if there exists $R \geq 0$ such that $\supp(T) \subseteq \nbh_R(A) \times \nbh_R(A)$, where $\nbh_R(A)$ denotes the $R$-neighborhood of $A$ in $X$.

\begin{defi} \label{defi:locAlgebras}
   Let $\clnXmodule$ be a fixed ample $(X, \Gamma, \cliffAlg_n)$-module. Let $A \subset X$ be a closed $\Gamma$-invariant subset.
   \begin{enumerate}[(1)]
      \item Denote by $\roeAlg(X; \cliffAlg_n)^\Gamma$ the \emph{$\cliffAlg_n$-linear $\Gamma$-equivariant Roe algebra}, that is, the $\Cstar$-algebra generated by all the $\Gamma$-equivariant, bounded, $\cliffAlg_n$-linear operators on~$\clnXmodule$ which are locally compact and of finite propagation. \label{item:defiRoeAlg}
      \item The closure of the set of all operators supported near $A$ with properties as in \labelcref{item:defiRoeAlg} is denoted by $\roeAlg(A \subset X; \cliffAlg_n)^\Gamma$.
      Then $\roeAlg(A \subset X; \cliffAlg_n)^\Gamma$ is an ideal in $\roeAlg(X; \cliffAlg_n)^\Gamma$. \label{item:defiPartialRoeAlg}
      
      \item Denote by $\roeAlgLoc(X; \cliffAlg_n)^\Gamma$ the \emph{$\cliffAlg_n$-linear $\Gamma$-equivariant localization algebra}, that is, the $\Cstar$-subalgebra of $\cont([1, \infty), \roeAlg(X; \cliffAlg_n)^\Gamma)$ generated by the bounded and uniformly continuous functions $L \colon [1,\infty) \to \roeAlg(X;\cliffAlg_n)^\Gamma$ such that the propagation of $L(t)$ is finite for all $t \geq 1$ and tends to zero as $t \to \infty$.
      
      \item \begin{sloppypar}Denote by $\roeAlgLoc(A \subset X; \cliffAlg_n)^\Gamma$ the  closure of the $\ast$-subalgebra of $\roeAlgLoc(X; \cliffAlg_n)^\Gamma$ consisting of those $L \colon [1,\infty) \to \roeAlg(X; \cliffAlg_n)^\Gamma$ such that $\supp(L(t)) \subseteq \nbh_{c(t)}(A\times A)$ for some function $c \colon [1,\infty) \to \Rgeq $ with $c(t) \to 0$ as $t \to \infty$.
      Then $\roeAlgLoc(A~\subset~X;~\cliffAlg_n)^\Gamma$ is an ideal of $\roeAlgLoc(X; \cliffAlg_n)^\Gamma$.\end{sloppypar}
      
      \item Evaluation at $1$ yields a surjective $\ast$-homomorphism $\ev_1 \colon \roeAlgLoc(X; \cliffAlg_n)^\Gamma \to \roeAlg(X; \cliffAlg_n)^\Gamma$, the kernel of which we denote by $\roeAlgLocZ(X; \cliffAlg_n)^\Gamma$.
      More generally, we define $\roeAlgLocP{A}(X; \cliffAlg_n)^\Gamma$ to be the preimage of $\roeAlg(A \subset X; \cliffAlg_n)^\Gamma$ under~$\ev_1$.
           
      \item If $Z \subseteq A \subseteq X$ are two $\Gamma$-invariant closed subsets, then we define the ideal $\roeAlgLocP{Z}(A \subset X; \cliffAlg_n)^\Gamma$ to consist of all the elements $L \in \roeAlgLoc(A \subset X; \cliffAlg_n)^\Gamma$ such that $\ev_1(L) \in \roeAlg(Z \subset X; \cliffAlg_n)^\Gamma$.
   \end{enumerate}
\end{defi}
\begin{rem}[Notation]If we wish to emphasize which $(X, \Gamma, \cliffAlg_n)$-module is used, we write $\roeAlgLoc(X, \clnXmodule; \cliffAlg_n)^\Gamma$.
However, since the $\KTh$-theory of these algebras does not depend on the choice of the ample module (see~\cref{subsec:functoriality}), we will usually not do so.
There are also the simpler versions of the Roe and localization algebras where there is no group action or no Clifford algebra present (take $\Gamma = 1$ or $n = 0$).
We denote these variants by dropping the group or the Clifford algebra in our notation, for instance $\roeAlgLoc(X; \cliffAlg_n)$, $\roeAlgLoc(X)^\Gamma$, $\roeAlgLoc(X)$, $\roeAlgLoc(X, H)$, \dots
\end{rem}

The most important feature of the ideals such as $\roeAlgLoc(A \subset X)$ is that their $\KTh$-theory agrees with the $\KTh$-theory of the corresponding algebra associated with the subspace $A$, see \cref{lem:idealsOfSubspace} below.

In addition, we will occasionally need the \emph{structure algebra}~$\structureAlg(X)^\Gamma$, which is defined to be the $\Cstar$-algebra generated by all $\Gamma$-equivariant, pseudo-local operators of finite propagation.
We also have the corresponding localization algebra version,~$\structureAlgLoc(X)^\Gamma$.

\subsection{Functoriality of localization algebras} \label{subsec:functoriality}

The $\KOTh$-theory of the localization algebra is functorial with respect to uniformly continuous and coarse maps.
In the context of Roe algebras (and similarly, in the Paschke duality picture of $\KOTh$-homology), the induced map on the algebra level is given by conjugation with an isometry which in a suitable sense \enquote{covers} the original map on space level.
Detailed treatments of these ideas can be found in~\cite[Chapters 5 and 6]{higson-roe:analyticKHomology},\ \cite{roe:indexTheoryCoarseGeometryTopologyOfManifolds,siegel:mayerVietorisAnalyticStructureGroup}.
In the case of localization algebras, functoriality is implemented by an appropriate family of covering isometries~\cite{yu:localizationAlgebrasAndCoarseBaumConnes,roe-qiao:onTheLocalizationAlgebraOfYu}.
In this section, we review this construction in a way that is adapted to the $\Gamma$-equivariant and $\cliffAlg_n$-linear setting.

Let $X$ and $Y$ be proper metric spaces endowed with free and proper $\Gamma$-actions.
A proper map $f \colon X \to Y$ is uniformly continuous and coarse (\enquote{\emph{ucc}}) if and only if there exists a non-decreasing function $S \colon \Rgeq \to \Rgeq$ which is continuous at $0$ with $S(0) = 0$ such that $d_Y(f(x_1), f(x_2)) < S(d_X(x_1,x_2))$ for all $x_1,x_2 \in X$.

\begin{defi}\label{defi:coverUCCoarseMap}
Let $f \colon X \to Y$ be a $\Gamma$-equivariant ucc map and fix an $(X,\Gamma, \cliffAlg_n)$-module $\clnXmodule_X$ and a $(Y,\Gamma, \cliffAlg_n)$-module $\clnXmodule_Y$.
We say that a uniformly continuous family of $\Gamma$-equivariant $\cliffAlg_n$-linear even isometries $V_t \colon \clnXmodule_X \to \clnXmodule_Y$, $t \in [1,\infty)$, \emph{covers}~$f$ if $\sup \left\{ d(y, f(x)) \mid (y,x) \in \supp(V_t) \right\} \to 0$ as $t \to \infty$.
\end{defi}

We note that in \cite{yu:localizationAlgebrasAndCoarseBaumConnes,roe-qiao:onTheLocalizationAlgebraOfYu} functoriality is discussed only with respect to proper Lipschitz maps but the construction generalizes directly to the case of ucc maps.
Indeed, extending the arguments from \cite[313]{yu:localizationAlgebrasAndCoarseBaumConnes} to ucc maps and the $\Gamma$-equivariant $\cliffAlg_n$-linear setup yields the following lemmas.

From now on we fix an ample $(X,\Gamma, \cliffAlg_n)$-module $\clnXmodule_X$ and an ample $(Y,\Gamma, \cliffAlg_n)$-module $\clnXmodule_Y$.
\begin{lem}
   Let $f \colon X \to Y$ be a $\Gamma$-equivariant ucc map.
   Then there exists a family of isometries $V_t \colon \clnXmodule_X \to \clnXmodule_Y \oplus \clnXmodule_Y$, $t \in [1,\infty)$, that covers $f$ as in \cref{defi:coverUCCoarseMap}.
\end{lem}

\begin{lem}
   Let $f \colon X \to Y$ be a $\Gamma$-equivariant ucc map and $V_t \colon \clnXmodule_X \to \clnXmodule_Y \oplus \clnXmodule_Y$ be a family of isometries that covers $f$.
   Then conjugation by $V_t$ induces a $\ast$-homomorphism 
   \begin{equation*}
     \Ad_{V_t} \colon \roeAlgLoc(X; \cliffAlg_n)^\Gamma \to \Mat_2( \roeAlgLoc(Y; \cliffAlg_n)^\Gamma),\quad L \mapsto (t \mapsto V_t L_t V_t^*).
   \end{equation*}
   The induced map $\left(\Ad_{V_t}\right)_\KPh \colon \KOThGr_\KPh(\roeAlgLoc(X; \cliffAlg_n)^\Gamma) \to \KOThGr_\KPh(\roeAlgLoc(Y; \cliffAlg_n)^\Gamma)$ does not depend on the choice of the family of isometries $V_t$ covering $f$.
   In particular, the $\KOTh$-theory of $\roeAlgLoc(X; \cliffAlg_n)^\Gamma$ does not depend on the choice of the ample $(X,\Gamma,\cliffAlg_n)$-module up to canonical isomorphism.
\end{lem}

Hence the map $f$ induces a well-defined map on the $\KTh$-theory of the localization algebras, and we will incorporate this fact in our notation by writing $f_\KPh = (\Ad_{V_t})_\KPh$.
Restricting $\Ad_{V_t}$ to the ideal $\roeAlgLocZ(X; \cliffAlg_n)^\Gamma$ yields functoriality for $\KOThGr_\KPh(\roeAlgLocZ(X; \cliffAlg_n)^\Gamma)$.
If $A \subseteq X$ and $B \subseteq Y$ are closed $\Gamma$-invariant subsets such that $f(A) \subseteq B$, we may employ the same procedure to obtain an induced map $f_* \colon \KThGr_\KPh(\roeAlgLocP{A}(X;\cliffAlg_n)^\Gamma) \to \KThGr_\KPh(\roeAlgLocP{B}(Y; \cliffAlg_n)^\Gamma)$.
Moreover, for a family of isometries $V_t$ which covers $f$, the isometry $V_1$ covers $f$ in the coarse sense.
Hence the induced maps on $\KOTh$-theory of the localization algebras are compatible with the induced maps on the $\KOTh$-theory of the Roe algebras.

\begin{rem}
   Let $H$ be an ample $(X,\Gamma)$-module.
   Using the ample $(X, \Gamma, \cliffAlg_n)$-module $\clnXmodule = H \tensGr \cliffAlg_n$, one can directly verify that $\roeAlgLoc(X; \cliffAlg_n)^\Gamma = \roeAlgLoc(X)^\Gamma \tensGr \cliffAlg_n$.
   This means that we have a canonical identification on the level of $\KOTh$-theory,
   \begin{equation*}
    \KOThGr_0(\roeAlgLoc(X; \cliffAlg_n)^\Gamma) = \KOThGr_0(\roeAlgLoc(X)^\Gamma \tensGr \cliffAlg_n) = \KOTh_n(\roeAlgLoc(X)^\Gamma),
   \end{equation*}
   where the latter identification is due to the Bott map from \cref{subsec:bott}.
\end{rem}

\begin{lem} \label{lem:idealsOfSubspace}
 Let $A \subseteq X$ be a closed $\Gamma$-invariant subspace.
 Let $V_t \colon H_A \to H_X$ be a family of isometries which covers the inclusion $A \hookrightarrow X$, where $H_Y$ is an ample $(Y, \Gamma)$-module for $Y \in \{X, A\}$.
 Then conjugation by $V_t$ determines $*$-homomorphisms as follows:
 \begin{align*}
    \Ad_{V_t} \colon \roeAlgLocZ(A)^\Gamma &\to \roeAlgLocZ(A \subset X)^\Gamma, \\
    \Ad_{V_t} \colon \roeAlgLoc(A)^\Gamma &\to \roeAlgLoc(A \subset X)^\Gamma, \\
    \Ad_{V_1} \colon \roeAlg(A)^\Gamma &\to \roeAlg(A \subset X)^\Gamma.
 \end{align*}
 All these $*$-homomorphisms induce isomorphisms on $\KOTh$-theory.
 
 Moreover, if $Z \subseteq A$ is another $\Gamma$-invariant closed subset, then $\Ad_{V_t} \colon \roeAlgLocP{Z}(A)^\Gamma \to \roeAlgLocP{Z}(A \subset X)^\Gamma$ induces an isomorphism on $\KTh$-theory.
\end{lem}
\begin{proof}
        This has been proved for the Roe algebra in~\cite{higson-roe-yu:coarseMVprinciple}.
        For the localization algebra $\roeAlgLoc(A)$, this has been directly established in~\cite{yu:localizationAlgebrasAndCoarseBaumConnes} in the special case of metric simplicial complexes.
        However, one can reduce the general case for $\roeAlgLoc(A)^\Gamma$ to the corresponding statement in (equivariant) $\KTh$-homology.
        Indeed, if we use very ample modules, it is shown in \cite{roe-qiao:onTheLocalizationAlgebraOfYu} that we have isomorphisms,
        \begin{equation*}
        \KOTh_{\KPh+1}\left(\structureAlg(A)^\Gamma/\roeAlg(A)^\Gamma\right) \overset{\iso}{\leftarrow} \KOTh_{\KPh+1}\left(\structureAlgLoc(A)^\Gamma / \roeAlgLoc(A)^\Gamma \right) \overset{\iso}{\to} \KOTh_{\KPh}\left( \roeAlgLoc(A)^\Gamma \right).
        \end{equation*}
        Similarly, we obtain isomorphisms,
                \begin{equation*}
                \begin{split}
        \KOTh_{\KPh+1}\left(\structureAlg(A \subset X)^\Gamma/\roeAlg(A\subset X)^\Gamma\right) \overset{\iso}{\leftarrow} \KOTh_{\KPh+1}\left(\structureAlgLoc(A\subset X)^\Gamma / \roeAlgLoc(A \subset X)^\Gamma \right) \\
        \overset{\iso}{\to} \KOTh_{\KPh}\left( \roeAlgLoc(A \subset X)^\Gamma \right),
        \end{split}
        \end{equation*}
        where the ideals $\structureAlg(A \subset X)^\Gamma$ and $\structureAlgLoc(A \subset X)^\Gamma$ are defined analogously as their counterparts in the setting of Roe algebras, however, with the additional condition that all operators should be locally compact on the complement of $A$.
        Furthermore, we observe that $\Ad_{V_t}$ intertwines these two sequences of isomorphisms, so it would suffice if
        \begin{equation*}
           \Ad_{V_1} \colon \structureAlg(A)^\Gamma / \roeAlg(A)^\Gamma  \to \structureAlg(A \subset X)^\Gamma/\roeAlg(A\subset X)^\Gamma,
        \end{equation*}
        induced isomorphisms on $\KOTh$-theory.
        This, however, is just the corresponding statement in the Paschke duality picture of $\KTh$-homology, which is proved in~\cite[Chapter 5]{higson-roe:analyticKHomology}, see also~\cite[Proposition 3.8]{siegel:mayerVietorisAnalyticStructureGroup}.
        
        Having established the isomorphisms for $\roeAlg(A)^\Gamma$ and $\roeAlgLoc(A)^\Gamma$, one can deduce from the Five Lemma that the map $\Ad_{V_t} \colon \roeAlgLocZ(A)^\Gamma \to \roeAlgLocZ(A \subset X)^\Gamma$ must also be an isomorphism on $\KOTh$-theory.
        Finally, the last statement follows similarly from a Five Lemma argument by considering the short exact sequence
        \begin{equation*}
        0 \to \roeAlgLocZ(A \subset X)^\Gamma \to \roeAlgLocP{Z}(A \subset X)^\Gamma \to \roeAlg(Z \subset X)^\Gamma \to 0.\qedhere
        \end{equation*}
\end{proof}

\subsection{External product} \label{subsec:externalProductLocAlgebras}
Let $X_i$ be proper metric spaces endowed with proper and free $\Gamma_i$-actions, $i=1,2$.
Suppose that $\clnXmodule_i$ is a $\Gamma_i$-equivariant $\cliffAlg_{n_i}$-linear ample $X_i$-module, $i=1,2$.
Then $\clnXmodule := \clnXmodule_1 \tensGr \clnXmodule_2$ is an ample $\Gamma$-equivariant $\cliffAlg_n$-linear $X$-module for $\Gamma := \Gamma_1 \times \Gamma_2$ and $n := n_1 + n_2$.
Then there is a canonical $*$-homomorphism
\begin{gather*}
\tensMap \colon \roeAlgLoc(X_1; \cliffAlg_{n_1})^{\Gamma_1} \tensGr \roeAlgLoc(X_2; \cliffAlg_{n_2})^{\Gamma_2} \to \roeAlgLoc(X_1 \times X_2; \cliffAlg_{n})^{\Gamma_1 \times \Gamma_2}, \\
 \tensMap(L_1 \tensGr L_2)(t) = L_1(t) \tensGr L_2(t), \quad L_i \in \roeAlgLoc(X_i; \cliffAlg_{n_i})^{\Gamma_{i}}, t \in [1,\infty].
\end{gather*}
If $Z_1 \subset X_1$ is a $\Gamma_1$-invariant closed subset, then $\tensMap$ restricts to a map 
\begin{equation*}
  \tensMap \colon \roeAlgLocP{Z_1}(X_1; \cliffAlg_{n_1})^{\Gamma_1} \tensGr \roeAlgLoc(X_2; \cliffAlg_{n_2})^{\Gamma_2} \to \roeAlgLocP{Z_1 \times X_2}(X_1 \times X_2; \cliffAlg_n)^{\Gamma_1 \times \Gamma_2}.
\end{equation*}
Combining this with the external product in $\KTh$-theory, we obtain the following external product for the $\KTh$-theory of localization algebras,
\begin{equation*}
  \KTh_{n_1}\left(\roeAlgLocP{Z_1}(X_1)^{\Gamma_1} \right) \tens \KTh_{n_2}\left(\roeAlgLoc(X_2)^{\Gamma_2}\right) \overset{\secProd}{\to} \KTh_{n_1 + n_2}\left(\roeAlgLocP{Z_1 \times X_2}(X_1 \times X_2)^{\Gamma_1 \times \Gamma_2} \right),
\end{equation*}
defined as $x \secProd y := \tensMap_\KPh(x \KThProd y)$ for $x \in \KThGr_0(\roeAlgLocP{Z_1}(X_1; \cliffAlg_{n_1})^{\Gamma_1})$, $y \in \KThGr_0(\roeAlgLoc(X_2; \cliffAlg_{n_2})^{\Gamma_2})$.

The external product is functorial, that is, $(f_1 \times f_2)_\KPh(x \secProd y) = (f_1)_\KPh(x) \secProd (f_2)_\KPh(y)$ for ucc maps $f_i \colon X_i \to Y_i$,  $i=1,2$.

Restricting the construction to $t = 1$ yields $\tensMap_1 \colon \roeAlg(X_1; \cliffAlg_{n_1})^{\Gamma_1} \tens \roeAlg(X_2; \cliffAlg_{n_2})^{\Gamma_2} \to \roeAlg(X_1 \times X_2; \cliffAlg_{n_1 + n_2})^{\Gamma_1 \times \Gamma_2}$ and we obtain an external product for Roe algebras,
\begin{equation*}
 \KTh_{n_1}(\roeAlg(X_1)^{\Gamma_1}) \tens \KTh_{n_2}(\roeAlg(X_2)^{\Gamma_2}) \overset{\secProd}{\to} \KTh_{n_1 + n_2}(\roeAlg(X_1 \times X_2)^{\Gamma_1 \times \Gamma_2}).
\end{equation*}

\section{Local index classes, secondary invariants and product formulas} \label{sec:localIndexProductFormulas}
In this section, we combine the $\cliffAlg_n$-linear localization algebras with the picture of $\KTh$-theory from~\cref{sec:gradedKtheory} in order to construct the (equivariant) local index classes associated to the spinor Dirac operators on complete spin manifolds.
Indeed, an element of $\KThGr_0(\roeAlgLoc(X; \cliffAlg_n))$ may be defined by a $\ast$-homomorphism $\contZGr \to \roeAlgLoc(X; \cliffAlg_n)$.
Our definition of the local index class is essentially the $\ast$-homomorphism given by the functional calculus of the $\cliffAlg_n$-linear spinor Dirac operator, see~\labelcref{eq:defLocalIndexClass} below.
A slight modification of this idea allows to define the (partial) secondary invariants in almost the same fashion, see~\cref{subsec:partialIndex}.

\subsection{The local index class} \label{subsec:localIndexClasses}

Let $X$ be a spin manifold together with a fixed complete Riemannian metric $g$.
Suppose that $X$ is equipped with a free and proper action of a discrete group $\Gamma$ by spin structure preserving isometries. 
Let $\clnDiracOp$ the $\cliffAlg_n$-linear Dirac operator acting on the $\cliffAlg_n$-spinor bundle $\clnSpinorBdl(X)$. 
That is, $\clnSpinorBdl(X) = \mathrm{P}_{\Spin(n)}(X) \times_l \cliffAlg_n$, where $\mathrm{P}_{\Spin(n)}(X)$ is the principal $\Spin(n)$-bundle of $X$ and $l$ is the representation of $\Spin(n)$ by left multiplication on $\cliffAlg_n$.
Right multiplication induces an action of $\cliffAlg_n$ on $\clnSpinorBdl(X)$.
This turns sections of $\clnSpinorBdl(X)$ into a right $\cliffAlg_n$-module and the space of $\Lp^2$-sections $\Lp^2(\clnSpinorBdl(X))$ into a Hilbert $\cliffAlg_n$-module.
Moreover, $\Lp^2(\clnSpinorBdl(X))$ is equipped with a unitary $\Gamma$-action such that it is a $(X, \Gamma, \cliffAlg_n)$-module.
Using a measurable trivialization of $\clnSpinorBdl(X)$ one can show that $\Lp^2(\clnSpinorBdl(X))$ is isomorphic to $\Lp^2(X) \tensGr \cliffAlg_n$, hence it is ample.
In the following, we consider the localization algebra to be formed on $\Lp^2(\clnSpinorBdl(X))$.

Consider the $*$-homomorphism,
\begin{equation}
 \varphi_\clnDiracOp \colon \contZGr \to \roeAlgLoc(X; \cliffAlg_n)^{\Gamma}, \quad \varphi_\clnDiracOp(f)(t) = f\left( \frac{1}{t} \clnDiracOp \right) \label{eq:defLocalIndexClass}
\end{equation}
By \cite[Proposition 10.5.2]{higson-roe:analyticKHomology}, $f( \frac{1}{t} \clnDiracOp)$ is locally compact for each $t$.
Moreover, if $f$ is a function with compactly supported Fourier transform, then the support of the Fourier transform of $x \mapsto f(t^{-1}x)$ becomes more and more concentrated at $0$ as $t$ grows.
Thus it follows from unit propagation speed of Dirac wave operators and the Fourier inversion formula (see~\cite[Chapter 10.3]{higson-roe:analyticKHomology}) that the propagation of $f( \frac{1}{t} \clnDiracOp )$ goes to zero as $t \to \infty$.
By an approximation argument all of this implies that $\varphi_\clnDiracOp(f) \in \roeAlgLoc(X; \cliffAlg_n)^\Gamma$ for all $f \in \contZGr$.
Since $\clnDiracOp$ is odd, $\varphi_{\clnDiracOp}$ is a \emph{graded} $\ast$-homomorphism.

\begin{defi} \label{defi:localIndexClass}
 The \emph{equivariant local index class} of the Dirac operator $\clnDiracOp$ is defined to be the $\KOTh$-theory class represented by the $\ast$-homomorphism $\varphi_{\clnDiracOp}$ above,
 \begin{equation*}
 \IndL^\Gamma(\clnDiracOp) = [\varphi_\clnDiracOp] \in \KOThGr_0\left(\roeAlgLoc(X; \cliffAlg_n)^{\Gamma} \right) = \KOTh_n\left(\roeAlgLoc(X)^\Gamma\right).
\end{equation*}
\end{defi}

\begin{rem}[{\cite[Section 4]{roe-qiao:onTheLocalizationAlgebraOfYu}, \cite[Section 3]{dumitrascu:asymptoticMorphismKHomology}}] \label{rem:relationToETheoryFundamentalClass}
  There is an asymptotic morphism $\gamma \colon \roeAlgLoc(X)^\Gamma \tensGr \contZ(X) \asympMorphism \cptOps$, $\gamma_t(L \tensGr f) = L(t) f$.
  Using a suitable product in $\ETh$-theory, this induces a map $\gamma_* \colon \KOTh_\KPh( \roeAlgLoc(X)^\Gamma) \to \KOTh^{-\KPh}_\Gamma(\contZ(X)) = \KOTh_\KPh^\Gamma(X)$.
  The class $\gamma_*(\IndL(\clnDiracOp))$ is explicitly represented by the following $\Gamma$-equivariant asymptotic morphism $\alpha \colon \contZGr \tensGr \contZ(X) \asympMorphism \cptOps(\Lp^2(\clnSpinorBdl)) \iso \cptOps \tensGr \cliffAlg_n$, $\alpha_t(f \tensGr g) = f( \frac{1}{t} \clnDiracOp) g$.
  Specializing to $X = \R^n$, we observe that $\gamma_*(\IndL(\clnDiracOp_{\R^n}))$ is precisely the Dirac element from \cref{subsec:bott}.
\end{rem}

\subsection{Partial secondary invariants and localized indices} \label{subsec:partialIndex}
We now assume, in addition to the previous setup, that the Riemannian metric $g$ on $X$ has uniformly positive scalar curvature \emph{outside} a $\Gamma$-invariant closed subset $Z \subset X$.
\begin{defi}
  We say a Riemannian metric $g$ on $X$ has \emph{uniformly positive scalar curvature (\enquote{upsc}) outside $Z \subset X$} if there is $\varepsilon > 0$ such that $\scalCurv_g(x) > \varepsilon$ for all $x \in X \setminus Z$.
\end{defi}
We will use the following lemma, proofs of which can be found in~\cite{Roe15:Partial,HPS14Codimension}.
\begin{lem} \label{lem:roesLemmaLocalizedIndex}
Let $\varepsilon >0$ such that the scalar curvature function $\scalCurv_g$ is uniformly bounded below by $4 \varepsilon^2$ on $X \setminus Z$.
 Then $(\ev_1 \circ \varphi_{\clnDiracOp})(f) = f(\clnDiracOp) \in \roeAlg(Z \subset X; \cliffAlg_n)^\Gamma$ for all $f \in \contZGr(-\varepsilon, \varepsilon)$.
\end{lem}
In other words, the restriction of $\varphi_{\clnDiracOp}$ to $\contZGr(-\varepsilon, \varepsilon)$ takes values in $\roeAlgLocP{Z}(X; \cliffAlg_n)^\Gamma$.
Since the inclusion of $\contZGr(-\varepsilon, \varepsilon)$ into $\contZGr$ is a homotopy equivalence, this means that the local index class $\IndL(\clnDiracOp) \in \KTh_n(\roeAlgLoc(X)^\Gamma)$ can be lifted to an element of $\KTh_n(\roeAlgLocP{Z}(X)^\Gamma)$.
To make this precise, we fix a graded $*$-homomorphism $\psi \colon \contZGr \to \contZGr(-\varepsilon, \varepsilon)$ which is homotopic to the identity when composed with the inclusion $\contZGr(-\varepsilon, \varepsilon) \hookrightarrow \contZGr$, see~\cref{lem:elementaryHomotopyEqu}.

\begin{defi} \label{defi:LocalizedLocalIndexClass}
 The \emph{partial secondary local index class} of the Dirac operator $\clnDiracOp$  is defined as follows,
 \begin{equation*}
 \IndLP{Z}^\Gamma (\clnDiracOp^g) := [\varphi_{\clnDiracOp} \circ \psi] \in \KOThGr_0 \left( \roeAlgLocP{Z}(X; \cliffAlg_n)^\Gamma \right) = \KOTh_n \left( \roeAlgLocP{Z}(X)^\Gamma \right),
\end{equation*}
 where $\psi \colon \contZGr \to \contZGr(-\varepsilon, \varepsilon)$ and $\varepsilon > 0$ are chosen as above.
\end{defi}
Here we have included $g$ in the notation of the partial secondary invariant to emphasize that this class depends on the metric.

 By \cref{lem:elementaryHomotopyEqu}, such a $\psi$ exists and is unique up to homotopy (it is just a homotopy inverse to the inclusion $\contZGr(-\varepsilon, \varepsilon) \hookrightarrow \contZGr$, hence $\IndLP{Z}^\Gamma (\clnDiracOp^g)$ is well-defined independently of the choice of appropriate $\varepsilon$ and $\psi$.
 
The partial secondary local index $\IndLP{Z}^\Gamma (\clnDiracOp^g)$ maps to the local index $\IndL^\Gamma(\clnDiracOp)$ under the map induced by the inclusion $\roeAlgLocP{Z}(X)^\Gamma \hookrightarrow \roeAlgLoc(X)^\Gamma$ because $\psi$ is homotopic to the identity on $\contZGr$.
In particular, for $Z = X$, we recover the local index class as defined in previous subsection.

The image of $\IndLP{Z}^\Gamma(\clnDiracOp^g)$ under $(\ev_1)_\KPh \colon \KTh_n(\roeAlgLocP{Z}(X)^\Gamma) \to \KTh_n(\roeAlg(Z \subset X)^\Gamma)$, denote it by $\Ind_Z^\Gamma(\clnDiracOp^g)$, is (an equivariant version of) the \emph{localized coarse index} of Roe, see~\cite[Proposition 3.11]{roe:indexTheoryCoarseGeometryTopologyOfManifolds} and \cite{Roe15:Partial}.
\begin{rem}
  If the $\Gamma$-action on $X$ is cocompact, then any non-empty $\Gamma$-invariant subset $Z$ is coarsely equivalent to $X$ and so the algebras $\roeAlgLocP{Z}(X)^\Gamma$ are all the same for non-empty $Z$.
  Moreover, note that it is a consequence of the Kazdan--Warner theorem~\cite{KW75Scalar} that on a closed manifold (such as $X/\Gamma$ if the action is cocompact) of dimension $\geq 3$ there are no restrictions against positive scalar curvature outside a subset of non-empty interior.
  We conclude that partial secondary local indices for subsets other than $Z = X$ or $Z = \emptyset$ are only interesting if $X$ is not $\Gamma$-cocompact.
\end{rem}
The partial secondary local index class can be used to distinguish metrics up to concordance relative to $Z$ in the following sense.

Given a Riemannian metric $g$ on a smooth manifold $X$, we denote the induced distance function on $X$ by $d_g \colon X \times X \to \Rgeq$.
We say two Riemannian metrics $g_0$ and $g_1$ are \emph{uniformly equivalent} if the identity maps $\id \colon (X, d_{g_0}) \to (X, d_{g_1})$ and $\id \colon (X, d_{g_1}) \to (X, d_{g_0})$ are uniformly continuous.
Observe that uniformly equivalent Riemannian metrics are also coarsely equivalent since the distance function is a length metric.

\begin{defi}\label{defi:concordanceRelZ}
 Let $g_0$, $g_1$ be two $\Gamma$-invariant Riemannian metrics on $X$ which have upsc outside $Z \subset X$.
 Suppose that $g_0$ and $g_1$ are uniformly equivalent.
 We say $g_0$ and $g_1$ are \emph{concordant relative to $Z$} if there exists a metric $h$ on $\R \times X$ such that
 \begin{enumerate}[(i)]
 \item the identity map $\id \colon (\R \times X, d_h) \to (\R \times X, d_{\D{t}^2 \oplus g_0})  $ is uniformly continuous,
 \item $h$ has upsc outside $\R \times Z$,
 \item $h$ restricts to $\D{t}^2 \oplus g_1$ on $[1,\infty) \times X$ and to $\D{t}^2 \oplus g_0$ on $(-\infty,0] \times X$.
 \end{enumerate}
\end{defi}
If $Z = \emptyset$, we just say that $g_0$ and $g_1$ are concordant.
\begin{rem}\label{rem:equalRelZ}
 Using a convex interpolation one can show that if $g_0$ and $g_1$ are \emph{equal} outside $Z$, then they are concordant relative to $Z$.
\end{rem}
\begin{rem}
  We restrict ourselves to comparing uniformly equivalent Riemannian metrics so as to ensure that the Roe and localization algebras do not depend on which metric we use to define them.
  Moreover, completeness is preserved under passing to a uniformly equivalent metric.
  Note that two Riemannian metrics are automatically uniformly equivalent if both are invariant with respect to a proper and cocompact group action. 
\end{rem}

\begin{prop}\label{thm:concordanceInvariance}
 Let $g_0$, $g_1$ be concordant relative to $Z$ as in \cref{defi:concordanceRelZ}.
 Then their partial secondary local index classes agree, $\IndLP{Z}^\Gamma(\clnDiracOp^{g_0}) = \IndLP{Z}^\Gamma(\clnDiracOp^{g_1})$.
\end{prop}
We defer the proof of \cref{thm:concordanceInvariance} to \cref{sec:partitionedManifold}, where it will be a consequence of the partitioned manifold index theorem for partial secondary invariants.

\subsection{The \texorpdfstring{$\rho$}{rho}-invariant of a positive scalar curvature metric} \label{subsec:rhoInvariant}
If the Riemannian metric $g$ has uniformly positive scalar curvature on \emph{all of $X$}, then by the previous subsection, we obtain a lift of the local index class to an element $\IndLP{\emptyset}^\Gamma(\clnDiracOp^g) \in \KTh_n(\roeAlgLocP{\emptyset}(X)^\Gamma)$.
Of course, $\roeAlgLocP{\emptyset}(X)^\Gamma = \roeAlgLocZ(X)^\Gamma$, and we take this element to be the $\rho$-invariant associated to the metric $g$.
\begin{defi} \label{defi:secondaryInvariant}
The equivariant \emph{$\rho$-invariant} of the upsc metric $g$ is
 \begin{equation*}
  \secInv^\Gamma(g) := \IndLP{\emptyset}^\Gamma(\clnDiracOp^g) = [\varphi_{\clnDiracOp} \circ \psi] \in \KOThGr_0 \left( \roeAlgLocZ(X; \cliffAlg_n)^\Gamma \right) = \KOTh_n \left( \roeAlgLocZ(X)^\Gamma \right)
 \end{equation*}
 for appropriate $\psi \colon \contZGr \to \contZGr(-\varepsilon, \varepsilon)$ and $\varepsilon > 0$ as in \cref{defi:LocalizedLocalIndexClass}.
\end{defi}
  
  All we need to define such a secondary invariant in $\KTh_n(\roeAlgLocZ(X))$ is that the spectrum of $\clnDiracOp$ does not contain zero.
  
\subsection{The relative index of two positive scalar curvature metrics} \label{subsec:twoPSCrelIndex}
  Let $g_0, g_1$ be two uniformly equivalent complete $\Gamma$-invariant metrics of upsc on $X$.
  Then one can construct a \emph{higher relative index} $\IndRel^\Gamma(g_0, g_1) \in \KTh_{n+1}(\roeAlg(X)^\Gamma)$, see~\cite{XY14Relative}.
  
  By \cref{rem:equalRelZ}, $g_0$ and $g_1$ are concordant relative to $X$.
  Let $h$ be a metric on $W = \R \times X$ which witnesses this (see~\cref{defi:concordanceRelZ}).
  In particular, $h$ restricts to $\D{t}^2 \oplus g_1$ on $[1,\infty) \times X$ and to $\D{t}^2 \oplus g_0$ on $(-\infty,0] \times X$.
  Thus it has upsc outside $Z = [0,1] \times X$.
  %Construct the Dirac operator $\clnDiracOp_W^h$ on $W$ using the metric $h$.
  Consequently, by \cref{subsec:partialIndex}, we obtain a partial secondary local index $\IndLP{[0,1] \times X}^\Gamma(\clnDiracOp_W^h) \in \KTh_{n+1}(\roeAlgLocP{[0,1] \times X}(W)^\Gamma)$ and a localized coarse index $\Ind^\Gamma_{[0,1] \times X}(\clnDiracOp_W^h) \in \KTh_{n+1}(\roeAlg([0,1] \times X \subset W)^\Gamma)$.
  By \cref{rem:equalRelZ,thm:concordanceInvariance}, these classes are independent of the particular choice of $h$.
  Since the canonical projection $\proj_X \colon [0,1] \times X \to X$ is a coarse equivalence, there is an isomorphism
  \begin{equation*}
    \KTh_{n+1}\left(\roeAlg([0,1] \times X \subset W )^\Gamma \right) \iso \KTh_{n+1}\left( \roeAlg([0,1] \times X)^\Gamma \right) \overset{(\proj_X)_*}{\to} \KTh_{n+1}\left( \roeAlg(X)^\Gamma \right).
  \end{equation*}
  After these preparations, we are ready to define the relative index.
  \begin{defi} \label{defi:relIndex}
    The equivariant \emph{relative index} associated to the pair of metrics~$g_0$,~$g_1$ is:
    \begin{equation*}
      \IndRel^\Gamma(g_0, g_1) := (\proj_X)_\KPh \Ind^\Gamma_{[0,1] \times X} (\clnDiracOp_W^h) \in \KTh_{n+1}(\roeAlg(X)^\Gamma).
    \end{equation*}
  \end{defi}
  If $g_0$ and $g_1$ are concordant, then by construction $\IndRel^\Gamma(g_0, g_1) = 0$.
  \begin{rem}
    If $g_0$ and $g_1$ have upsc only outside some $Z$ as in \cref{subsec:partialIndex}, then it is possible to define a relative index $\Ind_{\mathrm{rel}, Z}^\Gamma(g_0, g_1) \in \KTh_{n+1}(\roeAlg(X)^\Gamma / \roeAlg(Z \subset X)^\Gamma)$ but we will not pursue this in this paper.
  \end{rem}

\subsection{Product formulas} \label{subsec:productFormulas}
Suppose that $X_i$, $i \in \{1,2\}$, are $n_i$-dimensional spin manifolds endowed with complete Riemannian metrics $g_i$ and free and proper $\Gamma_i$-actions. 
Consider their product $X = X_1 \times X_2$ which is an $n := n_1 + n_2$-dimensional spin manifold.
A principal $\Spin(n)$-bundle covering the $\SO(n)$-frame bundle of $X$ may be obtained as the bundle associated to $\principalBdl_{\Spin(n_1)}(X_1) \times_{\Z_2} \principalBdl_{\Spin(n_2)}(X_2)$ via the inclusion $\Spin(n_1) \times_{\Z_2} \Spin(n_2) \hookrightarrow \Spin(n)$.
In view of the isomorphism $\cliffAlg_{n_1} \tensGr \cliffAlg_{n_2} \iso \cliffAlg_{n}$, this implies that we may identify the $\cliffAlg$-spinor bundles as follows: $\clnSpinorBdl(X) = \proj_1^* \clnSpinorBdl(X_1) \tensGr \proj_2^* \clnSpinorBdl(X_2)$, where $\proj_i \colon X_1 \times X_2 \to X_i$ are the canonical projection maps.
On the level of $\Lp^2$-sections, this means that we have an identification $\Lp^2(\clnSpinorBdl(X_1)) \tensGr \Lp^2(\clnSpinorBdl(X_2)) = \Lp^2(\clnSpinorBdl(X))$.
Hence we can use the description of the external product from \cref{subsec:externalProductLocAlgebras} in this context.
The $\cliffAlg$-linear Dirac operators $\clnDiracOp_{X_i}$ and $\clnDiracOp_X$ on $\clnSpinorBdl(X_i)$ and $\clnSpinorBdl(X)$, $i \in \{1,2\}$, respectively, satisfy the relation $\clnDiracOp_X = \clnDiracOp_{X_1} \tensGr \id + \id \tensGr \clnDiracOp_{X_2}$.

\thmProductFormula

Moreover, setting $Z_1 = X_1$ and $Z_1 = \emptyset$, we immediately deduce:
\begin{cor} \label{cor:productFormula}
  In particular,
 \begin{equation}
  \IndL^{\Gamma_1 \times \Gamma_2}(\clnDiracOp_X) = \IndL^{\Gamma_1}(\clnDiracOp_{X_1}) \secProd \IndL^{\Gamma_2}(\clnDiracOp_{X_2}). \label{eq:primaryProductFormula}
 \end{equation}
 If both of the Riemannian metrics $g_{X_1}$ on $X_1$ and $g_X = g_{X_1} \oplus g_{X_2}$ on $X_1 \times X_2$ have upsc, then
  \begin{equation}
  \secInv^{\Gamma_1 \times \Gamma_2}(g_X) = \secInv^{\Gamma_1}(g_{X_1}) \secProd \IndL^{\Gamma_2}(\clnDiracOp_{X_2}). \label{eq:secondaryProductFormula}
 \end{equation}
\end{cor}
The main part of the proof of \cref{thm:productFormula} consists of some standard computations which we indicate in the following two lemmas for the convenience of the reader.
\begin{lem}[{\cite[Appendix A.4]{higson-kasparov-trout:ABottPeriodicityTheoremForInfiniteDimensionalEuclideanSpace}}] \label{lem:expOfGradedProductOperators}
 Let $D_i \colon \hilbert_i \supseteq \domain(D_i) \to \hilbert_i$ be odd, \parensup{unbounded} self-adjoint operators on graded Hilbert spaces $\hilbert_i$, $i \in \{1,2\}$.
 Then we have the equality \parensup{of bounded operators on $\hilbert_1 \tensGr \hilbert_2$},
 \begin{equation}
  \exp(-(D_1 \tensGr \id + \id \tensGr D_2)^2) = \exp(-D_1^2) \tensGr \exp(-D_2^2). \label{eq:expOfGradedProductOperators}
 \end{equation}
\end{lem}
\begin{proof}
If both operators $D_1$ and $D_2$ are bounded, then the result follows immediately by the functional equation for the exponential function.
The general case can be reduced to the bounded case using the spectral theorem and an approximation argument.
\end{proof}
Below we work in the setup of \cref{thm:productFormula} and abbreviate $\clnDiracOp_{X_i}$ by $\clnDiracOp_i$ and $\clnDiracOp_X$ by $\clnDiracOp$.
\begin{lem} \label{lem:productFormula}
 The following diagram commutes:
 \begin{equation*}
 \begin{tikzcd}[row sep=small, column sep=10ex]
 \contZGr \rar{\varphi_{\clnDiracOp}} \dar{\comult} &  \roeAlgLoc(X;\cliffAlg_{n})^{\Gamma_1 \times \Gamma_2} \\
 \contZGr \tensGr \contZGr \rar{\varphi_{\clnDiracOp_1} \tensGr \varphi_{\clnDiracOp_2}} & \roeAlgLoc(X_1; \cliffAlg_{n_1})^{\Gamma_1} \tensGr \roeAlgLoc(X_2; \cliffAlg_{n_2})^{\Gamma_2}. \uar{\tensMap}
 \end{tikzcd}
 \end{equation*}
 \begin{proof}
  \begin{sloppypar}
  The statement and proof are essentially the same as, for example, in the proof of~\cite[Theorem 4.1]{dumitrascu:asymptoticMorphismKHomology}.
  It suffices to check commutativity for the generators $\eu^{-{\mathrm{x}}^2}$ and $\mathrm{x} \eu^{-\mathrm{x}^2}$ of $\contZGr$.
  Indeed, we have $\varphi_{\clnDiracOp}(\eu^{-\mathrm{x}^2})(t) = \eu^{- \frac{1}{t^2} \clnDiracOp^2} = \eu^{- \left( t^{-1} \clnDiracOp_1 \tensGr 1 + 1 \tensGr t^{-1} \clnDiracOp_2 \right)^2}$ and $\tensMap ( \varphi_{\clnDiracOp_1} \tensGr \varphi_{\clnDiracOp_2} (\comult(\eu^{-\mathrm{x}^2})) )(t) = \eu^{-\frac{1}{t^2} \clnDiracOp_1^2} \tensGr \eu^{-\frac{1}{t^2} \clnDiracOp_2^2}$.
  Thus, \cref{lem:expOfGradedProductOperators} with $D_i = \frac{1}{t} \clnDiracOp_i$ implies that $\varphi_{\clnDiracOp}(\eu^{-\mathrm{x}^2}) = \tensMap ( \varphi_{\clnDiracOp_1} \tensGr \varphi_{\clnDiracOp_2} (\comult(\eu^{-\mathrm{x}^2})))$, as required.
  A similar computation shows commutativity on the generator $\mathrm{x} \eu^{-\mathrm{x}^2}$.\qedhere
  \end{sloppypar}
 \end{proof}
\end{lem}

\begin{proof}[Proof of \cref{thm:productFormula}]
   By the assumptions we can find $\varepsilon > 0$ such that $\varphi_{\clnDiracOp_1}$ maps $\contZGr(-\varepsilon, \varepsilon)$ to $\roeAlgLocP{Z_1}(X_1; \cliffAlg_{n_1})$ and $\varphi_{\clnDiracOp}$ maps $\contZGr(-\varepsilon, \varepsilon)$ to $\roeAlgLocP{Z}(X; \cliffAlg_n)^{\Gamma_1 \times \Gamma_2}$.
    Choose a graded $*$-homomorphism $\psi \colon \contZGr \to \contZGr(-\varepsilon, \varepsilon)$ homotopy inverse to the inclusion $\iota \colon \contZGr(-\varepsilon, \varepsilon) \hookrightarrow \contZGr$ (as in \cref{defi:secondaryInvariant}).
    
 We then have $\IndLP{Z}^{\Gamma_1 \times \Gamma_2}(\clnDiracOp^g) = [\varphi_\clnDiracOp \circ \psi] \in \KTh_0(\roeAlgLocP{Z}(X; \cliffAlg_{n})^{\Gamma_1 \times \Gamma_2})$, $\IndLP{Z_1}^{\Gamma_1}(\clnDiracOp_1^{g_1}) = [\varphi_{\clnDiracOp_1} \circ \psi] \in \KTh_0(\roeAlgLocP{Z_1}(X_1; \cliffAlg_{n_1})^{\Gamma_1})$ and $\IndL^{\Gamma_2}(\clnDiracOp_2) = [\varphi_{\clnDiracOp_2}] = [\varphi_{\clnDiracOp_2} \circ \psi] \in \KTh_0(\roeAlgLoc(X_2; \cliffAlg_{n_2})^{\Gamma_2})$.    
   Thus, to prove the product formula, we need to show that the following diagram commutes up to homotopy:
   \begin{equation*}
      \begin{tikzcd}[row sep=small, column sep=large]
         \contZGr \rar{\psi} \dar{\comult} & \contZGr(-\varepsilon, \varepsilon) \rar{\varphi_\clnDiracOp} \dar{\comult} & \roeAlgLocP{Z}(X; \cliffAlg_n)^{\Gamma_1 \times \Gamma_2} \\
         \contZGr \tensGr \contZGr \rar{\psi \tensGr \psi} & \contZGr(-\varepsilon, \varepsilon) \tensGr \contZGr(-\varepsilon, \varepsilon) \rar{\varphi_{\clnDiracOp_1} \tensGr \varphi_{\clnDiracOp_2}} & \roeAlgLocP{Z_1}(X_1; \cliffAlg_{n_1})^{\Gamma_1} \tensGr \roeAlgLoc(X_2; \cliffAlg_{n_2})^{\Gamma_2}. \uar{\tensMap}
      \end{tikzcd}
   \end{equation*}
   Indeed, the left square commutes up to homotopy by \cref{cor:compressionCommutesWithComult}.
   The right square strictly commutes as it is a restriction of the diagram from \cref{lem:productFormula}.
\end{proof}

\begin{cor} \label{cor:productFormulaRelIndex}
  Let $X = X_1 \times X_2$ be as before.
  Now suppose that there are two upsc metrics $g_{1,0}$ and $g_{1,1}$ on $X_1$ which are in the same uniform equivalence class and such that $g_i = g_{1,i} \oplus g_{2}$ has upsc on $X$ for $i=0,1$.
  Then,
  \begin{equation}
    \IndRel^{\Gamma_1 \times \Gamma_2}(g_0, g_1) =  \IndRel^{\Gamma_1}(g_{1,0}, g_{1,1}) \secProd \Ind^{\Gamma_2}(\clnDiracOp_2).
  \end{equation}
\end{cor}
\begin{proof}
  By \cref{subsec:twoPSCrelIndex}, 
  \begin{align*}
    \IndRel^{\Gamma_1}(g_{1,0}, g_{1,1}) &= (\proj_{X_1})_\KPh \Ind^{\Gamma_1}_{[0,1] \times X_1}(\clnDiracOp_{\R \times X_1}^{h}) \\
    &= (\proj_{X_1})_\KPh (\ev_1)_\KPh \IndLP{[0,1] \times X_1}^{\Gamma_1}\left( \clnDiracOp_{\R \times X_1}^h \right),
  \end{align*}
  where, to define $\clnDiracOp_{\R \times X_1}^h$, we use an appropriate metric $h$ on $\R \times X_1$ interpolating between $g_{1,0}$ and $g_{1,1}$ on $[0,1] \times X_1$.
  
  Consider the maps
  \begin{align*}
      \tensMap &\colon \roeAlgLocP{[0,1] \times X_1}(\R \times X_1)^{\Gamma_1} \tens \roeAlgLoc(X_2)^{\Gamma_2} \to \roeAlgLocP{[0,1] \times X_1 \times X_2}(\R \times X_1 \times X_2)^{\Gamma_1 \times \Gamma_2}, \\
      t_1^\prime &\colon \roeAlg([0,1] \times X_1 )^{\Gamma_1} \tens \roeAlg(X_2)^{\Gamma_2} \to \roeAlg([0,1] \times X_1 \times X_2)^{\Gamma_1 \times \Gamma_2}, \\
      t_1 &\colon \roeAlg(X_1)^{\Gamma_1} \tens \roeAlg(X_2)^{\Gamma_2} \to \roeAlg(X_1 \times X_2)^{\Gamma_1 \times \Gamma_2}.
  \end{align*}
  defined as in \cref{subsec:externalProductLocAlgebras}.
  Then
  \begin{align*}
    \IndRel^{\Gamma_1 \times \Gamma_2}(g_{0}, g_{1}) &= (\proj_{X})_\KPh (\ev_1)_\KPh \IndLP{[0,1] \times X}^{\Gamma_1}\left( \clnDiracOp_{\R \times X}^{h \oplus g_2} \right), \\
    \intertext{and using the product formula \labelcref{eq:generalProductFormula},}
   &= (\proj_X)_\KPh (\ev_1)_\KPh \tensMap_\KPh \left(\IndLP{[0,1] \times X_1}^{\Gamma_1}\left( \clnDiracOp_{\R \times X_1}^h  \right) \KThProd \IndL^{\Gamma_2}(\clnDiracOp_2) \right) \\
   &= (\proj_X)_\KPh {(\tensMap_1^\prime)}_\KPh \left(\Ind_{[0,1] \times X_1}^{\Gamma_1}\left( \clnDiracOp_{\R \times X_1}^h  \right) \KThProd \Ind^{\Gamma_2}(\clnDiracOp_2) \right)\\
   &= (\tensMap_1)_\KPh \left(\left( (\proj_{X_1})_\KPh\Ind_{[0,1] \times X_1}^{\Gamma_1}\left( \clnDiracOp_{\R \times X_1}^h  \right) \right) \KThProd \Ind^{\Gamma_2}(\clnDiracOp_2) \right)\\
   &= (\tensMap_1)_\KPh \left( \IndRel^{\Gamma_1}(g_{1,0}, g_{1,1}) \KThProd \Ind^{\Gamma_2}(\clnDiracOp_2) \right)\\
   &= \IndRel^{\Gamma_1}(g_{1,0}, g_{1,1}) \secProd \Ind^{\Gamma_2}(\clnDiracOp_2).\qedhere
  \end{align*}
\end{proof}

\section{Compatibility with boundary maps} \label{sec:boundary}
The goal of this section is to show that the (partial) secondary invariants behave as expected with respect to Mayer--Vietoris boundary maps.
This amounts to an application of the principle \enquote{Boundary of Dirac is Dirac}.

\subsection{Localization algebras and Mayer--Vietoris sequences}
Let $X = X_1 \cup X_2$ be a cover of a proper metric space by two closed subspaces.
We wish to construct Mayer--Vietoris sequences relating the $\KOTh$-theory groups of the localization algebras of the spaces $X$, $X_1$, $X_2$ and $X_1 \cap X_2$.
To do this, the general principle is to employ the ideals associated to subspaces from \cref{defi:locAlgebras} and then to try to apply the abstract Mayer--Vietoris sequence we have discussed in \cref{subsec:longExactSeq}.
For the Roe algebra this has been implemented in \cite{higson-roe-yu:coarseMVprinciple}, a recent treatment which also deals with the structure algebra can be found in~\cite{siegel:mayerVietorisAnalyticStructureGroup}.
In this section, we sketch how to carry out this program for localization algebras.
In the context of metric simplicial complexes, this has been done by Yu~\cite{yu:localizationAlgebrasAndCoarseBaumConnes}.

\begin{defi}
    We say a cover $X = X_1 \cup X_2$ by two closed $\Gamma$-invariant subsets $X_1$, $X_2$ is (metrically) \emph{uniformly excisive} if there exists a function $C \colon \Rgr \to \Rgr$ with $C(r) \to 0$ as $r \to 0$ such that
    \begin{equation*}
        \nbh_r(X_1) \cap \nbh_r(X_2) \subseteq \nbh_{C(r)}(X_1 \cap X_2) \quad \text{for all $r > 0$.}
    \end{equation*}
\end{defi}

This is a slightly stronger requirement than \emph{coarse excisiveness} (or \emph{$\omega$-excisiveness} in the terminology of~\cite{siegel:mayerVietorisAnalyticStructureGroup}), where one does not require $C(r) \to 0$ as $r \to 0$.
However, if the cover $X = X_1 \cup X_2$ is coarsely excisive with $X_1, X_2$ closed and $X_1 \cap X_2$ is $\Gamma$-cocompact, then a compactness argument shows that it is uniformly excisive.
Moreover, if $X$ is a geodesic metric space, then any cover by two closed subsets is automatically uniformly excisive.

\begin{lem}[{\cite{higson-roe-yu:coarseMVprinciple}}] \label{lem:decompositionOfLocAlgebraMV}
    If $X = X_1 \cup X_2$ is uniformly excisive, then we have $\roeAlgLoc(X_1 \subset X)^{\Gamma} + \roeAlgLoc(X_2 \subset X)^{\Gamma} = \roeAlgLoc(X)^{\Gamma}$ and $\roeAlgLoc(X_1 \subset X)^{\Gamma} \cap \roeAlgLoc(X_2 \subset X)^{\Gamma} = \roeAlgLoc(X_1 \cap X_2 \subset X)^{\Gamma}$.
    The analogous statements for $\roeAlg$ and $\roeAlgLocZ$ are also true.
\end{lem}
\begin{proof}
   We only consider the statement for the localization algebra $\roeAlgLoc(X)^\Gamma$, since the other cases are proven analogously.
   The first claim is true since for each $T \in \roeAlgLoc(X)^{\Gamma}$, we have $\charFun_{X_1} T \in \roeAlgLoc(X_1 \subset X)^{\Gamma}$ and $\charFun_{X \setminus X_1} T \in \roeAlgLoc(X_2 \subset X)^{\Gamma}$, where $\charFun_A$ denotes the characteristic function of a subset $A \subset X$.
   The second claim follows since $\roeAlgLoc(X_1 \subset X)^{\Gamma} \cap \roeAlgLoc(X_2 \subset X)^{\Gamma} = \roeAlgLoc(X_1 \subset X)^\Gamma \roeAlgLoc(X_2 \subset X)^{\Gamma} = \roeAlgLoc(X_1 \cap X_2 \subset X)^\Gamma$, where the first equality is a general fact concerning ideals in $\Cstar$-algebras, and the latter follows from uniform excisiveness by general properties of the support of an operator, see~\cite[Lemma 6.3.6]{higson-roe:analyticKHomology}.
\end{proof}
In view of \cref{lem:idealsOfSubspace,lem:decompositionOfLocAlgebraMV}, the abstract Mayer--Vietoris sequence from \cref{subsec:longExactSeq} gives functorially compatible Mayer--Vietoris sequences for the $\KTh$-theory of $\roeAlg$, $\roeAlgLoc$ and $\roeAlgLocZ$ associated to a uniformly excisive decomposition.
In addition, we deduce the following more general version:
\begin{cor}
  Let $X = X_1 \cup X_2$ be uniformly excisive and $Z \subset X$ be closed $\Gamma$-invariant subset.
  Suppose furthermore that $Z = (Z \cap X_1) \cup (Z \cap X_2)$ is a coarsely excisive cover.
  Then:
\begin{gather*}
\roeAlgLocP{Z \cap X_1}(X_1 \subset X)^{\Gamma} + \roeAlgLocP{Z \cap X_2}(X_2 \subset X)^{\Gamma} = \roeAlgLocP{Z}(X)^{\Gamma}, \\ 
\roeAlgLocP{Z \cap X_1}(X_1 \subset X)^{\Gamma} \cap \roeAlgLocP{Z \cap X_2}(X_2 \subset X)^{\Gamma} = \roeAlgLocP{Z \cap X_1 \cap X_2}(X_1 \cap X_2 \subset X)^{\Gamma}.
\end{gather*}
Consequently, there is a long exact Mayer--Vietoris sequence as follows:
\begin{equation*} 
  \begin{tikzcd}[column sep=small, row sep=small]
    \dotsm \rar &\KTh_{n+1}\left( \roeAlgLocP{Z}(X)^\Gamma \right) \rar{\bdMV} & \KTh_{n}\left( \roeAlgLocP{Z \cap X_1 \cap X_2}(X_1 \cap X_2 \subset X)^\Gamma \right) \rar & {}\\
    {} \rar & \overset{\KTh_n\left(\roeAlgLocP{Z \cap X_1}(X_1 \subset X)^\Gamma \right)}{\underset{\KTh_n\left(\roeAlgLocP{Z \cap X_2}(X_2 \subset X)^\Gamma \right)}{\oplus}} \rar &  \KTh_n\left(\roeAlgLocP{Z}(X)^\Gamma \right) \rar{\bdMV} & \dotsm
  \end{tikzcd}
\end{equation*}

\end{cor}

\subsection{\enquote{Boundary of Dirac is Dirac}} \label{subsec:boundaryOfDiracIsDirac}
The following lemma is a standard fact.
However, we include a proof so as to demonstrate that it can be verified directly in our present setup.
\begin{lem} \label{lem:boundaryOfDiracIsOne}
   The Mayer--Vietoris boundary map $\bdMV \colon \KOTh_{1}(\roeAlgLoc(\R)) \to \KOTh_{0}(\roeAlgLoc(\{0\}))$ associated to the cover $\R = \Rgeq \cup \Rleq$ maps $\IndL(\clnDiracOp_\R) \in \KOTh_{1}(\roeAlgLoc(\R))$ to the unit element $1 \in \KOTh_0(\RReal) \iso \KOTh_{0}(\roeAlgLoc(\{0\}))$.
\end{lem}
\begin{proof}
   In our setup, we identify $\IndL(\clnDiracOp_\R) \in \KTh_1(\roeAlgLoc(\R))$ with $[\varphi_{\clnDiracOp_\R}] \KOThProd b  \in \KOThGr_0(\roeAlgLoc(\R) \tens \contZ(\R))$, where $b = b_1 \in \KOThGr_0\left(\contZ(\R) \tensGr \cliffAlgDual_1 \right)$ is the dual Dirac element, see~\cref{subsec:bott}, and $\varphi_{\clnDiracOp_\R} \colon \contZGr \to \roeAlgLoc(\R) \tensGr \cliffAlg_1$ is defined as in \cref{subsec:localIndexClasses}.
   Now observe that the following diagram commutes:
   \begin{equation*}
    \begin{tikzcd}[row sep=small, column sep=large]
     \KOThGr_{\KPh}\left(\contZ(\R) \tensGr \cliffAlgDual_1 \right) \rar{[\varphi_{\clnDiracOp_\R}] \KOThProd \_} \dar{\alpha_\KPh} & \KOThGr_\KPh \left( \roeAlgLoc(\R) \tens \contZ(\R) \right) \dlar{\gamma_\KPh} \\
      \KOThGr_\KPh(\RReal).
    \end{tikzcd}
   \end{equation*}
 Here $\alpha \colon \contZGr \tensGr \contZ(\R) \asympMorphism \cptOps \tensGr \cliffAlg_1$ and $\gamma \colon \roeAlgLoc(\R) \tens \contZ(\R) \asympMorphism \cptOps$ are the asymptotic morphisms mentioned in \cref{subsec:bott,rem:relationToETheoryFundamentalClass}.
 In particular, we obtain $\gamma_\KPh(\IndL(\clnDiracOp_\R)) = \alpha_\KPh(b) = 1$.
 
 The Mayer--Vietoris boundary map is induced by the inclusion
 \begin{equation*}
     \roeAlgLoc(\R) \tens \contZ(\R) = \contZ(\R, \roeAlgLoc(\R)) \hookrightarrow \pathsCstarMV\left(\roeAlgLoc(\R); \roeAlgLoc(\Rgeq \subset \R), \roeAlgLoc(\Rleq \subset \R )\right).
 \end{equation*}
 We will use the symbol $\pathsCstarMV$ as a shorthand to denote the latter $\Cstar$-algebra.
 We observe that it is equal to the following sum of ideals inside $\roeAlgLoc(\R) \tens \cont(\Rext)$,
 \begin{equation}
 \pathsCstarMV =  \roeAlgLoc(\Rgeq \subset \R) \tens \contZ([- \infty, \infty)) + \roeAlgLoc(\Rleq \subset \R )  \tens \contZ( (-\infty,\infty] ). \label{eq:pathsAsSumOfIdeals}
 \end{equation}
 Moreover, $\gamma$ extends to an asymptotic morphism
 \begin{equation*}
    \bar{\gamma} \colon  \roeAlgLoc(\R) \tens \cont(\Rext) \asympMorphism \bndOps(\Lp^2(\R)), \quad \bar{\gamma}_t(L \tens f) = L(t) f,
 \end{equation*}
 which, by \labelcref{eq:pathsAsSumOfIdeals}, restricts to an asymptotic morphism $\tilde{\gamma} \colon \pathsCstarMV \asympMorphism \cptOps(\Lp^2(\R))$.
 We obtain the following commutative diagram of (asymptotic) morphisms:
    \begin{equation*}
      \begin{tikzcd}[row sep=small]
        \roeAlgLoc(\R) \tens \contZ(\R) \rar[hook] \drar[dashed]{\gamma} & \pathsCstarMV \dar[dashed]{\tilde{\gamma}} & \roeAlgLoc(\{0\} \subset \R) \lar[hook] \dlar[dashed]{\ev} \\
		& \cptOps\left(\Lp^2(\R)) \right),
      \end{tikzcd}
   \end{equation*}
 All asymptotic morphisms above induce isomorphisms on $\KOTh$-theory.
 By definition, the Mayer--Vietoris boundary map is the the map $\bdMV \colon \KOThGr_1(\roeAlgLoc(\R)) = \KOThGr_0(\roeAlgLoc(\R) \tens \contZ(\R)) \to \KOThGr_0(\roeAlgLoc(\{0\} \subset \R)) \iso \KOThGr_0(\roeAlgLoc(\{0\}) \iso \KOTh_0(\RReal)$ induced by the upper row of the diagram composed with $\ev$.
 This finishes the proof since we already know that $\gamma_\KPh(\IndL(\clnDiracOp)) = 1$.
\end{proof}

\begin{thm}[Suspension isomorphism] \label{thm:generalBoundaryOfDiracIsDirac}
   Let $X$ be a proper metric space endowed with a free and proper $\Gamma$-action and $Z \subseteq X$ some closed $\Gamma$-invariant subset.
   Then the map
   \begin{equation*}
     \KTh_\KPh(\roeAlgLocP{Z}(X)^\Gamma) \to \KTh_{\KPh+1}(\roeAlgLocP{Z \times \R}(X \times \R)^\Gamma), \quad x \mapsto x \boxtimes \IndL(\clnDiracOp_\R),
   \end{equation*}
   is an isomorphism.
   Its inverse is given by the Mayer--Vietoris boundary map
   \begin{equation*}
   \bdMV \colon \KOTh_{\KPh+1}(\roeAlgLocP{Z \times \R}(X \times \R)^\Gamma) \to \KOTh_\KPh(\roeAlgLocP{Z}(X)^\Gamma)
   \end{equation*}
   associated to the cover $X \times \R = X \times \Rgeq \cup X \times \Rleq$.
   
   In particular, if $(X, g)$ is a complete Riemannian spin manifold with upsc outside~$Z$, then
   \begin{equation*}
   \bdMV \left( \IndLP{Z \times \R}^\Gamma(\clnDiracOp_{X \times \R}^{g\oplus \D{t}^2}) \right) = \IndLP{Z}^\Gamma(\clnDiracOp_X^{g}).
   \end{equation*}
\end{thm}
\begin{proof}
  A standard Eilenberg swindle shows that $\KTh_\KPh(\roeAlgLocP{Z \times \Rgeq}(X \times \Rgeq)^\Gamma)$ and $\KTh_\KPh(\roeAlgLocP{Z \times \Rleq}(X \times \Rleq)^\Gamma)$ vanish in all degrees.
  As a consequence, the Mayer--Vietoris boundary map $\bdMV$ from the statement of the theorem is an isomorphism.
  Therefore, it suffices for the first statement to show that $\bdMV(x \boxtimes \IndL(\clnDiracOp_\R)) = x$ for all $x \in \KOTh_n(\roeAlgLocP{Z}(X)^\Gamma)$.
  Let $W := X \times \R$.
  From the discussion of Mayer--Vietoris sequences in \cref{subsec:longExactSeq}, we obtain a commutative diagram as follows.
  \begin{equation*}
  \scalebox{0.715}{
    \begin{tikzcd}[row sep=small, ampersand replacement=\&]
      \KOTh_n\left(\roeAlgLocP{Z}(X)^\Gamma\right)\tens \KOTh_1\left(\roeAlgLoc(\R) \right) \rar[']{\id \tens \bdMV} \dar{\KOThProd} \& \KOTh_n\left(\roeAlgLocP{Z}(X)^\Gamma\right) \tens \KOTh_0\left(\roeAlgLoc(\{0\} \subset \R)\right) \dar{\KOThProd} \& \KOTh_n\left(\roeAlgLocP{Z}(X)^\Gamma\right) \tens \KOTh_0\left(\roeAlgLoc(\{0\})\right)  \lar{\iso} \dar{\KOThProd}\\
      \KOTh_{n+1}\left(\roeAlgLocP{Z}(X)^\Gamma \tensGr \roeAlgLoc(\R) \right) \rar{\bdMV} \dar{\tensMap_\KPh} \& \KOTh_n\left(\roeAlgLocP{Z}(X)^\Gamma \tensGr \roeAlgLoc(\{0\} \subset \R) \right) \dar{\tensMap_\KPh} \& \KOTh_n\left(\roeAlgLocP{Z}(X)^\Gamma \tensGr \roeAlgLoc(\{0\}) \right) \lar{\iso} \dar{\tensMap_\KPh} \\
      \KOTh_{n+1}\left(\roeAlgLocP{Z \times \R}(W)^\Gamma\right) \rar{\bdMV} \& \KOTh_{n}\left(\roeAlgLocP{Z \times \{0\}}(X \times \{0\} \subset W)^\Gamma\right) \& \KOTh_n\left(\roeAlgLocP{Z}(X)^\Gamma \right) \lar{\iso} 
    \end{tikzcd}}
  \end{equation*}
  In view of this diagram, it is enough that $\bdMV(\IndL(\clnDiracOp_\R)) \in \KOTh_0 \left( \roeAlgLoc( \{0\} \subset \R ) \right)$ agrees with the unit element $1 \in \KOTh_0(\RReal) \iso \KOTh_0\left(\roeAlgLoc(\{0\})\right) \iso \KOTh_0\left(\roeAlgLoc(\{0\} \subset \R)\right)$, which is precisely the content of \cref{lem:boundaryOfDiracIsOne}.
     Thus $\bdMV$ and taking the external product with $\IndL(\clnDiracOp_\R)$ are mutually inverse isomorphisms.

    By \cref{thm:productFormula}, we have $\IndLP{Z \times \R}( \clnDiracOp_{X \times \R}^{g \oplus \D{t}^2} ) = \IndLP{Z}^\Gamma(\clnDiracOp_X^{g}) \boxtimes \IndL(\clnDiracOp_\R)$. 
    Hence the second statement is a direct consequence of the first.
\end{proof}

In particular, we obtain the following statement for $\rho$-invariants, which was an essential technical step in \cite{PS14Rho} needed for the proof of the secondary partitioned manifold index theorem.
\begin{cor} \label{cor:boundaryOfRhoIsRho}
   Let $X$ be a spin manifold endowed with a free and proper $\Gamma$-action and $g$ a complete $\Gamma$-invariant Riemannian metric of upsc on $X$.
   Then
   \begin{equation*}
   \bdMV(\rho^\Gamma(g \oplus \D{t}^2)) = \rho^\Gamma(g),
   \end{equation*} where $\bdMV \colon \KOTh_{n+1}(\roeAlgLocZ(X \times \R)^\Gamma) \to \KOTh_n(\roeAlgLocZ(X)^\Gamma)$ is the Mayer--Vietoris boundary map associated to the cover $X \times \R = X \times \Rgeq \cup X \times \Rleq$.
\end{cor}
There is an analogous result for the relative index.
\begin{cor} \label{cor:boundaryOfRelIsRel}
  Let $X$ be an $n$-dimensional spin manifold with a free and proper $\Gamma$-action. 
  Let $g_0, g_1$ be two complete $\Gamma$-invariant Riemannian metrics of upsc on $X$ which are in the same uniform equivalence class.
   Then
   \begin{equation*}
   \bdMV(\IndRel^\Gamma(g_0 \oplus \D{t}^2, g_1 \oplus \D{t}^2)) = \IndRel^\Gamma(g_0, g_1),
   \end{equation*}
   where $\bdMV \colon \KOTh_{n+2}(\roeAlg(X \times \R)^\Gamma) \to \KOTh_{n+1}(\roeAlg(X)^\Gamma)$ is the Mayer--Vietoris boundary map associated to the cover $X \times \R = X \times \Rgeq \cup X \times \Rleq$.
\end{cor}
\begin{proof}
  We work with $W = \R \times X \times \R$ and consider a metric $h \oplus \D{t}^2$ on $W$ such that $h$ on $\R \times X$ restricts to $\D{s}^2 \oplus g_0$ on $(-\infty, 0] \times X$ and to $\D{s}^2 \oplus g_1$ on $[1,\infty) \times X$.
  Thus we obtain a localized index $\Ind^{\Gamma}_{[0,1] \times X \times \R}(\clnDiracOp_W^{h \oplus \D{t}^2}) = (\ev_1)_* \IndLP{[0,1] \times X \times \R}^{\Gamma}(\clnDiracOp_W^{h \oplus \D{t}^2}) \in \KTh_{n+2}(\roeAlg([0,1] \times X \times \R \subset W)^\Gamma)$, which is used to construct the relative index  $\IndRel^\Gamma(g_0 \oplus \D{t}^2, g_1 \oplus \D{t}^2)$ according to \cref{subsec:twoPSCrelIndex}.
  Consequently, by exploiting compatibility of the different Mayer--Vietoris sequences, it suffices to observe that 
  \begin{equation*}
    \bdMV  \left( \IndLP{[0,1] \times X \times \R}^\Gamma(\clnDiracOp_W^{h \oplus \D{t}^2}) \right) = \IndLP{[0,1] \times X }^\Gamma(\clnDiracOp_{\R \times X}^{h}),
  \end{equation*}
  which has been proved in \cref{thm:generalBoundaryOfDiracIsDirac}.
\end{proof}
A complete Riemannian manifold $Y$ is called \emph{hypereuclidean} if it admits a proper Lipschitz map $Y \to \R^q$ of degree $1$ into some Euclidean space $\R^q$ (if this is the case, then of course $q = \dim Y$).
Furthermore, we say that $Y$ is \emph{stably hypereuclidean} if $Y \times \R^k$ is hypereuclidean for some $k \geq 0$.
\thmDistinguishStabilizeByHypereucl
\begin{proof}
 It is enough to give a proof for the case that $Y$ is hypereuclidean and $\Lambda$ is the trivial group.
 This is because for each $k \geq 0$ we have a canonical map,
 \begin{equation*}
  \KThGr_{\KPh + q}(\roeAlgLocP{Z \times Y}(X \times Y)^{\Gamma \times \Lambda})  \to \KThGr_{\KPh + q + k}(\roeAlgLocP{Z \times Y \times \R^k}(X \times Y \times \R^k)^{\Gamma})
 \end{equation*}
 which is given by forgetting $\Lambda$-equivariance and taking the external product with $\IndL(\clnDiracOp_{\R^k})$.
 
 Now suppose that $f \colon Y \to \R^q$ is a proper Lipschitz map of degree $1$.
 Then the induced map on  $\KTh$-homology takes the fundamental class of $Y$ to the fundamental class of $\R^q$.
 On the level of localization algebras, this implies that $f_* \IndL(\clnDiracOp_Y) = \IndL(\clnDiracOp_{\R^q})$.
 Thus the map $(\id_X \times f)_\KPh \colon \KThGr_{\KPh + q}(\roeAlgLocP{Z \times Y}(X \times Y)) \to \KThGr_{\KPh + q}(\roeAlgLocP{Z \times \R^q}(X \times \R^q))$ takes elements of the form $x \secProd \IndL(\clnDiracOp_Y)$ to $x \secProd \IndL(\clnDiracOp_{\R^q})$.
 So, it even suffices to prove the first claim for $Y = \R^q$.
 However, for $Y = \R^q$, the result follows from (an iterated application of) \cref{thm:generalBoundaryOfDiracIsDirac}.
    
    The second claim is a consequence of the product formula from \cref{thm:productFormula}.
\end{proof}

\subsection{Partitioned manifolds}\label{sec:partitionedManifold}
In this subsection, we establish a new partitioned manifold index theorem for partial secondary invariants which generalizes the secondary partitioned manifold index theorem~\cite[Theorem 1.22]{PS14Rho}.
\begin{defi} \label{defi:partitionedMfd}
 Let $W$ be a spin manifold endowed with a free and proper $\Gamma$-action and a $\Gamma$-invariant Riemannian metric $h$.
 Suppose that $X \subset W$ is a closed submanifold of codimension $1$ that itself is endowed with a $\Gamma$-invariant Riemannian metric $g$.
 We say $(W,h)$ is \emph{partitioned by $(X,g)$} if $W \setminus X$ has two connected components, both of which are $\Gamma$-invariant individually, and the metric $h$ restricts to $g \oplus \D{t}^2$ on a fixed tubular neighborhood $X \times (-r, r) \subset W$.
\end{defi}
 In this case, we denote the closures of the connected components of $W \setminus X$ by $W_-$ and $W_+$.
 Then $W_\pm$ are closed $\Gamma$-invariant submanifolds of $W$ with common boundary $\bd W_\pm = X$.
 
 For certain applications it is convenient to have the flexibility of working with different distance functions on partitioned manifolds than the one induced by the Riemannian metric.
 In particular, this will be important in \cref{subsec:delocAPS} to deal with manifolds with non-connected boundary.
 Thus we introduce the following technical concept.
\begin{defi}\label{defi:adaptedToPartition}
 Let $(W,h)$ be partitioned by $(X,g)$.
 We say a ($\Gamma$-invariant) distance function $d_W \colon W \times W \to \Rgeq$ is \emph{adapted to the partition} if $(W, d_h) \to (W, d_W)$ is uniformly continuous and coarse and the restriction of $d_W$ to $X$ is uniformly and coarsely equivalent to the restriction of $d_h$.
\end{defi}
Note that we do not require that $d_g$ and the restriction of $d_h$ are uniformly and coarsely equivalent.
However, under the above assumptions, the inclusion map $(X, d_g) \to (W, d_h)$ is always uniformly continuous and thus also coarse (on each connected component of $X$).
\begin{rem}[Convention]
  When working with partitioned manifolds, we will implicitly assume that a fixed distance function that is adapted to the partition has been chosen.
  Moreover, all metric concepts such as balls or neighborhoods as well as Roe and localization algebras are to be defined with respect to this fixed distance function.
\end{rem}
\begin{defi}\label{defi:admissible}
 Let $(W,h)$ be partitioned by $(X,g)$.
 Let $Z \subset W$ be a closed $\Gamma$-invariant subset.
 We say \emph{$Z$ is admissible with respect to $W_+$} if 
 \begin{enumerate}[(i)]
  \item the subset $Z$ is of product structure\footnote{That is, $Z \cap (X \times (-r,r)) = (Z \cap X) \times (-r,r)$.} on the tubular neighborhood of $X$,
  \item for every $R > 0$ there exists $S > 0$ such that
 \begin{equation}
  \nbh_R(Z \cap W_+) \cap \nbh_R(X) \subseteq \nbh_S(Z \cap X). \label{eq:admissibleZ}
 \end{equation}
 \end{enumerate}
\end{defi}

\begin{ex}
 If $W_+ = X \times [0, \infty)$ and $Z \cap W_+ = Z \cap X \times [0,\infty)$, then $Z$ is admissible with respect to $W_+$.
\end{ex}

\begin{rem}
  If $Z$ is admissible with respect to $W_+$, then the following is also true:
  For all $R > 0$ there exists $S > 0$ such that
  \begin{gather}
    \nbh_R(Z) \cap \nbh_R(W_-) \subseteq \nbh_S(Z \cap W_-),\label{eq:consequ1OfAdmissible} \\
    \nbh_R(Z \cap W_-) \cap \nbh_R(Z \cap W_+) \subseteq \nbh_S(Z \cap X). \label{eq:ZcoverCoarselyExc}
  \end{gather}
  To check this, one uses the fact that the cover of $W$ by $W_+$ and $W_-$ is coarsely excicive.
  In fact, \labelcref{eq:ZcoverCoarselyExc} says that the cover of $Z$ by $Z \cap W_-$ and $Z \cap W_+$ is also coarsely excicive.
\end{rem}

We have the following partitioned manifold index theorem for partial secondary invariants.

\thmPartitionedManifold
Before turning to the proof of \cref{thm:partitionedManifold}, we discuss a few consequences.
Firstly, taking $Z = \emptyset$, this recovers the secondary partitioned manifold index theorem of Piazza--Schick~\cite[Theorem 1.22]{PS14Rho} in all dimensions.
Moreover, we obtain a partitioned manifold index theorem for the relative index as in the corollary below.
\begin{cor} \label{thm:relativeIndexPartitionedManifold}
 For $i=0,1$, let $(W,h_i)$ be a complete Riemannian spin manifold endowed with a free and proper $\Gamma$ action and suppose that it is partitioned by $(X,g_i)$.
 Assume that $h_0$ and $h_1$ are in the same uniform equivalence class.
 Then the Mayer--Vietoris boundary map
 \begin{equation*}
  \bdMV \colon \KTh_{n+2}(\roeAlg(W)^\Gamma) \to \KTh_{n+1}(\roeAlg(X)^\Gamma),
 \end{equation*}
  associated to the cover $W = W_+ \cup W_-$ satisfies
  \begin{equation*}
   \bdMV\left( \IndRel^\Gamma(h_0,h_1) \right) = \IndRel^\Gamma(g_0, g_1).
  \end{equation*}
\end{cor}
\begin{proof}
The corollary follows formally from \cref{thm:partitionedManifold} along the lines of the proof of \cref{cor:boundaryOfRelIsRel}.
\end{proof}
Furthermore, granted \cref{thm:partitionedManifold}, we are able to prove concordance invariance of partial secondary invariants which is still open from \cref{subsec:partialIndex}.
\begin{proof}[Proof of \cref{thm:concordanceInvariance}]
  Suppose that $g_0$, $g_1$ are two uniformly equivalent $\Gamma$-invariant metrics on $X$ and let $h$ be a metric on $W := \R \times X$ such that the conditions of \cref{defi:concordanceRelZ} are satisfied.
  We identify $X$ with (say) $\{-1\} \times X \subseteq W$ and view $(W, h)$ to be partitioned by $(X, g_0)$ with $W_- = (-\infty,-1] \times X$ and $W_+ = [-1, \infty) \times X$.
  If we take $d_W$ to be the distance function induced by $\D{t}^2 \oplus g_0$, then it follows from the assumptions on $h$ that $d_W$ a metric on $W$ adapted to the partition in the sense of \cref{defi:adaptedToPartition}.
  Thus, by \cref{thm:partitionedManifold}, we have
  \begin{equation*}
  \bdMV \IndLP{\R \times Z}^\Gamma(\clnDiracOp^h_W) = \IndLP{Z}^\Gamma(\clnDiracOp^{g_0}_X)
  \end{equation*}
  Thus it follows from \cref{thm:generalBoundaryOfDiracIsDirac} that
  \begin{equation*}
  \IndLP{\R \times Z}^\Gamma(\clnDiracOp^h_W) = \IndL(\clnDiracOp_\R) \boxtimes \IndLP{Z}^\Gamma(\clnDiracOp^{g_0}_X).
  \end{equation*}
  However, we can also identify $X$ with $\{2\} \times X \subset W$ and apply the same argument again to show that
  \begin{equation*}
  \IndLP{\R \times Z}^\Gamma(\clnDiracOp^h_W) = \IndL(\clnDiracOp_\R) \boxtimes \IndLP{Z}^\Gamma(\clnDiracOp^{g_1}_X).
  \end{equation*}
  By \cref{thm:generalBoundaryOfDiracIsDirac} these two facts imply that $\IndLP{Z}^\Gamma(\clnDiracOp^{g_0}_X) = \IndLP{Z}^\Gamma(\clnDiracOp^{g_1}_X)$.
\end{proof}

We now proceed with technical lemmas needed for the proof of \cref{thm:partitionedManifold}.
If $W$ is partitioned by $X$, then a $(W, \Gamma, \cliffAlg_{n+1})$-module $\clnXmodule$ can be restricted to a $(W_+, \Gamma, \cliffAlg_{n+1})$-module $\clnXmodule_{+} := \ran(\charFun_{W_+})$ using the projection in $\clnXmodule$ corresponding to the characteristic function of $W_+$.
This will be used implicitly in the following.
\begin{lem}\label{thm:quotientIsoPM}
 Let that $Z \subseteq W$ be a closed $\Gamma$-invariant subset that is admissible with respect to $W_+$. 
 Then for all $p \in \N$:
 \begin{enumerate}[(i)]
  \item The inclusion $\clnXmodule_{+} \subset \clnXmodule$ induces $\ast$-isomorphisms:\label{item:exciseOneHalfIso}
   \begin{multline*}
  \frac{\roeAlgLocP{Z\cap W_+}(W_+; \cliffAlg_{p})^\Gamma}{\roeAlgLocP{Z \cap X}(X \subset W_+; \cliffAlg_{p})^\Gamma}  \underset{\iso}{\overset{i}{\to}} \frac{\roeAlgLocP{Z \cap W_+}(W_+ \subset W; \cliffAlg_{p})^\Gamma}{\roeAlgLocP{Z \cap X}(X \subset W; \cliffAlg_{p})^\Gamma} \\ \underset{\iso}{\overset{j}{\to}} \frac{\roeAlgLocP{Z}(W; \cliffAlg_{p})^\Gamma}{\roeAlgLocP{Z \cap W_-}(W_- \subset W; \cliffAlg_{p})^\Gamma}. 
 \end{multline*}
 If $Z$ is also admissible with respect to $W_-$, then the inverse of these isomorphisms is induced by the expectation
 \begin{equation*}
 \Phi \colon \roeAlg(W; \cliffAlg_p)^\Gamma \to \roeAlg(W_+; \cliffAlg_p)^\Gamma, \quad T \mapsto \charFun_{W_+} T \charFun_{W_+}.
 \end{equation*}
 \item The following $\ast$-homomorphism is injective: \label{item:exciseOneHalfPartialInj}
  \begin{equation*}
  \frac{\roeAlgLocP{Z\cap W_+}(W_+; \cliffAlg_p)^\Gamma}{\roeAlgLocP{Z \cap X}(X \subset W_+; \cliffAlg_p)^\Gamma}  \hookrightarrow \frac{\roeAlgLoc(W_+; \cliffAlg_p)^\Gamma}{\roeAlgLoc(X \subset W_+; \cliffAlg_p)^\Gamma}.
 \end{equation*}
 \end{enumerate}
\end{lem}
\begin{proof}
To simplify the exposition, we drop the Clifford algebra and the group $\Gamma$ from the notation during this proof.
  
  To show that the map $i$ in \labelcref{item:exciseOneHalfIso} is an isomorphism, we need,
  \begin{gather}
      \roeAlg(Z\cap W_+ \subset W_+) \cap \roeAlg(Z \cap X \subset W) = \roeAlg(Z \cap X \subset W_+) \label{eq:exciseOneHalfInj1}\\
    \roeAlgLoc(W_+) \cap \roeAlgLoc(X \subset W) = \roeAlgLoc(X \subset W_+), \label{eq:exciseOneHalfInj2}\\
    \roeAlgLocP{Z\cap W_+}(W_+) + \roeAlgLocP{Z \cap X}(X \subset W) = \roeAlgLocP{Z \cap W_+}(W_+ \subset W). \label{eq:exciseOneHalfSur}
  \end{gather}
  Indeed, consider the expectation $\Phi \colon \roeAlg(W) \to \roeAlg(W_+), T \mapsto \charFun_{W_+} T \charFun_{W_+}$.
  Note that $\supp(\Phi(T)) \subseteq \supp(T) \cap W_+ \times W_+$ for all $T \in \roeAlg(W)$.
  In particular, $\Phi$ extends to an expectation $\roeAlgLoc(W) \to \roeAlgLoc(W_+)$ which we also denote by $\Phi$.
  We have $\Phi(\roeAlg(Z \cap X \subset W)) = \roeAlg(Z \cap X \subset W_+)$ and $\Phi(\roeAlgLoc(X \subset W)) = \roeAlgLoc(X \subset W_+)$ which proves \labelcref{eq:exciseOneHalfInj1,eq:exciseOneHalfInj2}.
  To prove \labelcref{eq:exciseOneHalfSur}, let $L \in \roeAlgLocP{Z \cap W_+}(W_+ \subset W)$.
  Then one can verify that $\Phi(L) \in \roeAlgLocP{Z \cap W_+}(W_+)$ and $L - \Phi(L) \in \roeAlgLoc(X \subset W)$.
  Moreover, we have $\ev_1(L) - \Phi(\ev_1(L)) \in \roeAlg(Z \cap W_+ \subset W) \cap \roeAlg(X \subset W)$.
  In addition, using \labelcref{eq:admissibleZ}, we obtain the following equality for this intersection of Roe algebras:
  \begin{equation}
    \roeAlg(Z \cap W_+ \subset W) \cap \roeAlg(X \subset W) = \roeAlg(Z \cap X \subset W). \label{eq:admissibleZConsequence}
  \end{equation}
   Thus the decomposition $\Phi(L) + (L - \Phi(L)) = L$ proves \labelcref{eq:exciseOneHalfSur} and shows that the inverse of $i$ is always induced by $\Phi$ (even if $Z$ is not admissible with respect to $W_-$).
   
   Similarly, one can use \labelcref{eq:ZcoverCoarselyExc} to verify that $j$ is injective.
   To see that $j$ is surjective, fix $R > 0$ and let $\xi \colon [1,\infty) \to \R$ be a continuous function such that $\xi(1) = 3 R$ and $\xi(t) = 0$ for all $t \geq 2$.
   For each $r \geq 0$ set $W_{\geq r} := W_+ \setminus \nbh_r(W_-)$ and for $t \geq 1$ set $\chi_t := \charFun_{W_{\geq \xi(t)}}$.
   Instead of using $\Phi$ as above, we define
   \begin{equation*}
     \Psi \colon \roeAlgLoc(W) \to \roeAlgLoc(W), \quad L \mapsto (t \mapsto L(t) \chi_t).
   \end{equation*}
  Now let $L \in \roeAlgLocP{Z}(W)$ and suppose $\propag(L(t)) \to 0$ and that $T := \ev_1(L)$ has finite propagation $R$ as well as $\supp(T) \subseteq \nbh_R(Z) \times \nbh_R(Z)$.
   Then $\Psi(L) \in \roeAlgLoc(W_+ \subset W)$ and $L - \Psi(L) \in \roeAlgLoc(W_- \subset W)$ since $\propag(L(t)) \to 0$ and $\chi_t = \charFun_{W_+}$ for $t \geq 2$.
   Moreover, we consider $S := \ev_1(\Psi(L)) = T \charFun_{W_{\geq 3R}}$ and observe that
   \begin{equation*}
    \supp(S) \subseteq (\nbh_R(Z) \cap W_{\geq 2R}) \times (\nbh_R(Z) \cap W_{\geq 3R}).
   \end{equation*}
   Since $\nbh_R(Z) \cap W_{\geq 2R} \subseteq \nbh_R(Z \cap W_+)$ it follows that $S \in \roeAlg(Z \cap W_+ \subset W)$ and hence $\Psi(L) \in \roeAlgLocP{Z \cap W_+}(W_+ \subset W)$.
   Let $S^\prime = \ev_1(L - \Psi(L)) = T \charFun_{\nbh_{3R}(W_-)}$.
   Then
    \begin{equation*}
    \supp(S^\prime) \subseteq (\nbh_R(Z) \cap \nbh_{4R}(W_-)) \times (\nbh_R(Z) \cap \nbh_{3R}(W_-)),
   \end{equation*}
   which due to \labelcref{eq:consequ1OfAdmissible} implies that $S^\prime \in \roeAlg(Z \cap W_-)$.
   In summary, we have $\Psi(L) - L \in \roeAlgLocP{Z \cap W_-}(W_- \subset W)$ and $\Psi(L) \in \roeAlgLocP{Z \cap W_+}(W_+ \subset W)$, and this proves that $j$ is surjective, hence an isomorphism.
   
   In contrast, it is \emph{not true in general} that $\Phi(\roeAlgLocP{Z \cap W_-}(W_- \subset W)) \subseteq \roeAlgLocP{Z \cap X}(X \subset W)$ and $\Phi(\roeAlgLocP{Z}(W)) \subseteq \roeAlgLocP{Z \cap W_+}(W_+ \subset W)$.
   However, if $Z$ is also admissible with respect to $W_-$, then these statements are in fact true and one can verify that $\Phi$ is inverse to $j$ (as follows from an argument using \labelcref{eq:admissibleZ,eq:consequ1OfAdmissible,eq:ZcoverCoarselyExc} with the roles of $W_+$ and $W_-$ reversed).
  
  Statement \labelcref{item:exciseOneHalfPartialInj} is a consequence of a version of \labelcref{eq:admissibleZConsequence} restricted to $W_+$, which also follows from \labelcref{eq:admissibleZ}.
\end{proof}
The main ingredient to prove \cref{thm:partitionedManifold} is the following \enquote{swapping lemma} which implies that we can modify one half of a partitioned manifold without changing the image of the Mayer--Vietoris boundary map.
Hence it reduces the general partitioned manifold index theorem to the product situation.
This idea originated in a proof of Roe's partitioned manifold index theorem due to Higson~\cite{higson91:cobordism}.
\begin{lem}[{Compare \cite[Lemma 3.1]{higson91:cobordism}}]\label{thm:swappingLemma}
 Let $(W,h)$ and $(\widetilde{W}, \widetilde{h})$ both be partitioned by $(X,g)$ and $d_W$ and let  $d_{\widetilde{W}}$ be metrics adapted to the partitions.
 Suppose $\widetilde{W}_+ = W_+$ and that the restrictions of $h$ and $\widetilde{h}$ as well as $d_W$ and $d_{\widetilde{W}}$ agree on $W_+$, respectively.
 Let $Z \subseteq W$, $\widetilde{Z} \subseteq \widetilde{W}$ be admissible with respect to $W_+ = \widetilde{W}_+$ with $Z \cap W_+ = \widetilde{Z} \cap \widetilde{W}_+$.
 Assume that $h$ has upsc outside $Z$ and $\widetilde{h}$ has upsc outside $\widetilde{Z}$.
 Then 
   \begin{equation*}
   \bdMV\left( \IndLP{Z}^\Gamma(\clnDiracOp_W^h) \right) = \bdMV\left( \IndLP{\widetilde{Z}}^\Gamma(\clnDiracOp_{\widetilde{W}}^{\widetilde{h}}) \right).
  \end{equation*}
\end{lem}
\begin{proof}
  By \cref{thm:quotientIsoPM} \labelcref{item:exciseOneHalfIso} (applied to $Z = W$), we have isomorphisms
  \begin{equation}\label{eq:quotientIsoPMNoPsc}
  \frac{\roeAlgLoc(W; \cliffAlg_{n+1})^\Gamma}{\roeAlgLoc(W_- \subset W; \cliffAlg_{n+1})^\Gamma} \overset{\Phi}{\to} \frac{\roeAlgLoc(W_+; \cliffAlg_{n+1})^\Gamma}{\roeAlgLoc(X \subset W_+; \cliffAlg_{n+1})^\Gamma}  \overset{\widetilde{\Phi}}{\leftarrow} \frac{\roeAlgLoc(\widetilde{W}; \cliffAlg_{n+1})^\Gamma}{\roeAlgLoc(\widetilde{W}_- \subset \widetilde{W}; \cliffAlg_{n+1})^\Gamma},
 \end{equation}
 where the maps $\Phi$ and $\widetilde{\Phi}$ are induced by $T \mapsto \charFun_{W_+} T \charFun_{W_+}$.
 Consider $\varphi := \varphi_{\clnDiracOp_W^h} \colon \contZGr \to \roeAlgLoc(W; \cliffAlg_{n+1})^\Gamma$ and $\widetilde{\varphi} := \varphi_{\clnDiracOp_{\tilde{W}}^{\tilde{h}}} \colon \contZGr \to \roeAlgLoc(\widetilde{W}; \cliffAlg_{n+1})^\Gamma$ which have been defined in \cref{subsec:localIndexClasses}.
 Let $\pi \colon \roeAlgLoc(W; \cliffAlg_{n+1})^\Gamma \to \roeAlgLoc(W; \cliffAlg_{n+1})^\Gamma / \roeAlgLoc(W_- \subset W; \cliffAlg_{n+1})^\Gamma$ be the canonical quotient map and define $\widetilde{\pi}$ analogously.
 Since $\clnDiracOp_W^h = \clnDiracOp_{\widetilde{W}}^{\widetilde{h}}$ on $W_+$, it follows from propagation speed arguments (see~\cite[Corollary 10.3.4, Proposition 10.3.5]{higson-roe:analyticKHomology}, \cite[992]{PS14Rho}) that for all $f \in \contZGr$,
 \begin{equation*}
   \charFun_{W_+} f\left(\frac{1}{\mathrm{t}}\clnDiracOp_W^h\right) \charFun_{W_+} - \charFun_{W_+} f\left( \frac{1}{\mathrm{t}} \clnDiracOp_{\widetilde{W}}^{\widetilde{h}} \right) \charFun_{W_+} \in \roeAlgLoc(X \subset W_+; \cliffAlg_{n+1})^\Gamma.
 \end{equation*}
 This implies that the following $\ast$-homomorphisms are equal:
 \begin{equation*}
   \Phi \circ \pi \circ \varphi = \widetilde{\Phi} \circ \widetilde{\pi} \circ \widetilde{\varphi} \colon \contZGr \to \roeAlgLoc(W_+; \cliffAlg_{n+1})^\Gamma / \roeAlgLoc(X \subset W_+; \cliffAlg_{n+1})^\Gamma. 
 \end{equation*}
  By \cref{thm:quotientIsoPM} we also obtain isomorphisms,
  \begin{multline} \label{eq:quotientIsoPM}
  \frac{\roeAlgLocP{Z}(W; \cliffAlg_{n+1})^\Gamma}{\roeAlgLocP{Z \cap W_-}(W_- \subset W; \cliffAlg_{n+1})^\Gamma} \overset{j}{\leftarrow} \frac{\roeAlgLocP{Z\cap W_+}(W_+; \cliffAlg_{n+1})^\Gamma}{\roeAlgLocP{Z \cap X}(X \subset W_+; \cliffAlg_{n+1})^\Gamma}  \\
  \overset{\tilde{j}}{\rightarrow} \frac{\roeAlgLocP{\widetilde{Z}}(\widetilde{W}; \cliffAlg_{n+1})^\Gamma}{\roeAlgLocP{\widetilde{Z} \cap \widetilde{W}_-}(\widetilde{W}_- \subset \widetilde{W}; \cliffAlg_{n+1})^\Gamma},
 \end{multline}
 and the map 
  \begin{equation*}
  k \colon \frac{\roeAlgLocP{Z\cap W_+}(W_+; \cliffAlg_{n+1})^\Gamma}{\roeAlgLocP{Z \cap X}(X \subset W_+; \cliffAlg_{n+1})^\Gamma} \hookrightarrow \frac{\roeAlgLoc(W_+; \cliffAlg_{n+1})^\Gamma}{\roeAlgLoc(X \subset W_+; \cliffAlg_{n+1})^\Gamma}
  \end{equation*} 
   induced by inclusion is injective.
 Choose an appropriate homotopy equivalence $\psi \colon \contZGr \to \contZGr(-\varepsilon, \varepsilon)$ such that $\varphi \circ \psi$ and $\widetilde{\varphi} \circ \psi$ take values in $\roeAlgLocP{Z}(W; \cliffAlg_{n+1})^\Gamma$ and $\roeAlgLocP{\widetilde{Z}}(\widetilde{W}; \cliffAlg_{n+1})^\Gamma$, respectively.
 We also consider the quotient map
  \begin{equation*}
  \pi_Z \colon \roeAlgLocP{Z}(W; \cliffAlg_{n+1})^\Gamma \to \roeAlgLocP{Z}(W; \cliffAlg_{n+1})^\Gamma / \roeAlgLocP{Z \cap W_-}(W_- \subset W; \cliffAlg_{n+1})^\Gamma,
  \end{equation*} 
  and define $\widetilde{\pi}_Z$ analogously.
 We then have 
 \begin{equation*}
   k \circ j^{-1} \circ \pi_Z \circ \varphi \circ \psi = \Phi \circ \pi \circ \varphi \circ \psi = \widetilde{\Phi} \circ \widetilde{\pi} \circ \widetilde{\varphi} \circ \psi = k \circ \widetilde{j}^{-1} \circ \widetilde{\pi}_Z \circ \widetilde{\varphi} \circ \psi
 \end{equation*}
 and thus by injectivity of $k$,
 \begin{equation*}
 j^{-1} \circ \pi_Z \circ \varphi \circ \psi =  \widetilde{j}^{-1} \circ \widetilde{\pi}_Z \circ \widetilde{\varphi} \circ \psi \colon \contZGr \to \frac{\roeAlgLocP{Z\cap W_+}(W_+; \cliffAlg_{n+1})^\Gamma}{\roeAlgLocP{Z \cap X}(X \subset W_+; \cliffAlg_{n+1})^\Gamma}.
 \end{equation*} 
 This proves that
 \begin{align*}
  (j^{-1} \circ \pi_Z)_\KPh&\IndLP{Z}^\Gamma(\clnDiracOp_W^h)) = (\widetilde{j}^{-1} \circ \widetilde{\pi}_Z)_\KPh \IndLP{\widetilde{Z}}^\Gamma(\clnDiracOp_{\widetilde{W}}^{\tilde{h}})\\ &\in \KTh_{n+1}\left( \frac{\roeAlgLocP{Z\cap W_+}(W_+)^\Gamma}{\roeAlgLocP{Z \cap X}(X \subset W_+)^\Gamma}\right) \overset{\bd_{W_+}}{\to} \KTh_n(\roeAlgLocP{Z \cap X}(X \subset W_+)^\Gamma).
 \end{align*}
 This completes the proof of the lemma since the Mayer--Vietoris boundary maps can be computed by composing $(j^{-1} \circ \pi_Z)_\KPh$ (respectively $(\widetilde{j}^{-1} \circ \widetilde{\pi}_Z)_\KPh$) with the boundary map $\bd_{W_+}$ displayed above, see~\cref{rem:alternativeMVBoundary}.
\end{proof}

\begin{proof}[Proof of \cref{thm:partitionedManifold}]
  Applying \cref{thm:swappingLemma} twice reduces the theorem to the product situation.
  Indeed, we may consider a new partitioned manifold $\widetilde{W}$ with $\widetilde{W_+} = W_+$, $\widetilde{W}_- = X \times (-\infty, 0]$, $\widetilde{Z} \cap W_+ = Z \cap W_+$, $\widetilde{Z} \cap W_- = Z \cap X \times (-\infty, 0]$ and $h = g \oplus \D{t}^2$ on $W_-$.
  We also need to construct a distance function $d_{\widetilde{W}}$ that is adapted to the partition.
  To obtain that, we just glue the Euclidean product distance function on $X \times (-\infty,0]$ to the metric on $W_+$ (where we use the restriction of the fixed adapted metric from $W$ to $W_+$), compare also \cref{rem:adaptedDistanceFunction} below.
  Then, by \cref{thm:swappingLemma}, we may work with $\widetilde{W}$ instead of $W$.
  However, in this case $\widetilde{Z}$ is also admissible with respect to $\widetilde{W}_-$, so we can swap the roles of $W_+$ and $W_-$ in \cref{thm:swappingLemma} to also replace $W_+$ by $X \times [0, \infty)$.
  Thus we only need to consider $W = X \times \R$, $Z = Z \cap X \times \R$, $h = g \oplus \D{t}^2$, and this special case of the theorem was already proved in \cref{thm:generalBoundaryOfDiracIsDirac}.
\end{proof}

\section{Geometric applications} \label{sec:higherInvariants}

\subsection{Delocalized APS-index theorem}\label{subsec:delocAPS}
  In this subsection, we fix $W$ to be an $n+1$-dimensional spin manifold with boundary $\bd W = X$ endowed with a complete Riemannian metric $h$ and a proper and free isometric action of a discrete group $\Gamma$.
    Moreover, we assume that $X \hookrightarrow W$ is a coarse equivalence.
  Suppose that $h$ is a product metric $g_X \oplus \D{t}^2$ on a collar neighborhood of $X$.
  Let $W_\infty$ be the spin manifold obtained from $W$ by attaching an infinite cylinder $C := X \times [0,\infty)$ to the boundary $\bd W = X$.
  The Riemannian metric $h$ extends by $g_X \oplus \D{t}^2$ to a complete Riemannian metric $h_\infty$ on $W_\infty$.
  
  \begin{rem}[Adapted distance function on $W_\infty$]\label{rem:adaptedDistanceFunction}
    We also construct a distance function $d_{W_\infty}$ on $W_\infty$:
    We furnish $X$ with the restriction of the distance function $d_h$ from $W$ to $X$ and denote this distance by $d_X$.
 Then we endow $C$ with a distance function $d_C$ defined by $d_C((x,t),(x^\prime,t^\prime)) = \sqrt{d_X(x,x^\prime)^2 + |t-t|^2}$ for $x, x^\prime \in X$, $t,t^\prime \in \R$.
  We let $d_{W_\infty}$ be the distance function on $W_\infty$ that is obtained from~$d_h$ and~$d_C$ by gluing it along $X$.\footnote{More precisely, $d_{W_\infty}$ restricts to $d_h$ on $W$, to $d_C$ on $C$ and for $z \in W$, $y \in C$ we set $d_{W_\infty}(z,y) = \inf_{x \in X}\left(d_h(z,x) + d_C((x,0),y) \right)$.}
  It follows that $W_\infty$ is partitioned by $X$ and $d_{W_\infty}$ is a distance function that is adapted to the partition in the sense of \cref{defi:adaptedToPartition}.
  However, in general the distance function induced by $h_\infty$ will be coarser than $d_{W_\infty}$ (in particular, if $X$ has more than one connected component).
  In the following, we will always use the distance function $d_{W_\infty}$ to define the Roe and localization algebras on $W_\infty$.
  This is crucial so as for \cref{thm:coarseEquToBoundaryRoeAlgZero} below to hold.
  \end{rem}  
  
  \begin{lem}\label{thm:coarseEquToBoundaryRoeAlgZero}
    We have $\KTh_\KPh(\roeAlg(W_\infty)^\Gamma) = 0$.
    In particular, there exists a unique element $\rho^\Gamma(W_\infty) \in \KTh_{n+1}(\roeAlgLocZ(W_\infty)^\Gamma)$ which maps to $\IndL^\Gamma(\clnDiracOp_{W_\infty}) \in \KTh_{n+1}(\roeAlgLoc(W_\infty)^\Gamma)$.
  \end{lem}
  \begin{proof}
  By assumption, $W_\infty$ is coarsely equivalent to the flasque space $C = X \times [0,\infty)$ and thus $\KTh_\KPh(\roeAlg(W_\infty)^\Gamma)$ vanishes (see~\cite[Lemma 6.4.2]{higson-roe:analyticKHomology} and the discussion in \cref{subsec:compactToNonCompact}).
  The second statement follows due to exactness.
  \end{proof}
  
\begin{defi}\label{defi:rhoOfBordism}
  We define
  \begin{equation*}
    \rho^\Gamma(W) := \bdMV(\rho^\Gamma(W_\infty)) \in \KTh_{n}(\roeAlgLocZ(X)^\Gamma),
  \end{equation*}
  where $\rho^\Gamma(W_\infty)$ is as in \cref{thm:coarseEquToBoundaryRoeAlgZero} and we use the Mayer--Vietoris boundary map associated to the cover $W_\infty = C \cup W$.
\end{defi}
\begin{rem}
  \cref{thm:partitionedManifold} (applied to $W_\infty$ with $Z = W_\infty$) and naturality of the Mayer--Vietoris sequence show that $\rho^\Gamma(W)$ maps to $\IndL^\Gamma(\clnDiracOp_X) \in \KTh_{n}(\roeAlgLoc(X)^\Gamma)$.
 Since the $\KTh$-homology fundamental class of a spin manifold does not vanish (see~\cite[Lemma~12.2.4]{higson-roe:analyticKHomology}), the local index class $\IndL^\Gamma(\clnDiracOp_X)$ is always non-zero.
 Thus $\rho^\Gamma(W) \in \KTh_n(\roeAlgLocZ(X)^\Gamma)$ never vanishes either.
  However, $\Ind^\Gamma(\clnDiracOp_X) = (\ev_1)_\KPh \IndL^\Gamma(\clnDiracOp_X)$ vanishes by exactness.
  This constitutes a variant of bordism invariance for the coarse index, see~\cite{W12Bordism}.
\end{rem}
We may think of the element $\rho^\Gamma(W)$ as a secondary invariant which is associated to $W$ viewed as a null-bordism for $X$.

\begin{sloppypar}
If we suppose that $g_X$ has upsc, then we also have a secondary invariant $\rho^\Gamma(g_X)$ associated to $g_X$.
In the following, we will identify the difference between these two secondary invariants $\rho^\Gamma(g_X)$ and $\rho^\Gamma(W)$.
Indeed, if $g_X$ has upsc, then $W_\infty$ has upsc outside $W$.
It follows that there is a localized index $\Ind_W^\Gamma(\clnDiracOp_{W_\infty}) \in \KTh_{n+1}(\roeAlg(W~\subseteq~W_\infty))$.
We denote the image of this class under the isomorphism $\KTh_{\KPh}(\roeAlg(W~\subseteq~W_\infty)^\Gamma) \iso \KTh_\KPh(\roeAlg(X)^\Gamma)$ by $\Ind_W^\Gamma(\clnDiracOp_W) \in \KTh_{n+1}(\roeAlg(X)^\Gamma)$, compare~\cite{PS14Rho}.
\end{sloppypar}

\thmRefinedDeloalizedAPS
\begin{proof}
During this proof, we will consider $W_\infty$ as constructed above. However, for technical reasons, we now identify $X$ with $X \times \{1\} \subset C$ instead of $X \times \{0\}$ so that $X \cap W = \emptyset$.
  Using \cref{thm:coarseEquToBoundaryRoeAlgZero}, we obtain maps
  \begin{gather*}
  r \colon \KTh_{n+1}(\roeAlgLocP{W}(W_\infty)^\Gamma) \to \KTh_{n+1}(\roeAlgLoc(W_\infty)^\Gamma) \iso \KTh_{n+1}(\roeAlgLocZ(W_\infty)^\Gamma), \\
   s \colon \KTh_{n+1}(\roeAlg(W \subset W_\infty)^\Gamma) \iso \KTh_{n+2}\left(\frac{\roeAlg(W_\infty)^\Gamma}{\roeAlg(W \subset W_\infty)}\right) \overset{\delta}{\to}  \KTh_{n+1}(\roeAlgLocP{W}(W_\infty)^\Gamma),
  \end{gather*}
  where $\delta$ is the boundary map associated to the extension $0 \to \roeAlgLocP{W}(W_\infty)^\Gamma \to \roeAlgLoc(W_\infty)^\Gamma \to \roeAlg(W_\infty)^\Gamma/\roeAlg(W \subset W_\infty)^\Gamma \to 0$.
    By construction, we have $r(\IndLP{W}(\clnDiracOp_{W_\infty})) = \rho^\Gamma(W_\infty)$ and $(\ev_1)_\KPh(\IndLP{W}(\clnDiracOp_{W_\infty})) = \Ind_W^\Gamma(\clnDiracOp_{W_\infty})$.
  One can verify that
  \begin{gather*}
   r \circ i = \id, \qquad (\ev_1)_\KPh \circ s = \id, \quad r \circ s = 0.
  \end{gather*}
    Hence $s$ (respectively $r$) determines a splitting of the long exact sequence in $\KTh$-theory associated to $0 \to \roeAlgLocZ(W_\infty)^\Gamma \to \roeAlgLocP{W}(W_\infty)^\Gamma \to \roeAlg(W \subset W_\infty)^\Gamma \to 0$.
      \begin{equation*}
    \begin{tikzcd}[column sep=small]
    \KTh_{n+1}(\roeAlgLoc(W_\infty)^\Gamma) & & \KTh_{n+2}(\roeAlg(W_\infty)^\Gamma / \roeAlg(W \subset W_\infty)) \dar{\iso} \dlar[bend right=10,swap]{\delta} \\
      \KTh_{n+1}(\roeAlgLocZ(W_\infty)^\Gamma) \rar[swap]{i} \uar{\iso} \drar[bend right=20]{\bdMV^\prime} & \KTh_{n+1}(\roeAlgLocP{W}(W_\infty)^\Gamma) \dar{\bdMV} \rar[swap]{(\ev_1)_\KPh} \ular[bend right=10] \lar[bend right=10]{r} & \KTh_{n+1}(\roeAlg(W \subset W_\infty)^\Gamma) \lar[bend right=10]{s} \\
      & \KTh_n(\roeAlgLocZ(X)^\Gamma) & \KTh_{n+1}(\roeAlg(X)^\Gamma) \uar{k}[swap]{\iso} \lar{\bd}
    \end{tikzcd}
  \end{equation*}
  In particular, $i \circ r + s \circ (\ev_1)_\KPh = \id$ and thus
  \begin{equation*}
    \IndLP{W}^\Gamma(\clnDiracOp_{W_\infty}) = i(\rho^\Gamma(W_\infty)) + s(\Ind_W^\Gamma(\clnDiracOp_{W_\infty})).
  \end{equation*}
  Moreover, via an additional diagram chase one can check that
  \begin{gather*}
    \bdMV \circ i = \bdMV^\prime, \qquad
    \bdMV \circ s = \bd \circ k^{-1}.
  \end{gather*}
  From this we deduce 
 \begin{align*}
    \bdMV \IndLP{W}^\Gamma(\clnDiracOp_{W_\infty}) &= \bdMV^\prime \rho^\Gamma(W_\infty)) + \bd \circ k^{-1}(\Ind_W^\Gamma(\clnDiracOp_{W_\infty})\\
     &= \rho^\Gamma(W) + \bd (\Ind_W^\Gamma(\clnDiracOp_W)).
  \end{align*}
  Finally, we apply \cref{thm:partitionedManifold} to $W_\infty$ partitioned by $X \times \{1\}$ and with $Z = W$ to obtain $\bdMV \IndLP{W}^\Gamma(\clnDiracOp_{W_\infty}) = \rho^\Gamma(g_X)$.
  This concludes the proof.
  \end{proof}

\begin{cor}[\cite{PS14Rho,XY14Positive}]\label{thm:DelocalizedAPS}
  Let $\iota \colon X \hookrightarrow W$ denote the inclusion map.
  \begin{equation*}
   \iota_\KPh \bd(\Ind_W^\Gamma(\clnDiracOp_W)) = \iota_\KPh \rho^\Gamma(g_X).
  \end{equation*}
\end{cor}
\begin{proof}
  By construction, $\rho^\Gamma(W)$ lies in the kernel of $\KTh_\KPh(\roeAlgLocZ(X)^\Gamma) \overset{\iota_\KPh}{\to} \KTh_\KPh(\roeAlgLocZ(W))$.
  Thus the corollary is an immediate consequence of \cref{thm:refinedDeloalizedAPS}.
\end{proof}

\subsection{Stability of higher secondary invariants on closed manifolds}
In this subsection, we discuss applications to higher secondary invariants for psc metrics on closed spin manifolds.
We begin with some preliminary definitions.
\begin{defi}[\cite{PS14Rho}]\label{defi:univStructGp}
  Let $\Gamma$ be a countable discrete group.
  Define the universal structure group for $\Gamma$ by
  \begin{equation*}
    \structureGp_{\KPh}^\Gamma := \colim_{X \subset \Efree \Gamma} \KTh_\KPh\left( \roeAlgLocZ(X)^\Gamma \right),
  \end{equation*}
  where the colimit ranges over $\Gamma$-invariant cocompact subsets of $\Efree \Gamma$, a universal space for free $\Gamma$-actions.
\end{defi}

If $X$ is a cocompact free $\Gamma$-space, then the $\KTh$-theory of the equivariant Roe algebra $\roeAlg(X)^\Gamma$ is canonically isomorphic to the $\KTh$-theory of the reduced group $\Cstar$-algebra $\CstarRed \Gamma$, see~\cite[Lemma 12.5.3]{higson-roe:analyticKHomology}.
Thus from the short exact sequence $0 \to \roeAlgLocZ(X)^\Gamma \to \roeAlgLoc(X)^\Gamma \to \roeAlg(X)^\Gamma$, we obtain a long exact sequence as follows:
\begin{equation}
  \dotsm \to \KTh_{\KPh+1}(\CstarRed \Gamma) \overset{\bd}{\to} \structureGp_\KPh^\Gamma \to \KTh_\KPh(\Bfree \Gamma) \to \KTh_\KPh(\CstarRed \Gamma) \to \dotsm.\label{eq:higsonRoeEqui}
\end{equation}

\begin{defi}
  Let $M$ be a closed $n$-dimensional spin manifold together with a continuous map $u \colon M \to \Bfree \Gamma$ (for instance, take $\Gamma = \pi_1(M)$ and $u$ classifying the universal covering).
  Let $g$ be a psc metric on $M$.
  Let $\hat{M} \to M$ be the $\Gamma$-covering classified by $u$ and denote the lift of $g$ to $\hat{M}$ by $\hat{g}$.
  
  The \emph{higher $\rho$-invariant} of $g$ is defined as follows:
  \begin{equation*}
    \rho^{u}(g) := \hat{u}_* \rho^{\Gamma}(\hat{g}) \in \structureGp_n^\Gamma,
  \end{equation*}
  where $\rho^{\Gamma}(\hat{g}) \in \KTh_n(\roeAlgLocZ(\hat{M})^\Gamma)$ is the equivariant $\rho$-class of $\hat{g}$.
   
  Similarly, we obtain a \emph{higher relative index} for two psc metrics $g_0, g_1$ on $M$,
\begin{equation*}
  \IndRel^u(g_0, g_1) := \IndRel^\Gamma(\hat{g}_0, \hat{g}_1) \in  \KTh_{n+1}(\CstarRed \Gamma) \iso \KTh_{n+1}(\roeAlg(\hat{M})^\Gamma).
\end{equation*}
\end{defi}
 If $\Gamma$ is torsion-free, then the Baum--Connes conjecture is equivalent to vanishing of $\structureGp_\ast^\Gamma$.
  As currently there is no known counter example to the Baum--Connes conjecture, we need to work with groups that have torsion in order to find non-zero higher $\rho$-invariants in $\structureGp_\ast^\Gamma$.

Given two closed spin manifolds with maps $u_i \colon M_i \to \Bfree \Gamma$, endowed with psc metrics $g_i$, $i=0,1$, we say that $(M_0, u_0, g_0)$ and $(M_1,u_1,g_1)$ are \emph{bordant} if there exists a compact spin manifold with reference map $v \colon W \to \Bfree \Gamma$ and boundary $\bd W = (-M_0) \sqcup M_1$ such that $W$ is endowed with a psc metric $h$ that is a product metric on a collar neighborhood of the boundary and agrees with $g_i$ on $M_i$.

Recall that if $g_0, g_1$ are concordant, then $\IndRel^{u}(g_0,g_1) = 0$.
It follows from \cref{thm:DelocalizedAPS} that $\bd(\IndRel^u(g_0, g_1)) = \rho^{u}(g_0) - \rho^{u}(g_1)$.
Moreover, if $(M_0, u_0, g_0)$ and $(M_1, u_1, g_1)$ are bordant, then $\rho^{u_1}(g_1) = \rho^{u_2}(g_2)$.  
Bordism classes of spin manifolds with positive scalar curvature metric as above form a group $\mathrm{Pos}_\ast^{\mathrm{spin}}(\Bfree \Gamma)$ which is part of Stolz' positive scalar curvature sequence.
In fact, in \cite{PS14Rho,XY14Positive} a map from the Stolz sequence to the Higson--Roe sequence \labelcref{eq:higsonRoeEqui} is constructed and the delocalized APS index theorem \cref{thm:DelocalizedAPS} is a central ingredient for that.

As an application of \cref{thm:distinguishStabilizeByHypereucl}, we prove the following results which say that if the higher $\rho$-invariant can distinguish two psc metrics, then it can still distinguish them after taking the product with certain aspherical manifolds.
\begin{prop} \label{prop:productWithNonPosInjective}
  Let $N$ be a closed aspherical spin manifold such that $\Lambda = \pi_1(N)$ has finite asymptotic dimension.
  Let $\Gamma$ be a countable discrete group.
  Then the map $\structureGp_\KPh^\Gamma \to \structureGp_{\KPh+q}^{\Gamma \times \Lambda}$ given by external product with the local index class  $\IndL^\Lambda(\clnDiracOp_{\tilde{N}}) \in \KTh_q(\roeAlgLoc(\tilde{N})^\Lambda)$ is split-injective.
\end{prop}
\corDistinguishAfterProductWithAspherical
 In particular, these assumptions imply that $(M_0 \times N, u_0 \times \id_N, g_0 \oplus g_N)$ and $(M_1 \times N, u_1 \times \id_N, g_1 \oplus g_N)$ are not bordant.
\begin{proof}[Proof of \cref{prop:productWithNonPosInjective}]
  Since $N$ is aspherical, we may choose $\Efree(\Gamma \times \Lambda) = \Efree \Gamma \times \tilde{N}$, where $\tilde{N}$ is the universal covering of $N$.
  Then 
  \begin{equation*}
  S_\KPh^{\Gamma \times \Lambda} = \colim_{X \subset \Efree \Gamma} \KTh_{\KPh}(\roeAlgLocZ(X~\times~\tilde{N})^{\Gamma \times \Lambda}).
  \end{equation*}
  Since $\pi_1(N)$ has finite asymptotic dimension, $\tilde{N}$ is stably hypereuclidean by \cite[Theorem 3.5]{dranishnikov:onHypereuclideanManifolds}.
  For each $X \subset \Efree \Gamma$ \cref{thm:distinguishStabilizeByHypereucl} gives a map $r_X \colon \KTh_{\KPh+n}(\roeAlgLocZ(X \times \tilde{N}))^{\Gamma \times \Lambda} \to \KTh_{\KPh}(\roeAlgLocZ(X)^\Gamma)$ left-inverse to taking the external product with $\IndL^\Lambda(\clnDiracOp_{\tilde{N}})$.
  Moreover, it follows readily from the proof of \cref{thm:distinguishStabilizeByHypereucl} that these maps are natural in $X$.
  This proves the proposition.
\end{proof}
\begin{rem}
  Stability results concerning the higher relative index $\IndRel^u(g_0, g_1)$  analogous to \cref{cor:distinguishAfterProductWithAspherical} can be obtained in a similar fashion for instance by applying \cref{thm:distinguishStabilizeByHypereucl} to the partial secondary invariant $\IndLP{[0,1] \times \hat{M}}^\Gamma(\clnDiracOp_{\R \times \hat{M}}^h)$~(see~\cref{defi:relIndex}).
\end{rem}

The manifold $N$ itself does not admit a psc metric because it is aspherical and its fundamental group has finite asymptotic dimension.
In fact, if we take $N$ to be a psc manifold, then the analogue of \cref{cor:distinguishAfterProductWithAspherical} is false:
\begin{lem}
  Let $M$ and $N$ be closed spin manifolds which both admit psc metrics individually.
  Then any two product metrics which have psc on $M \times N$ are concordant.
\end{lem}
Here we say that a Riemannian metric $h$ on $M \times N$ is a \emph{product metric} if it can be written as $h = g_M \oplus g_N$ for some metrics $g_M, g_N$ on $M$, respectively $N$.
\begin{proof}
  Let $g_M$ and $g_N$ be psc metrics on $M$, respectively $N$. Denote the product metric by $h = g_M \oplus g_N$.
  Suppose that $\tilde{h} = \tilde{g}_M \oplus \tilde{g}_N$ is another product metric which has psc on $M \times N$.
  We will now show that $h$ and $\tilde{h}$ are actually isotopic as psc metrics (it is a standard fact that isotopy implies concordance).
  By assumption, $\tilde{h}$ has psc, so either $\tilde{g}_M$ or $\tilde{g}_N$ has psc and assume w.l.o.g.~that it is $\tilde{g}_M$.
  By compactness we may find $\varepsilon > 0$ such that $\varepsilon \tilde{g}_M \oplus \left( t g_N + (1-t) \tilde{g}_N\right)$ has psc for all $t \in [0,1]$.
  By inserting appropriate rescalings of $\tilde{g}_M$, this implies that $\tilde{g}_M \oplus \tilde{g}_N$ and $\tilde{g}_M \oplus g_N$ are isotopic psc metrics.
  Applying the same argument again, now $g_N$ playing the role of $\tilde{g}_M$, shows that $\tilde{g}_M \oplus g_N$ is isotopic to $g_M \oplus g_N$.
\end{proof}
In particular:
\begin{prop}\label{prop:notAProductMetric}
  Let $M$ and $N$ be closed spin manifolds which both admit psc metrics individually.
  If $h_0, h_1$ are psc metrics on $M \times N$ such that $\IndRel^v(h_0, h_1) \neq 0$ for some $v \colon M \times N \to \Bfree \Gamma$, then at least one of $h_0$ and $h_1$ is not concordant to a product metric.
\end{prop}
\subsection{From closed manifolds to non-compact complete manifolds}\label{subsec:compactToNonCompact} In this subsection, we demonstrate how the theory we have developed so far can be applied to construct examples of complete upsc metrics on non-compact manifolds which are distinguished by certain partial secondary invariants.
As input for the following constructions we will always start with psc metrics on closed manifolds which can be distinguished by the higher $\rho$-invariant.
Such examples can be obtained from the methods of Weinberger--Yu~\cite{weinberger-yu:finitePartOfOperatorKtheoryForGroupsFinitelyEmbeddable} and Xie--Yu~\cite{xie-yu:HigherRhoInvariantsAndTheModuliSpaceOfPSC}.

We start with a corollary of the secondary partitioned manifold index theorem, \cref{thm:partitionedManifold}.
\thmCorOfPartitionedMfd
\begin{proof}
  Let $\hat{W} \to W$ be the $\Gamma$-covering of $W$ corresponding to the map $W \to \Bfree \Gamma$.
  Then $(\hat{W}, \hat{h}_i)$ is partitioned by $(\hat{M}, \hat{g}_i)$, where $\hat{M}$ is the restriction of $\hat{W}$ to $M$ and $\hat{h}_i$, $\hat{g}_i$ are the corresponding lifts of the Riemannian metrics. 
  \Cref{thm:partitionedManifold} implies that $\IndLP{\hat{W}_-}^\Gamma(\clnDiracOp_{\hat{W}}^{h_0}) \neq \IndLP{\hat{W}_-}^\Gamma(\clnDiracOp_{\hat{W}}^{h_1})$, from which the corollary follows due to \cref{thm:concordanceInvariance}.
\end{proof}

Of course, this applies in particular to $W = M \times \R$.
However, using our stability result about products with hypereuclidean manifolds (\cref{thm:distinguishStabilizeByHypereucl}), we can generalize the product situation to higher codimensions, as will be explained in the following.

Recall that a proper metric space $X$ endowed with a free and proper $\Gamma$-action is called \emph{flasque} if there exists a $\Gamma$-equivariant coarse map $s \colon X \to X$ such that
\begin{enumerate}[(i)]
  \item $s$ is coarsely equivalent to $\id_X$,
  \item for every compact subset $K \subseteq X$, there exists $l_0 \in \N$ such that $s^l(X) \cap K = \emptyset$ for all $l \geq l_0$,
  \item for all $R > 0$, there exists $S > 0$ such that $d_X(s^l(x), s^l(x^\prime)) < S$ for all $l \geq 0$ and $x, x^\prime \in X$ with $d_X(x,x^\prime) < R$.
\end{enumerate}
A standard Eilenberg swindle argument shows that $\KTh_\KPh(\roeAlg(X)^\Gamma)$ vanishes in all degrees if $X$ is flasque, see~\cite[Proposition 9.4]{roe:indexTheoryCoarseGeometryTopologyOfManifolds}, \cite[Lemma 6.4.2]{higson-roe:analyticKHomology}.
It follows directly from the definition that $X \times Y$ is flasque if $X$ is flasque and $Y$ is an arbitrary proper metric space.
\begin{defi}\label{defi:coarseNegligible}
  Let $X, Y$ be proper metric spaces, both of which are endowed with a free and proper isometric $\Gamma$-action.
  We say that a $\Gamma$-equivariant coarse map $f \colon X \to Y$ is \emph{coarsely negligible} if there exist $\Gamma$-equivariant maps $f^\prime \colon X \to X^\prime$, $f^{\prime \prime} \colon X^\prime \to Y$ such that $X^\prime$ is flasque and $f$ is coarsely equivalent to $f^{\prime \prime} \circ f^\prime$.
 
  We say a subset $Z \subseteq Y$ is coarsely negligible in $Y$ if the inclusion map $Z \hookrightarrow Y$ is coarsely negligible.\footnote{Our notion of a coarsely negligible subset is less general and more geometric than the concept of a \enquote{coarsely $A$-neglibible subset} from \cite[Definition 3.9]{HPS14Codimension}.}
\end{defi}
If $f$ is coarsely negligible, then it follows from functoriality of the Roe algebra that the map $f_\KPh \colon \KTh_\KPh(\roeAlg(X)^\Gamma) \to \KTh_\KPh(\roeAlg(Y)^\Gamma)$ is zero.
Moreover, if $f \colon Z \to Y$ is coarsely negligible, then so is $\id_X \times f \colon X \times Z \to X \times Y$ for any proper metric space~$X$.

\begin{ex}
  If $Z \subseteq Y$ is a compact subset of a non-compact complete Riemannian manifold $Y$, then $Z$ is coarsely negligible in $Y$. 
  To prove this, one uses that $Z$ is contained in a bounded neighborhood of some geodesic ray, compare~\cite[Proposition 3.10]{HPS14Codimension}.
\end{ex}
\begin{ex}
  If $Y$ is an arbitrary proper metric space, then $Y \times [0,\infty)$ is coarsely negligible in $Y \times \R$ (since $Y \times [0,\infty)$ is itself flasque).
\end{ex}  

\begin{lem}\label{thm:coarselyNegligibleImpliesInj}
  Let $Z \subseteq X$ be a $\Gamma$-invariant coarsely negligible subset.
  Then the map $\KTh_\KPh(\roeAlgLocZ(X)^\Gamma) \to \KTh_\KPh(\roeAlgLocP{Z}(X)^\Gamma)$ is injective.
\end{lem}
\begin{proof}
  Consider the boundary map $\bd \colon \KTh_{\KPh+1}(\roeAlg(Z \subset X)^\Gamma) \to \KTh_\KPh(\roeAlgLocZ(X)^\Gamma)$ associated to $0 \to \roeAlgLocZ(X)^\Gamma \to \roeAlgLocP{Z}(X)^\Gamma \to \roeAlg(Z \subset X)^\Gamma \to 0$.
  By naturality, $\bd$ factors through $\KTh_\KPh(\roeAlg(Z \subset X)^\Gamma) \to \KTh_\KPh(\roeAlg(X)^\Gamma)$.
  However, since $Z$ is coarsely negligible, the latter map is zero and hence $\bd$ is zero.
  By exactness, we conclude that $\KTh_\KPh(\roeAlgLocZ(X)^\Gamma) \to \KTh_\KPh(\roeAlgLocP{Z}(X)^\Gamma)$ is injective.
\end{proof}

\thmProductCoarselyNegligible

\begin{ex}
  The theorem applies to $Y = \R^q$ and $Z = [0,\infty) \times \R^{q-1}$.
\end{ex}
\begin{sloppypar}
\begin{proof}%[Proof of \cref{thm:productCoarselyNegligible}
  Let $X = \hat{M}$, the covering of $M$ classified by $u$.
  If $g_0 \oplus g_Y$ and $g_1 \oplus g_Y$ were concordant on $M \times Y$ relative to $M \times Z$, then the lifted metrics $\hat{g}_0 \oplus g_Y$, $\hat{g}_1 \oplus g_Y$ would be concordant on $X \times Y$ relative to $X \times Z$.
  Thus, due to \cref{thm:concordanceInvariance}, it suffices to show that $\IndLP{X \times Z}^\Gamma(\clnDiracOp_{X \times Y}^{\hat{g}_0 \oplus g_Y}) \neq \IndLP{X \times Z}^\Gamma(\clnDiracOp_{X \times Y}^{\hat{g}_1 \oplus g_Y})$.
  Indeed, since $Y$ is stably hypereuclidean, \cref{thm:distinguishStabilizeByHypereucl} (applied with $Z = \emptyset$) shows that $\rho^\Gamma(\hat{g}_0 \oplus g_Y) \neq \rho^\Gamma(\hat{g}_1 \oplus g_Y)$.
  Moreover, since $Z$ is coarsely negligible in $Y$, the $\Gamma$-invariant subset $X \times Z$ is coarsely negligible in $X \times Y$.
  By definition, $\rho^\Gamma(\hat{g}_i \oplus g_Y)$ maps to $\IndLP{X \times Z}^\Gamma(\clnDiracOp_{X \times Y}^{\hat{g}_i \oplus g_Y})$, $i=0,1$, under the map $\KTh_\KPh(\roeAlgLocZ(X \times Y)^\Gamma) \to \KTh_\KPh(\roeAlgLocP{X \times Z}(X \times Y)^\Gamma)$, which completes the proof since \cref{thm:coarselyNegligibleImpliesInj} states that this map is injective.
\end{proof}
\end{sloppypar}

\appendix
\section{Explicit descriptions in terms of projections and unitaries} \label{sec:comparison}
In this appendix, we focus on the complex case. 
In contrast to previous sections, we now reserve the notation $\KTh_j(A)$ for $j \in \Z/2\Z$ to denote \enquote{ordinary} complex $\KTh$-theory of an ungraded $\Cstar$-algebra $A$ defined in terms of projections and unitaries.
In particular, in the following we always use $\KThGr_\KPh(A)$ if we mean the $\KTh$-groups as discussed in \cref{sec:gradedKtheory}.

We describe the partial secondary local index classes as elements in $\KTh_0(\roeAlgLocP{Z}(X))$ and $\KTh_1(\roeAlgLocP{Z}(X))$.
This is done in essentially the same way as Xie--Yu~\cite{XY14Positive} define the local index class and the $\rho$-invariant.
We show that this agrees with the elements defined in \cref{subsec:partialIndex} up to a sign and a natural isomorphism $\KThGr_0(A \tensGr \cliffAlgC_n) \iso \KTh_n(A)$ for ungraded $\Cstar$-algebras $A$.

We drop the group action to simplify the notation, but it would not entail any additional technical difficulties to include it.

\subsection{\texorpdfstring{$\KTh$}{K}-theory of trivially graded \texorpdfstring{$\Cstar$}{C*}-algebras} \label{subsec:triviallyGraded}
We need an explicit isomorphism between (the complex version of) the picture of $\KTh$-theory explained in \cref{sec:gradedKtheory} and complex $\KTh$-theory for trivially graded algebras defined in terms of projections and unitaries.

\begin{prop}[{\cite[p.149f]{higson-guentner:GroupCstarAlgebrasAndKTheory}}] \label{prop:complexUngradedKTheory}
  For every trivially graded $\Cstar$-algebra $A$ there are natural isomorphisms $\Theta_{2k} \colon \KThGr_0(A \tensGr \cliffAlgC_{2k}) \to \KTh_0(A)$ and $\Theta_{2k+1} \colon \KThGr_0(A \tensGr \cliffAlgC_{2k+1}) \to \KTh_1(A)$ such that the following diagram of isomorphisms commutes,
  \begin{equation*}
    \begin{tikzcd}[row sep=small]
      \KThGr_0(A \tensGr \cliffAlgC_{2k+1}) \rar{\KThProd b} \dar{\Theta_{2k+1}} & \KThGr_0(\contZ(\R) \tens A \tensGr \cliffAlgC_{2k}) \dar{\Theta_{2k}} \\
      \KTh_1(A) \rar{\delta} & \KTh_0(\contZ(\R) \tens A),
    \end{tikzcd}
  \end{equation*}
  where $\delta$ is the standard suspension isomorphism in $\KTh$-theory and $\KThProd b$ is the Bott isomorphism from \cref{subsec:bott}.
\end{prop}
\begin{proof}
  Since by (formal) periodicity $\KThGr_0(A \tensGr \cliffAlgC_{2k}) = \KThGr_0(A)$ and $\KThGr_0(A \tensGr \cliffAlgC_{2k+1}) = \KThGr_0(A \tensGr \cliffAlgC_1)$, we can restrict ourselves to the case $k=0$.
  Let $\varphi \colon \contZGr \to A \tensGr \cptOps$ be given and consider the unitary,
\begin{equation*}
 U_\varphi := \varphi^+\left( \frac{\mathrm{x} - \iu}{\mathrm{x} + \iu} \right) \in \left(A \tensGr \cptOps \right)^+,
\end{equation*}
where $(A \tensGr \cptOps)^+$ denotes the unitization of $A \tensGr \cptOps$.
Observe that $U_\varphi$ is equal to the identity modulo $A \tensGr \cptOps$.
Since $A$ is trivially graded, the ungraded tensor product $A \tens \cptOps$ coincides with $A \tensGr \cptOps$ when we neglect the grading.
We will use the symbol $A \tens \cptOps$ if we wish to consider it as an ungraded algebra, and $A \tensGr \cptOps$ if we want to emphasize the grading.
Using this convention, the graded $\Cstar$-algebra $A \tensGr \cptOps$ can be identified with the matrix algebra $\Mat_2(A \tens \cptOps)$, where the grading automorphisms is conjugation with the multiplier 

\begin{equation*}\epsilon = \begin{pmatrix}
 1 & 0 
 \\ 0 & -1
 \end{pmatrix}.
 \end{equation*}
Then we have
\begin{equation}
    \epsilon U_\varphi \epsilon = U_\varphi^*. \label{eq:UphiAndGrading}
\end{equation}
In particular, $\epsilon U_\varphi$ is a self-adjoint unitary in $\Mat_2 \left( (A \tens \cptOps )^+\right)$ eqal to $\epsilon$ modulo $A \tens \cptOps$.
As a consequence, $\mathrm{P}_\varphi := \frac{1}{2} \left(1 + \epsilon U_\varphi \right)$ is a projection in $\Mat_2\left((A \tens \cptOps)^+\right)$ equal to $\mathrm{P}_\epsilon := \frac{1}{2} \left(1 + \epsilon\right)$ modulo $A \tens \cptOps$.
We define $\Theta_0([\varphi]) := [P_\varphi] - [P_\epsilon] \in \KTh_0(A \tens \cptOps) = \KTh_0(A)$.

Similarly, if we have a $\ast$-homomorphism $\varphi \colon \contZGr \to A \tensGr (\cliffAlgC_1 \tensGr \cptOps)$, we again form the unitary $U_\varphi = \varphi^+(\frac{\mathrm{x} - \iu}{\mathrm{x} + \iu}) \in (A \tensGr \cliffAlgC_1 \tensGr \cptOps)^+$.
Using the explicit description of the Clifford algebra $\cliffAlgC_1 \iso \C \oplus \C$, we may consider $U_{\varphi, 1} = \proj_1^+ (U_\varphi)$, a unitary in $(A \tens \cptOps)^+$.
We define $\Theta_1([\varphi]) := [U_{\varphi, 1}] \in \KTh_1(A \tens \cptOps) = \KTh_1(A)$.

The maps $\Theta_i$ defined above are well-defined because if $\varphi$ and $\psi$ are homotopic, then so are $U_\varphi$ and $U_\psi$.
Since the additive structure on $\KThGr_0$ is essentially defined by block sum inside $A \tensGr \cptOps$, it can be verified that the maps $\Theta_n$ are additive.
Moreover, it follows from \cite[Lemma 1.4]{higson-guentner:GroupCstarAlgebrasAndKTheory} that these maps are isomorphisms.

Finally, due to naturality, it suffices to consider the case $A = \contZ(\R)$ to show that the diagram in the proposition commutes.
The group $\KTh_0(\contZ(\R) \tensGr \cliffAlgC_1)$ is generated by the Bott element $b$ which is represented by the $\ast$-homomorphism $\beta \colon \contZGr \to \contZ(\R, \cliffAlgC_1) \iso \contZ(\R) \oplus \contZ(\R)$, $f \mapsto (x \mapsto (f(x), f(-x)))$.
Thus $\Theta_1(b)$ is represented by the unitary $\frac{\mathrm{x} - \iu}{\mathrm{x} + \iu}$ in $\contZ(\R)^+$ which has winding number $+1$ and thus represents the standard generator of $\KTh_1(\contZ(\R))$.
Thus $\delta(\Theta_1(b)) \in \KTh_0(\contZ(\R^2))$ is represented by the Bott projection.
It can be checked that the element $b \KThProd b \in \KThGr_0(\contZ(\R^2))$ is represented by the $\ast$-homomorphism 
\begin{equation*}
\varphi \colon \contZGr \to \contZ(\R^2, \Mat_2(\C)), \quad f \mapsto f\left( \begin{pmatrix} 
0 & \mathrm{x} + \iu \mathrm{y} 
\\ \mathrm{x} - \iu \mathrm{y} & 0                                                                                                                                                                                                                                                                                                                                          \end{pmatrix}\right).
\end{equation*}
A direct computation shows that $P_\varphi$ is also the Bott projection, hence the diagram commutes.
\end{proof}

\subsection{Reduced spinor bundles}
So far, we have worked with the $\cliffAlgC_n$-linear (or \emph{$n$-multi-graded} in the terminology of~\cite{higson-roe:analyticKHomology}) spinor bundle and Dirac operator.
Here we review the equivalent viewpoint using irreducible Clifford modules.
If $n$ is even, there is up to isomorphism only one irreducible Clifford module which we denote by $\spinorBdl(n)$.
It automatically carries a grading $\spinorBdl(n) = \spinorBdl^{(0)}(n) \oplus \spinorBdl^{(1)}(n)$.
If $n$ is odd, there are two irreducible Clifford modules $\spinorBdl_+(n)$ and $\spinorBdl_-(n)$ which are ungraded.
Let $X$ be a Riemannian spin manifold, then the reduced spinor bundle is the associated bundle $\spinorBdl(X) = \principalBdl_{\Spin}(X) \times_{\Spin(n)} \spinorBdl(n)$.
In the odd-dimensional case the representations of $\Spin(n)$ coming from the two irreducible Clifford modules are isomorphic, so it does not matter which we choose.
We may realize $\spinorBdl(n)$ concretely as a left ideal inside $\cliffAlgC_n$ so that $\spinorBdl(X)$ is a sub-bundle of $\clnSpinorBdl(X)$ which is $\cliffAlgC(\tangentBdl^* X)$-invariant.
In particular, the $\cliffAlgC_n$-linear Dirac operator $\clnDiracOp$ restricts to the spinor Dirac operator $\diracOp$ on $\spinorBdl(X)$.
Let $\clnXmodule := \Lp^2(X, \clnSpinorBdl(X))$ and $\Xmodule := \Lp^2(X, \spinorBdl(X))$.
The first is always a graded Hilbert $\cliffAlgC_n$-module whereas the latter is a Hilbert space, furnished with a grading iff $n$ is even.

Suppose $n = 2k > 0$. Then $\End_\C(\spinorBdl(n)) = \cliffAlgC_n$ and there is a one to one correspondence between (possibly unbounded) $\cliffAlgC_n$-linear operators on $\clnXmodule$ and $\C$-linear operators on $\Xmodule$.
Indeed, every $\cliffAlgC_n$-linear operator on $\clnXmodule$ keeps $\Xmodule$ invariant and is uniquely determined by its restiction to $\Xmodule$.
On the level of Roe algebras this yields a canonical isomorphism (of graded $\Cstar$-algebras),
\begin{equation}
  \roeAlg(X, \mathfrak{H}; \cliffAlgC_{2k}) = \roeAlg(X, \Xmodule). \label{eq:reducedSpinorRoeAlgebraEven}
\end{equation}

In the odd-dimensional case $n = 2k+1$, we have $\cliffAlgC_n = \End_\C(\spinorBdl_+(n)) \oplus \End_\C(\spinorBdl_-(n))$.
A similar argument as above yields,
\begin{equation}
  \roeAlg(X, \clnXmodule; \cliffAlgC_{2k+1}) = \roeAlg(X, \Xmodule) \oplus \roeAlg(X, \Xmodule). \label{eq:reducedSpinorRoeAlgebraOdd}
\end{equation}

The identifications \labelcref{eq:reducedSpinorRoeAlgebraEven,eq:reducedSpinorRoeAlgebraOdd} hold analogously for the structure algebra $\structureAlg$ and the all localization algebras $\structureAlgLoc$ and $\roeAlgLocP{Z}$.

\subsection{Local index classes in terms of projections and unitaries}

A \emph{normalizing function} is a continuous odd non-decreasing function $\chi \colon \R \to [-1,1]$ such that $\lim_{x \to \pm \infty} \chi(x) = \pm 1$.
Let $L_{\diracOp} \colon  [1, \infty) \to \structureAlg(X, \Xmodule)$, $L_{\diracOp}(t) = \chi\left(\frac{1}{t}\diracOp \right)$, where $\chi$ is a normalizing function.
Let $Z \subseteq X$ be a closed subset (possibly $Z = \emptyset$ or $Z = X$) such that the scalar curvature function on $X$ is uniformly positive \emph{outside of $Z$}.
Then by \cref{lem:roesLemmaLocalizedIndex}, we have $L_{\diracOp}^2 - 1 \in \roeAlgLocP{Z}(X, \Xmodule)$ provided that we have chosen $\chi$ such that $\chi^2 = 1$ outside $(-\varepsilon, \varepsilon)$. 

 If $n > 0$ is even, then $\Xmodule$ is graded and $L_{\diracOp}(t)$ is an odd operator for all $t$, that is, with respect to the grading $\Xmodule = \Xmodule^{(0)} \oplus \Xmodule^{(1)}$, we have
\begin{equation*}
 L_\diracOp(t) = \begin{pmatrix}
            0 & L_{\diracOp}^-(t) \\ L_{\diracOp}^+(t) & 0
          \end{pmatrix}.
\end{equation*}
Let $v \colon \Xmodule^{(0)} \to \Xmodule^{(1)}$ be a unitary which intertwines the $\contZ(X)$-representations (for instance, take $v$ to be Clifford multiplication with a measurable unit co-vector field).
Then $v^* L_{\diracOp}^+ \in \structureAlgLoc(X, \Xmodule^{(0)})$ is a unitary modulo $\roeAlgLocP{Z}(X, \Xmodule^{(0)})$ and hence defines a class $[v^* L_{\diracOp}^+] \in \KTh_1(\structureAlgLoc(X)/\roeAlgLocP{Z}(X))$.
We let
\begin{equation*}
  \IndLTildeP{Z}(\diracOp) := \bd [v^* L_{\diracOp}^+] \in \KTh_0(\roeAlgLocP{Z}(X)).
\end{equation*}

If $n$ is odd, then $\frac{1}{2}(1 + L_\diracOp)$ is a projection modulo $\roeAlgLocP{Z}(X, \Xmodule)$ and we set
\begin{equation*}
  \IndLTildeP{Z}(\diracOp) := \bd \left[  \frac{1}{2}(1 + L_\diracOp) \right] \in \KTh_1(\roeAlgLocP{Z}(X)).
\end{equation*}

\begin{rem}
  For $Z = X$ (respectively $Z = \emptyset$), the element $\IndLTildeP{Z}(\diracOp)$ agrees with the local index class (respectively the $\rho$-invariant) defined in \cite{XY14Positive}.
  This can be proved using that the map $\structureAlgLoc(X) / \roeAlgLoc(X) \to \structureAlg(X) / \roeAlg(X)$ (respectively $\structureAlgLoc(X) / \roeAlgLocZ(X) \to \structureAlg(X)$) defined by evaluation at $1 \in [1, \infty)$ is a $\KTh$-theory isomorphism, see~\cite{roe-qiao:onTheLocalizationAlgebraOfYu}.
\end{rem}

\begin{thm}
  The isomorphism $\Theta_n$ from \cref{prop:complexUngradedKTheory},
  \begin{equation*}
    \Theta_n \colon \KThGr_0(\roeAlgLocP{Z}(X; \cliffAlgC_n)) = \KThGr_0(\roeAlgLocP{Z}(X) \tensGr \cliffAlgC_n) \to \KTh_n(\roeAlgLocP{Z}(X)),
  \end{equation*}
  takes $\IndLP{Z}(\clnDiracOp)$ as defined in \cref{subsec:partialIndex} to $-\IndLTildeP{Z}(\diracOp)$.
\end{thm}
\begin{proof}
  Let $\varepsilon > 0$ such that $f(\clnDiracOp)$ and $f(\diracOp)$ lie in $\roeAlg(Z \subset X)$ for $f \in \contZ((-\varepsilon, \varepsilon))$ and choose a homotopy inverse $\psi \colon \contZGr \to \contZGr(-\varepsilon, \varepsilon)$ to the inclusion $\contZGr(-\varepsilon, \varepsilon) \hookrightarrow \contZGr$.
  Then $\IndLP{Z}(\clnDiracOp) = [\varphi_\clnDiracOp \circ \psi]$.
  We can choose $\psi$ in such a way that it extends to a homotopy equivalence $\bar{\psi} \colon \cont([-\infty, \infty]) \to \cont([-\varepsilon, \varepsilon])$, where we identify $\cont([-\varepsilon,\varepsilon])$ with the subspace of $\cont(\Rext)$ consisting of those functions which are constant on $[-\infty, -\varepsilon]$ as well as on $[\varepsilon, \infty]$.
  Let $\tilde{\chi}$ be the normalizing function $\tilde{\chi}(x) := \frac{x}{\sqrt{1 + x^2}}$.
  We can assume that $\chi := \bar{\psi}(\tilde{\chi})$ is still a normalizing function which by construction satisfies $\chi^2 = 1$ outside $(-\varepsilon, \varepsilon)$.
  Let $\tilde{\tau}(x) := \frac{x - \iu}{\sqrt{1 + x^2}}$, then $\tilde{\tau}(x) = \tilde{\chi}(x) - \frac{\iu}{\sqrt{1 + x^2}}$, so $\tilde{\chi} \equiv \tilde{\tau}$ modulo $\contZ(\R)$.
  Set $\tau := \bar{\psi}(\tilde{\tau})$, then $\chi \equiv \tau$ modulo $\contZ((-\varepsilon, \varepsilon))$.  
  The unitary $U := U_{\varphi_\clnDiracOp \circ \psi}$ from \cref{prop:complexUngradedKTheory}\footnote{Note that we have enough room to carry out the construction of $\Theta_n$ from the proof of \cref{prop:complexUngradedKTheory} inside $\roeAlgLocP{Z}(X, \clnXmodule; \cliffAlgC_n)$ (without taking the tensor product with the compact operators) and we shall do so here.} is given by $\varphi_\clnDiracOp^+(\tau^2)$ since $\tilde{\tau}^2 = \frac{\mathrm{x}-\iu}{\mathrm{x}+\iu}$.
  In particular, $U = T^2$, where $T(t) := \tau(\frac{1}{t} \clnDiracOp)$ is a unitary in $\structureAlgLoc(X, \clnXmodule; \cliffAlg_n)$.
  Furthermore, $T(t) - \chi(\frac{1}{t} \clnDiracOp) = (\tau - \chi)(\frac{1}{t} \clnDiracOp)$, so $T - L_\clnDiracOp \in \roeAlgLocP{Z}(X, \clnXmodule; \cliffAlgC_n)$, where $L_\clnDiracOp(t) = \chi(\frac{1}{t} \clnDiracOp)$.
  
  Suppose that $n > 0$ is even.
  By \labelcref{eq:reducedSpinorRoeAlgebraEven}, we have $\roeAlgLocP{Z}(X, \clnXmodule; \cliffAlgC_n) = \roeAlgLocP{Z}(X, \Xmodule)$ and $L_\clnDiracOp = L_\diracOp \in \structureAlgLoc(X, \clnXmodule; \cliffAlg_n) = \structureAlgLoc(X, \Xmodule)$.
  Let $V := \left( \begin{smallmatrix} 0 & v^* \\ v & 0 \end{smallmatrix} \right)$ and $W := V T \in \structureAlgLoc(X, \Xmodule)$.
  Then $W$ is a unitary, which is equal to $V L_{\diracOp} = \left( \begin{smallmatrix} v^* L_\diracOp^+ & 0 \\ 0 & v L_{\diracOp}^- \end{smallmatrix} \right)$ modulo $\roeAlgLocP{Z}(X, \Xmodule)$.
  By the definition of the boundary map in $\KTh$-theory (\cite[Chapter 8.1]{wegge-olsen:KTheory}), we have
\begin{equation}
    \IndLTildeP{Z}(\diracOp) = \bd [v^* L_{\diracOp}^+] = \left[ W \begin{pmatrix} 1 & 0 \\ 0 & 0 \end{pmatrix} W^* \right] - \left[ \begin{pmatrix} 1 & 0 \\ 0 & 0 \end{pmatrix} \right] \in \KTh_0( \roeAlgLocP{Z}(X)). \label{eq:boundaryFormula}
\end{equation}
Let $\epsilon = \left( \begin{smallmatrix} 1 & 0 \\ 0 & -1 \end{smallmatrix} \right)$ be the grading operator on $\Xmodule = \Xmodule^{(0)} \oplus \Xmodule^{(1)}$ and $P_\epsilon = \frac{1}{2}(1 + \epsilon) = \left( \begin{smallmatrix} 1 & 0 \\ 0 & 0 \end{smallmatrix} \right)$.
Since the complex conjugate of $\tilde{\tau}(x)$ is $- \tilde{\tau}(-x)$, we have $\epsilon T = - T^* \epsilon$.
A direct computation using this fact together with $U^2 = T$ shows that $\frac{1}{2}(1 + \epsilon U) = 1 - \epsilon T P_\epsilon T^* \epsilon$.
Then we have
\begin{align*}
 \Theta_n([\varphi_\clnDiracOp \circ \psi]) &=  \left[ \frac{1}{2} \left(1 + \epsilon U \right) \right] - [P_\epsilon] \\ 
 &= \left[1 - \epsilon T P_\epsilon T^* \epsilon \right] - [P_\epsilon] \\
 &= [P_\epsilon] - [\epsilon T P_\epsilon T^* \epsilon] \\
 &= [P_\epsilon] - [W P_\epsilon W^*] \qquad \text{(by conjugation with $V \epsilon$)}  \\
 &= - \bd [v^* L_{\diracOp}^+] = - \IndLTildeP{Z}(\diracOp) \qquad \text{(by \labelcref{eq:boundaryFormula}).}
\end{align*}

If $n$ is odd, then by \labelcref{eq:reducedSpinorRoeAlgebraOdd}, we have $\roeAlgLocP{Z}(X, \clnXmodule; \cliffAlgC_n) = \roeAlgLocP{Z}(X, \Xmodule) \oplus \roeAlgLocP{Z}(X, \Xmodule)$ and $\proj_1(L_\clnDiracOp) = L_\diracOp$.
Then
\begin{equation*}
  \IndLTildeP{Z}(\diracOp) = \bd \left[ \frac{1}{2}\left(1 + L_\diracOp \right) \right] = \left[ \eu^{-2 \pi \iu \frac{1}{2} \left(1 + \proj_1(L_\clnDiracOp) \right)} \right] = \left[ \proj_1^+ \varphi_{\clnDiracOp}^+ \left( \eu^{-\pi \iu (1 + \chi)} \right) \right].
\end{equation*}
Here we have used the explicit description of the boundary map in terms of the exponential function (see~\cite[Exercise 9.E]{wegge-olsen:KTheory}).
The unitary $\eu^{-\pi \iu (1 + \chi)} = - \eu^{-\pi \iu \bar{\psi}(\tilde{\chi})}$ in $\contZ((-\varepsilon, \varepsilon))^+ \iso \cont(S^1)$ has winding number $-1$, whereas $\bar{\psi}\left(\frac{\mathrm{x} - \iu}{\mathrm{x} + \iu} \right)$ has winding number $+1$.
Consequently,
\begin{equation*}
  - \IndLTildeP{Z}(\diracOp) = \left[ \proj_1^+ \varphi_{\clnDiracOp}^+ \left( \bar{\psi}\left( \frac{\mathrm{x} - \iu}{\mathrm{x} + \iu} \right) \right) \right] = \left[\proj_1^+(U) \right] = \Theta_n \left( [\varphi_\clnDiracOp \circ \psi] \right). \qedhere
  \end{equation*}
\end{proof}

\printbibliography
\listoffixmes

\end{document}